%% file: main.tex
\pdfoutput = 1

\documentclass[onefignum,onetabnum]{siamart220329}

\input{siamshared}

\begin{document}

\maketitle
    
\begin{abstract}
In this paper, we study min-max optimization problems on Riemannian manifolds. We introduce a Riemannian Hamiltonian function, minimization of which serves as a proxy for solving the original min-max problems.  Under the Riemannian Polyak--\L{}ojasiewicz  condition on the Hamiltonian function, its minimizer corresponds to the desired min-max saddle point.  We also provide cases where this condition is satisfied. For geodesic-bilinear optimization in particular, solving the proxy problem leads to the correct search direction towards global optimality, which becomes challenging with the min-max formulation. To minimize the Hamiltonian function, we propose Riemannian Hamiltonian methods (RHM) and present their convergence analyses. We extend RHM to include consensus regularization and to the stochastic setting. We illustrate the efficacy of the proposed RHM in applications such as subspace robust Wasserstein distance, robust training of neural networks, and generative adversarial networks. 
\end{abstract}



\section{Introduction}


In this paper, we consider the Riemannian manifold constrained min-max problem
\begin{equation}
    \min_{x \in \M_x} \max_{y \in \M_y} f(x, y), \label{main_riem_minmax}
\end{equation}
where $\M_x, \M_y$ are complete Riemannian manifolds and $f : \M_x \times \M_y \xrightarrow{} \sR$ is a jointly smooth real-valued function. The aim is to find a global saddle point $(x^*, y^*)$ that satisfies for all $(x,y) \in \M_x\times \M_y$,
\begin{equation}
    f(x^*, y) \leq f(x^*, y^*) \leq f(x, y^*). \label{global_saddle_pt}
\end{equation}
Examples of Riemannian manifolds of interest include the sphere manifold, the Stiefel manifold, the manifold of orthogonal matrices, the manifold of doubly stochastic matrices, and the symmetric positive definite manifold, to name a few \cite{absil2009optimization,boumal2020introduction,shi2021coupling,bhatia2009positive}.

When both $\M_x, \M_y$ are the Euclidean space, problem (\ref{main_riem_minmax}) reduces to the classical min-max problem, which has been widely studied for applications including adversarial training \cite{madry2017towards}, robust learning \cite{el1997robust}, non-linear feature learning~\cite{rakotomamonjy08a,aflalo11a,jawanpuria11a,jawanpuria15b}, generative adversarial networks \cite{goodfellow2014generative,arjovsky2017wasserstein,schafer2019competitive}, constrained optimization \cite{bertsekas2014constrained}, multi-task learning~\cite{jawanpuria12a,jawanpuria15a}, and fair statistical inference \cite{madras2018learning}, among others. When $f$ is convex in $x$ and concave in $y$ (convex-concave), the existence of a global saddle point is guaranteed by the well-established minimax theorem \cite{neumann1928theorie,sion1958general}. Algorithms converging to such saddle points include the optimistic gradient descent ascent (OGDA) algorithm \cite{popov1980modification} and the extra-gradient algorithm (EG) \cite{extragradient1976}, which have been analyzed in \cite{nemirovski2004prox,monteiro2010complexity,monteiro2011complexity,mokhtari2020convergence}. For the general nonconvex-nonconcave setting, however, the saddle point, be it {local} or global, may not exist \cite{jin2020local}, and it remains challenging to establish convergence for both OGDA and EG.

On Riemannian manifolds, there exist cases where many nonconvex (or nonconcave) functions turn out to be geodesic convex (or concave), a generalized notion of convexity on Riemannian manifolds \cite{sra2015conic}. This ensures the existence of a global saddle point on manifolds under the generalized min-max theorem \cite{stacho1980minimax,zhang2022minimax}. 
Furthermore, there is a growing interest in the Riemannian min-max problem \eqref{main_riem_minmax} with applications such as low-rank tensor learning~\cite{jawanpuria18a,jawanpuria18b}, orthonormal generative adversarial networks \cite{muller2019orthogonal,brock2018large}, subspace robust Wasserstein distances \cite{paty19,lin2020projection,huang2021riemannian,jawanpuria21a}, and adversarial neural network training \cite{huang2020gradient}.
It is, therefore, motivating to study the min-max problem on manifolds.

Nevertheless, existing works that systematically study the Riemannian min-max problem are sparse. In \cite{huang2020gradient}, a Riemannian gradient descent ascent (RGDA) method has been proposed, yet the analysis is restricted to $\M_y$ being a convex subset of the Euclidean space and $f(x,y)$ being strongly concave in $y$. A recent paper \cite{zhang2022minimax} has formally characterized the optimality conditions of the Riemannian min-max problem for geodesic convex geodesic concave functions. A Riemannian corrected extra-gradient (RCEG) algorithm has been proposed and analyzed. 
A follow-up work \cite{jordan2022first} completes the analysis of RGDA and RCEG under geodesic (strongly) convex (strongly) concave settings.

\subsection*{Contributions}
In this paper, we propose a class of methods for solving the min-max problem \eqref{main_riem_minmax} on Riemannian manifolds, which we call Riemannian Hamiltonian methods (RHM). The idea is to minimize the squared norm of the Riemannian gradient of \eqref{main_riem_minmax}, known as the Riemannian Hamiltonian. Minimizing the Hamiltonian function serves as a good proxy for solving problem (\ref{main_riem_minmax}). Under the Riemannian Polyak--\L{}ojasiewicz (PL) condition  \cite{zhang2016riemannian} on the Hamiltonian function, its minimizer recovers the desired saddle point. {A key motivation to consider the proxy problem instead of the original min-max problem is for geodesic-bilinear problems, where solving the proxy problem leads to the correct direction towards global optimality while existing methods either cycle or converge extremely slowly (discussed in Section \ref{motivating_example_sect}).}
In addition, the Hamiltonian gradient methods have been considered for solving min-max problems in the Euclidean space, which show great promise in accelerating and stabilizing the convergence to saddle points \cite{abernethy2019last,balduzzi2018mechanics,mescheder2017numerics,loizou2020stochastic}. 
This paper generalizes many of those analysis to Riemannian manifolds. 

It should be emphasized that the proposed {generalization} to manifolds is nontrivial as the analysis for the Euclidean counterparts, such as in \cite{abernethy2019last}, rely heavily on the matrix properties of the Jacobian. Generalization to Riemannian manifolds require adherence to Riemannian operations independent of the matrix structure. Another challenge is to deal with the varying inner product (Riemannian metric) structure on manifolds. We handle the above by devising novel proof strategies and proposing a metric-aware Riemannian Hamiltonian function that respects the manifold geometry. 


In particular, we show global linear convergence of any Riemannian solver to saddle points of problem \eqref{main_riem_minmax} as long as the Riemannian Hamiltonian of $f$ satisfies the Riemannian PL condition \cite{zhang2016riemannian}. We show this occurs when $f$ is geodesic strongly convex geodesic strongly concave, and also for some nonconvex functions with sufficient geodesic linearity. We additionally extend the proposed RHM to incorporate a consensus regularization and to the stochastic setting, and prove their convergence. Existing Riemannian algorithms for solving (\ref{main_riem_minmax}) such as \cite{zhang2022minimax} make use of the exponential map to update the iterates on the manifolds. In this work, we discuss convergence results with exponential as well as general retraction maps on manifolds. 

We empirically show the convergence of our proposed RHM algorithms for different min-max functions and compare them with existing baselines. We further demonstrate the usefulness of RHM algorithms in various applications such as learning subspace robust Wasserstein distance, robust training of neural networks and training of generative adversarial networks.


\subsection*{Organizations} 
The rest of the paper is organized as follows. Section \ref{preliminary_sect} reviews the preliminary knowledge on Riemannian geometry and Riemannian optimization as well as introduces various functions classes on Riemannian manifolds. We also briefly discuss the existing literature on mix-max optimization in the Euclidean space and on Riemannian manifolds. In Section \ref{sect_hamiltonian}, we propose the Riemannian Hamiltonian function and RHM algorithms, as well as analyze their convergence under the Riemannian PL condition. We provide three cases when such condition is satisfied. Section \ref{RHM_con_sect} introduces and analyzes the Riemannian Hamiltonian consensus method. Sections \ref{stochastic_hm_sect} and \ref{retraction_sect} extend the proposed methods to stochastic settings and to the case of retraction. In Section \ref{application_sect}, we empirically compare our algorithms with different baselines on various applications. 
Section~\ref{sec:conclusion} concludes the paper.

\section{Preliminaries}
\label{preliminary_sect}

In this section, we give a brief overview of Riemannian geometry and relevant ingredients required for Riemannian optimization. For a more complete treatment of the topic, see \cite{absil2009optimization,boumal2020introduction}. We also briefly discuss some of the existing works on min-max optimization.

\subsection{Riemannian geometry and optimization}

\subsubsection*{Basic Riemannian geometry}
Riemannian manifold $\M$ is a manifold with a Riemannian metric, which is a smooth, symmetric positive definite function $g: T_p\M \times T_p\M \xrightarrow{} \sR$ on every tangent space $T_p\M$, with $p \in \M$. It is usually written as an inner product $\langle \cdot, \cdot \rangle_p$. The metric structure induces a norm for any tangent vector $\xi \in T_p\M$, which is $\| \xi \|_p := \sqrt{\langle \xi, \xi\rangle_p}$. For a linear operator on the tangent space $H: T_p\M \xrightarrow{} T_p\M$, its operator norm is defined as $\| H \|_p := \max_{\xi \in T_x\M : \| \xi\|_p = 1} \| H[\xi]\|_p$.

A geodesic on the manifold $\gamma: [0,1] \xrightarrow{} \M$ is the locally shortest curve with zero acceleration. The exponential map at $p$, ${\rm Exp}_p: T_p\M \xrightarrow{} \M$ is defined as the end point of a geodesic along the initial velocity. That is, ${\rm Exp}_p(\xi) = \gamma(1)$ where $\gamma'(0) = \xi$, $\gamma(0) = p$ for any $\xi \in T_p\M$. Riemannian distance is computed as $d(p, q) = \int_0^1 \|\gamma'(t) \|_{\gamma(t)} dt$ where $\gamma(t)$ is the distance  minimizing geodesic connecting $p,q \in \M$. 
In a totally normal neighbourhood $\Omega$ where there exists a unique geodesic between any $p, q \in \Omega$, the exponential map has a well-defined inverse ${\rm Exp}_p^{-1}: \M \xrightarrow{} T_p\M$ and the Riemannian distance can be written as $d(p,q) = \| {\rm Exp}_p^{-1}(q) \|_p = \| {\rm Exp}_{q}^{-1}(p)\|_{q}$. Parallel transport $\Gamma_{p}^q : T_p\M \xrightarrow{} T_q\M$ transports tangent vector along the geodesic while being isometric, i.e., $\langle \xi , \zeta \rangle_p  = \langle \Gamma_p^q \xi, \Gamma_p^q \zeta \rangle_q$ for any $\xi, \zeta \in T_p\M$.

\subsubsection*{Riemannian product manifolds} 
The product of Riemannian manifolds $\M = \M_x \times \M_y$ is a Riemannian manifold with the Riemannian metric defined as, for any $p = (x,y) \in \M$, and $(u,u'), (v,v') \in T_p\M$, $\langle (u, u'), (v, v') \rangle_p = \langle u, v \rangle^{\M_x}_x + \langle u', v' \rangle^{\M_y}_y$, where $\langle \cdot, \cdot \rangle^{\M_x}$, $\langle \cdot ,\cdot\rangle^{\M_y}$ are Riemannian metrics on $\M_x, \M_y$ respectively. From the metric, one can derive the geodesic, the exponential map, parallel transport, Riemannian distance, which also admit a product structure. See more details in \cite{boumal2020introduction}.

\subsubsection*{Riemannian optimization ingredients}
Riemannian optimization treats the constrained problem as an unconstrained problem on manifold by generalizing the notions of gradient and Hessian. For a differentiable function $h:\M \xrightarrow{} \sR$, the Riemannian gradient at $p$, $\grad h(p)$ is a tangent vector that satisfies $\langle \grad h(p), \xi \rangle_p = \D h(p)[\xi]$ for any $\xi \in T_p\M$ where $\D h(p)[\xi]$ is the directional derivative of $h$ along $\xi$. The Riemannian Hessian of $h$, $\hess h(p): T_p\M \xrightarrow{} T_p\M$ is a symmetric linear operator, defined as the covariant derivative of the Riemannian gradient. For a bi-function $f: \M_x \times \M_y \xrightarrow{} \sR$, we can similarly define Riemannian partial gradient $\grad_x f(x,y), \grad_y f(x,y)$ as Riemannian gradient for $x, y$, holding the other variable constant. The Riemannian cross derivative $\grad_{xy}^2 f(x,y): T_x\M_x \xrightarrow{} T_y \M_y$ is defined as $\grad_{xy}^2 f(x,y)[u]  := \D_x \grad_y f(x,y)[u]$ and similarly for $\grad_{yx}^2 f(x,y)$.

\subsubsection*{Riemannian geodesic convex optimization}
Geodesic convexity \cite{vishnoi2018geodesic,boumal2020introduction} generalizes the notion of convexity to Riemannian manifold. A \textit{geodesic convex set} $\Omega \subseteq \M$ requires for any two points in the set, there exist a geodesic (on $\M$) connecting them that lies entirely in the set. From this definition, any connected, complete Riemannian manifold is geodesic convex itself. A function $h: \Omega \xrightarrow{} \sR$ is \textit{geodesic convex} if for any $p, q \in \Omega$, it satisfies that $h(\gamma(t)) \leq (1-t) h(p) + t h(q)$ for $t \in [0,1]$ and $\gamma$ is a geodesic connecting $p,q$. A function is \textit{geodesic linear} if it is both geodesic convex and geodesic concave. A twice differentiable function $h$ is \textit{geodesic} $\mu$-\textit{strongly convex} if $\frac{ d^2h(\gamma(t))}{dt^2} \geq \mu$. We call a function $h(p)$ g-(strongly)-convex if it is geodesic (strongly) convex. Similarly, we call a function $f(x,y)$ g-(strongly)-convex-concave if it is geodesic (strongly) convex in $x$ and geodesic (strongly) concave in $y$.

Next, we define the spectrum of a linear operator on the tangent space, which is used to analyze the Riemannian Hessian as well as the Riemannian cross derivatives in the subsequent sections.

\begin{definition}[Spectrum of a linear operator]
Consider a linear operator $T : V \xrightarrow{} W$ where $V, W$ are two inner product spaces. 
If $V = W$, and $T$ is symmetric, i.e., $T = T^*$, where $T^*$ is the adjoint operator of $T$, then we say $(\lambda, v)$ is an eigenpair of $T$ if $T[v] = \lambda v$. In general, when $V \neq W$, the singular value $\sigma$ of $T$ is the square root of the eigenvalues of $T^* \circ T$. 
\end{definition}

We use $\lambda_{\min}/\lambda_{\max}$ and $\sigma_{\min}/\sigma_{\max}$ to represent the smallest/largest eigenvalues and singular values, respectively. We also use $\lambda_{\rm |min|}$ to denote the minimum eigenvalue in magnitude. Below, we introduce several function classes on manifolds, generalizing the Lipschitz continuity as well as the Polyak--\L{}ojasiewicz condition from the Euclidean space \cite{royden1988real,polyak1963gradient} 

\begin{definition}[Lipschitz continuity \cite{boumal2020introduction}]
Let $L_0, L_1, L_2 > 0$. \hfill
\begin{enumerate}[label={(\arabic*).}]
    \item A real-valued function $h: \M \xrightarrow{} \sR$ is $L_0$-Lipschitz continuous if for all $p \in \M$, $\| \grad h(p) \|_p \leq L_0$.
    \item  A vector field $V \in \mathfrak{X}(\M)$ is $L_1$-Lipschitz continuous if for all $p \in \M$ and $s \in T_p\M$ such that $q = {\rm Exp}_p(s) \in \Omega$, a totally normal neighbourhood of $p$, it satisfies $\| \Gamma_q^p V(q) - V(p) \|_x \leq L_1 \|s \|_p$. 
    \item A linear operator $H(p): T_p\M \xrightarrow{} T_p\M$ is $L_2$-Lipschitz continuous if for all $p \in \M$ and $q = {\rm Exp}_p(s) \in \Omega$, it satisfies $\| \Gamma_{q}^p \circ H(q) \circ \Gamma_p^q - H(p) \|_p \leq L_2 \| s \|_p$.
\end{enumerate}
\end{definition}

\begin{definition}[Polyak--\L{}ojasiewicz (PL) condition on Riemannian manifold \cite{zhang2016riemannian,kasai2018riemannian,han2021improved}]
\label{riemannian_PL_condition}
A function $h:\M \xrightarrow{} \sR$ satisfies the PL condition on Riemannian manifold if for any $p \in \M$, there exists $\delta > 0$ such that $ \frac{1}{2}\| \grad h(p) \|^2_p \geq \delta( h(p) - h(p^*))$, where $p^* = \argmin_{p \in \M} h(p)$ is the global minimizer of $h$.
\end{definition}

The following lemma shows the connection between smoothness of a function on manifold and its Lipschitz Riemannian gradient, which is fundamental for convergence analysis.
\begin{lemma}[Lipschitz Riemannian gradient and smoothness \cite{boumal2020introduction}]
\label{lemma_smoothness_grad}
For a function $h: \M \xrightarrow{} \sR$, its Riemannian gradient is $L_1$-Lipschitz continuous if and only if $\| \hess h(p) \|_p \leq L_1$ for all $p \in \M$. Suppose $h$ has $L_1$-Lipschitz Riemannian gradient, then $h$ is $L_1$-smooth on $\M$ with $|h(q) -  h(p) - \langle \grad h(p), s \rangle_p |  \leq \frac{L_1}{2} \| s\|_p^2$, for all $q = {\rm Exp}_p(s) \in \Omega$ and $p \in \M$.
\end{lemma}

{
\subsection*{Notations}
Here, we summarize the main notations used in the paper. We use $\nabla, \nabla^2$, $\grad$, and $\hess$ to represent the Euclidean gradient, Euclidean Hessian, Riemannian gradient, and Riemannian Hessian respectively. The boldface $\bnabla$ is used to denote the Riemannian connection. For a bi-function $f(x,y)$, we denote $\nabla_x f(x,y)$, $\nabla_y f(x,y)$ as the partial Euclidean derivative with respect to $x$, $y$, respectively, if $x, y \in \sR^d$. Similarly for $x, y\in \M$, $\grad_x f(x,y), \grad_y f(x,y)$ denote the partial Riemannian gradients. We also make use of $\grad_{xy}^2 f(x,y), \grad_{yx}^2 f(x,y)$ to represent the  Riemannian cross derivatives. We use $\langle \cdot, \cdot \rangle_p^\M$ to represent the Riemannian metric at $p \in\M$. When the manifold considered is clear, we omit the superscript for clarity. Furthermore, we use $\langle \cdot, \cdot \rangle_2$ to denote the Euclidean inner product. 
}

\subsection{Min-max optimization}
\label{related_work_sect}
Here we discuss related works on min-max optimization both in the Euclidean space and on Riemannian manifolds.

\subsubsection*{In Euclidean space}
In the Euclidean space (i.e., $\mathbb{R}^n$), the standard gradient descent ascent (GDA) that follows the min-max gradient is known to cycle or diverge for simple convex-concave objectives \cite{mertikopoulos2018cycles}. To address the cycling issue, the optimistic gradient descent ascent algorithm (OGDA)~\cite{popov1980modification} modifies the GDA update to include an additional gradient momentum. 
On the other hand, the extra-gradient algorithm~(EG) \cite{extragradient1976} employs an additional min-max gradient step at every iteration. 
As shown in \cite{mokhtari2020unified,mokhtari2020convergence}, both OGDA and EG methods approximate the proximal point method \cite{rockafellar1976monotone} and converge sublinearly under convex-concave settings \cite{nemirovski2004prox,monteiro2010complexity} and linearly under strongly-convex-strongly-concave settings \cite{tseng1995linear,mokhtari2020unified}. 

However, for the more general nonconvex-nonconcave settings, finding a global saddle point satisfying \eqref{global_saddle_pt} is difficult and several existing works \cite{chasnov2020convergence,adolphs2019local,mazumdar2019finding,schafer2019competitive} aim to find a {local} saddle point that satisfies \eqref{global_saddle_pt} in a local neighbourhood. It should be noted that when the function is convex-concave, all local saddle points are global.

A necessary set of conditions for the saddle points is that they satisfy the first-order stationarity, i.e., the gradients with respect to $x$ and $y$ vanish. This motivates the Euclidean Hamiltonian gradient descent (EHGD) \cite{mescheder2017numerics,balduzzi2018mechanics,abernethy2019last,loizou2020stochastic} approach for solving the min-max problem, which minimizes the sum of the squares of the gradient norms with respect to $x$ and $y$. 
It should be noted that EHGD works under the assumption that all such stationary points are global min-max saddle points \cite{abernethy2019last,loizou2020stochastic}. Cases are discussed where this assumption is satisfied, which allows EHGD to converge to a global min-max saddle point of the original min-max problem \cite{abernethy2019last,loizou2020stochastic}. Further, studies \cite{mescheder2017numerics,balduzzi2018mechanics,abernethy2019last,loizou2020stochastic} demonstrate good empirical performance of EHGD in a variety of applications. 

It should be noted that EHGD approaches have only been studied for unconstrained problems in the Euclidean space. Challenges in the constrained settings appear with definition of the Hamiltonian and subsequent analysis.




\subsubsection*{On Riemannian manifolds} 
There is a growing theoretical and empirical interest in solving min-max problems under Riemannian optimization framework \cite{lin2020projection,huang2021riemannian,huang2020gradient,zhang2022minimax}. 
An extension of the GDA algorithm to manifolds, named RGDA, has been proposed in \cite{huang2020gradient}. However, \cite{huang2020gradient} considers a min-max setting in which the minimization problem (in $x$) is on a manifold, but the maximization problem (in $y$) is on a convex set.  In addition, it analyzes the convergence when the maximization problem over $y$ is strongly concave. Hence, \cite{huang2020gradient} does not study the general Riemannian min-max problem (\ref{main_riem_minmax}). It discusses the convergence of their algorithm to first-order stationary points of the min-max problem. Additionally, they propose different stochastic extensions of their algorithm and analyze their convergence.

Recently, \cite{zhang2022minimax} has proposed a Riemannian corrected extra-gradient algorithm (RCEG) for the Riemannian min-max problems (\ref{main_riem_minmax}), which contains two steps. First, RCEG takes a step similar to the RGDA update. Then, starting from the newly obtained point, RCEG combines the RGDA direction with the direction of the first step. In the g-convex-concave settings, this correction allows \cite{zhang2022minimax} to prove {(local)} convergence of RCEG to global min-max saddle points of (\ref{main_riem_minmax}). {The convergence is however analyzed only for averaged iterates. After the submission of this work, we notice that a recent paper \cite{jordan2022first} proves both last-iterate and average-iterate convergence of RCEG to saddle points under g-convex-concave and g-strongly-convex-concave settings. They also discuss average-iterate and last-iterate convergence of RGDA under g-convex-concave and g-strongly-convex-concave settings, respectively. Nevertheless, the convergence analysis requires a bounded domain (and curvature) and a carefully chosen stepsize that depends on the curvature and diameter bound of the domain. In contrast, we have shown in this work global convergence to saddle points with stepsize that only depends on the Lipschitz constants of the objective.

More details on the RGDA and RCEG algorithms as well as the comparisons on the convergence analysis are in Appendix \ref{app:sec:rgda_rceg}.
}

\section{Riemannian Hamiltonian gradient methods}
\label{sect_hamiltonian}




As mentioned earlier, the Euclidean Hamiltonian approach \cite{mescheder2017numerics,balduzzi2018mechanics,abernethy2019last,loizou2020stochastic} is a popular approach to tackle the min-max problem (\ref{main_riem_minmax}) when $\M_x $ and $\M_y$ are restricted to the Euclidean space. Specifically, the Euclidean Hamiltonian function $\mathcal{E}$ is defined as,
\begin{equation}
    \mathcal{E}(x,y) := \frac{1}{2}\| \nabla_xf(x,y) \|_2^2 + \frac{1}{2}\| \nabla_yf(x,y) \|_2^2, \label{eq:EuclideanHamiltonian}
\end{equation}
where $\nabla_xf(x,y)$ and $\nabla_y f(x,y)$ are the partial derivatives of $f$ with respect to $x$ and $y$, respectively. Here, $\|\cdot \|_2$ denotes the Frobenius norm. The global minimum of the function $\mathcal{E}$ is attained when $ \mathcal{E}(x,y) = 0$, i.e., $\nabla_xf(x,y) = 0$ and $\nabla_yf(x,y) =0$. This corresponds to a first-order stationary point of the function $f$. Hence, minimization of $\mathcal{E}$ in (\ref{eq:EuclideanHamiltonian}), becomes a good proxy to solve the original min-max problem.

Building on the Euclidean Hamiltonian approach, generalization to the Riemannian min-max problem (\ref{main_riem_minmax}) requires understanding of first-order stationary points on manifolds $\M_x$ and $\M_y$. These are necessarily identified with the points where the Riemannian gradient of $f$ vanishes. This leads to our proposed Riemannian Hamiltonian function as 
\begin{equation}
    \gH(x,y) := \frac{1}{2}\| \grad_x f(x,y)  \|_x^2 + \frac{1}{2}\| \grad_y f(x,y) \|_y^2, \label{RiemHamiltonian}
\end{equation}
where $\grad_x f(x,y)$ and $\grad_y f(x,y)$ are the Riemannian partial gradients of $f$ with respect to $x$ and $y$ respectively. Here, $\| \grad_x f(x,y)  \|_x^2 = \langle \grad_x f(x,y), \grad_x f(x,y) \rangle_x^{\gM_x}$ is the square of the gradient norm in the Riemannian metric sense on $\gM_x$. Similarly, $\| \grad_y f(x,y)  \|_y^2 \allowbreak = \langle \grad_y f(x,y), \grad_y f(x,y) \rangle_y^{\gM_y}$ is the square of the norm on $\gM_y$. 

\begin{remark}
The proposed Riemannian Hamiltonian function (\ref{RiemHamiltonian}) generalizes the Euclidean Hamiltonian function (\ref{eq:EuclideanHamiltonian}) in two different ways: 
\begin{itemize}
    \item[1)] Equation (\ref{RiemHamiltonian}) implicitly embeds the manifold geometry of $\M_x, \M_y$ into the Hamiltonian function.
    \item[2)] Equation (\ref{RiemHamiltonian}) generalizes the Euclidean metric considered in (\ref{eq:EuclideanHamiltonian}) to a Riemannian metric. This generalization allows to use other varying metrics for min-max problems \textit{in the Euclidean space}, e.g., the Fisher information metric \cite{douik2019manifold} or real-projective space metrics \cite[Chapter~2]{absil2009optimization}. 
\end{itemize}
\end{remark}





It should be noted that the Riemannian Hamiltonian (\ref{RiemHamiltonian}) can be viewed on the product manifold $\M = \M_x \times \M_y$, i.e., for $p = (x, y) \in \M$, the Riemannian gradient is $\grad_p f(p) = (\grad_x f(x,y), \grad_y f(x,y))$, and therefore, $\gH(x,y) =  \|\grad_p f(p) \|^2_p$. Hence, we propose to solve the following problem on the product manifold as
\begin{equation}
    \min_{p \in \M} \left\{ \gH(p) =\frac{1}{2}\| \grad f(p) \|^2_p \right \} . \label{MainRiemHamiltonian}
\end{equation}
Similar to the EHGD approaches \cite{abernethy2019last,loizou2020stochastic}, we work with the following assumption.
\begin{assumption}
\label{stationary_assumption}
The objective $f$ admits at least one stationary point and all stationary points are global min-max saddle points.
\end{assumption}
It is worth noticing that under Assumption \ref{stationary_assumption}, solving (\ref{MainRiemHamiltonian}) is equivalent to solving (\ref{main_riem_minmax}). On Riemannian manifolds, Assumption \ref{stationary_assumption} holds when $f$ is g-convex-concave. 

We now show that the Riemannian gradient of the Riemannian Hamiltonian $\gH(p)$ admits a simple expression.

\begin{proposition}
\label{gradient_hamiltonian_prop}
Riemannian gradient of $\gH$ is $\grad \gH(p) = \hess f(p) [\grad f(p)]$.
\end{proposition}
\begin{proof}
First, we see that $\gH$ is a smooth function on the manifold due to the smoothness of $f$ and its Riemannian gradient (formally characterized later in Proposition \ref{prop_ham_smooth}). For any smooth vector field ${U}: \M \xrightarrow{} T\M$, denoted as $U \in \mathfrak{X}(\M)$, we have ${U} \gH = \langle \grad \gH, {U} \rangle$, where $\langle \cdot, \cdot\rangle$ is the Riemannian metric (on any tangent space). Let $\bnabla$ be the Riemannian connection (or the Levi-Civita connection) of $\M$, which provides a way to differentiate vector fields on manifolds. By definition, the Riemannian connection satisfies the metric compatibility property \cite{absil2009optimization,boumal2020introduction}, i.e., ${U} \langle {V}, {W} \rangle = \langle \bnabla_{U} {V}, {W} \rangle + \langle {V}, \bnabla_{U} {W} \rangle$ for any vector fields ${U}, {V}, {W}$. Also, by definition, application of the Riemannian Hessian of $f: \M \xrightarrow{} \sR$ along a vector field ${U}$ is $\hess f [{U}] = \bnabla_{{U}} \grad f$. Based on these claims, we show
\begin{align*}
    {U} \gH =  \frac{1}{2} {U} \langle \grad f, \grad f \rangle = \langle \bnabla_{U} \grad f, \grad f \rangle &= \langle \hess f [{U}], \grad f \rangle \\
    &= \langle \hess f[\grad f], {U} \rangle,
\end{align*}
where the last equality follows from the self-adjoint property of the Riemannian Hessian. The proof is complete by noticing $\langle \hess f[\grad f], {U} \rangle = \langle \grad \gH, {U} \rangle$ for any ${U}$.
\end{proof}

\begin{remark}
The importance of the varying metric in the proposed Riemannian Hamiltonian (\ref{RiemHamiltonian}), can be observed in Proposition \ref{gradient_hamiltonian_prop}, where we obtain a simple expression for the Riemannian gradient of $\gH$. This allows to connect the properties of $\gH$ with that of the min-max objective $f$, discussed in detail later in Section \ref{sec:problem_classes}.
\end{remark}

\begin{remark} 
It should be noted that for the Euclidean case when $x \in \sR^m, y \in \sR^n$, existing works \cite{balduzzi2018mechanics,abernethy2019last,loizou2020stochastic} analyze the Hamiltonian methods in the form of $\bJ^\top v$, where $\bJ$ is an asymmetric Jacobian matrix and $v$ is the min-max gradient $(\nabla_xf(x,y), \allowbreak  -\nabla_y f(x,y))$. For the same setting, however, Proposition \ref{gradient_hamiltonian_prop} obtains the Hamiltonian gradient as $\bH \nabla f$, where $\bH$ and $\nabla f$ are the (Euclidean) Hessian matrix and gradient vector $\nabla f = (\nabla_x f(x,y),  \nabla_y f(x,y))$, respectively. This is not surprising as $\bJ^\top v = \bH \nabla f$. Proposition \ref{gradient_hamiltonian_prop} allows to analyze the performance of the Riemannian Hamiltonian approach in terms of the symmetric Riemannian Hessian operator. The analysis in \cite{abernethy2019last,loizou2020stochastic} heavily rely on the matrix structure of $\bJ$ and makes use of the linear algebraic properties of the Jacobian. Our approach, thanks to Proposition \ref{gradient_hamiltonian_prop}, adheres to general Riemannian manifolds as we directly deal with the operator, which is independent of the matrix structure. Hence, many of the subsequent analysis in this paper differ from \cite{abernethy2019last,loizou2020stochastic}. 
\end{remark}

\begin{algorithm}[t]
 \caption{Riemannian Hamiltonian methods (RHM)}
 \label{RHM}
 \begin{algorithmic}[1]
  \STATE Initialize $p_0 = (x_0,y_0) \in \M$.
  \FOR{$t = 0,...,T$}
  \STATE Compute the step $\xi(p_t)$ from the gradient $\grad \gH(p_t) = \hess f(p_t)[\grad f(p_t)]$.
  \STATE Update $p_{t+1} = {\rm Exp}_{p_t}\big( \xi(p_t) \big)$.
  \ENDFOR
  \STATE \textbf{Output:} $p_T$.
 \end{algorithmic} 
\end{algorithm}

To minimize the Riemannian Hamiltonian \eqref{MainRiemHamiltonian}, one can apply first-order Riemannian solvers including Riemannian steepest descent \cite{udriste2013convex}, Riemannian conjugate gradient \cite{ring2012optimization}, or second-order solvers, such as Riemannian trust-regions \cite{absil2007trust,boumal2015riemannian}, provided the Hessian (or approximated Hessian) of the Hamiltonian is available. We refer to such class of methods for solving min-max problems on manifolds collectively as Riemannian Hamiltonian methods (RHM). Its procedures are outlined in Algorithm \ref{RHM}, where the step $\xi(p_t)$ is computed depending on the selected solver.

{
\begin{remark}
We remark that Algorithm \ref{RHM} aims to solve a proxy optimization problem \eqref{MainRiemHamiltonian} where we only require the first-order information, i.e., $\grad \gH(p)$. Although the Hessian of $f$, i.e., $\hess f(p)[\grad \gH(p)]$ is used in Algorithm \ref{RHM}, this essentially corresponds to the gradient information of the proxy problem. Furthermore, from the computational perspective, Algorithm \ref{RHM} only requires one evaluation of Hessian-vector product per iteration. This is much more efficient than second-order methods, such as Riemannian trust region \cite{absil2007trust,boumal2015riemannian} or cubic regularized Newton methods \cite{agarwal2021adaptive} that require at least several oracles to such Hessian vector product each iteration. Finally, when Hessian of $f$ is unavailable, we find finite difference approximation is sufficient to achieve convergence in practice.  
\end{remark}
}

We analyze the performance of the proposed RHM. In particular, we aim to obtain the global minimizer $p^*$ of $\gH$, which satisfies $\gH(p^*) = 0$ with RHM. However, this may not always be numerically tractable without additional structures on the Riemannian Hamiltonian. One such structure is assuming the Riemannian Hamiltonian is g-convex, for which RHM converges to the optimal $p^*$  (g-convexity guarantees convergence to global optimality). This, however, may not lead to interesting problem classes for $f$. Moreover, there is no guarantee that $\gH$ is a g-convex even when $f$ is g-convex-concave. 

Another interesting structure is the Polyak--\L{}ojasiewicz (PL) condition. The PL condition \cite{polyak1963gradient} amounts to a sufficient condition to establish linear convergence for gradient-based methods to global optimality \cite{karimi2016linear}. The Riemannian version of the PL condition (Definition \ref{riemannian_PL_condition}) has been studied in \cite{zhang2016riemannian,kasai2018riemannian,zhou2019faster,han2021improved}. In Section \ref{convergence_main_subsect}, we impose the Riemannian PL condition on the Hamiltonian $\gH$ as it allows convergence of RHM to global optimality. It should be noted that functions satisfying the Riemannian PL condition subsume g-(strongly)-convex functions. In Section \ref{sec:problem_classes}, we discuss many interesting function classes of $f$ that allow the Hamiltonian $\gH$ to satisfy the condition.



\subsection{Convergence analysis}
\label{convergence_main_subsect}

To analyze the convergence of RHM, we focus on the Riemannian steepest descent direction in the main text, i.e., $\xi(p_t) =  - \eta_t  \grad \gH(p_t)$ with either fixed stepsize or variable stepsize computed from backtracking line-search \cite{boumal2019global,boumal2020introduction}. We include the details of implementing the Riemannian conjugate gradient and Riemannian trust-region methods together with their convergence analysis in Appendix \ref{RHM_CG_TR}. We make the following standard assumption \cite{absil2009optimization, boumal2020introduction,zhang2016riemannian,sato2019riemannian,zhang2022minimax} throughout the rest of the paper. We assume that our manifolds $\M_x$ and $\M_y$ are complete (and so is $\M = \M_x \times \M_y$).
\begin{assumption}
\label{smoothness_assumption}
The objective $f$, its Riemannian gradient, and its Riemannian Hessian are $L_0, L_1, L_2$-Lipschitz continuous, respectively. 
\end{assumption}
In the next proposition, we show that the Riemannian Hamiltonian $\gH$ is $L$-smooth.
\begin{proposition}[Smoothness of Riemannian Hamiltonian]
\label{prop_ham_smooth}
Under Assumption \ref{smoothness_assumption}, the Riemannian Hamiltonian is $L$-smooth with $L = L_0 L_2 + L_1^2$, i.e., for any $p \in \M, q = {\rm Exp}_p(\xi)$, it satisfies $\gH(q) \leq \gH(p) + \langle \grad\gH(p), \xi \rangle_p + \frac{L}{2} \| \xi \|^2_p$.
\end{proposition}
\begin{proof}
According to Lemma \ref{lemma_smoothness_grad}, it is sufficient to show that the Riemannian gradient of $\gH$ is $L$-Lipschitz. From Proposition \ref{gradient_hamiltonian_prop} and Assumption \ref{smoothness_assumption}, we have for any $p \in \M$, $q = {\rm Exp}_{p}(s) \in \Omega$, the domain of exponential map around $p$,
\begin{align*}
    \| \Gamma_p^q \, \grad \gH(p) - \grad \gH(q) \|_q &= \| \Gamma_p^q \, \hess f(p)[\grad f(p)] -  \hess f(q)[\grad f(q)] \|_q \\
    &\leq \| \Gamma_p^q \, \hess f(p)[\grad f(p)] - \hess f(q) [\Gamma_p^q \, \grad f(p)] \|_q \\
    &\qquad + \| \hess f(q) [\Gamma_p^q \grad f(p)] - \hess f(q)[\grad f(q)] \|_q \\
    &= \| \hess f(p)[\grad f(p)] - \Gamma_q^p \, \hess f(q)[\Gamma_p^q \, \grad f(p)] \|_p \\
    &\qquad + \| \hess f(q) [\Gamma_p^q \grad f(p)] - \hess f(q)[\grad f(q)] \|_q \\
    &\leq L_2 \| \grad f(p)\|_p \, \| s\|_p + L_1 \| \Gamma_p^q \, \grad f(p)  -\grad f(q)\|_q \\
    &\leq  (L_0 L_2 + L_1^2) \| s\|_p,
\end{align*}
where we apply the triangle inequality and the isometry property of parallel transport. 
\end{proof}

If the Hamiltonian $\gH$ satisfies the Riemannian PL condition, then we show that Algorithm \ref{RHM} with the steepest descent update (RHM-SD) converges linearly to the global minimizer of $\gH$.

We begin with the convergence result for RHM-SD with fixed stepsize.


\begin{theorem}[Linear convergence of RHM-SD with fixed stepsize]
\label{theorem_convergence_RHGD}
Let $f$ satisfy Assumption \ref{smoothness_assumption} and $\gH$ satisfy the Riemannian PL condition, i.e., $\frac{1}{2}\| \grad \gH(p) \|^2_p \geq \delta \gH(p)$ (with $\gH(p^*) = 0$). Consider Algorithm \ref{RHM} using steepest descent direction with fixed stepsize $\eta_t = \eta = 1/L$, where $L = L_0 L_2 + L_1^2$. Then, the iterates $p_t$ satisfy $\| \grad f(p_t) \|^2_{p_t} \leq (1  - \frac{\delta}{L} )^t \|\grad f(p_0) \|^2_{p_0}$. 
\end{theorem}
\begin{proof}
From the smoothness of the Riemannian Hamiltonian $\gH$ (Proposition~\ref{prop_ham_smooth}, Lemma~\ref{lemma_smoothness_grad}) and the gradient update in Algorithm \ref{RHM}, we have
\begin{align*}
    \gH(p_{t+1}) - \gH(p_t) &\leq - \eta \| \grad \gH(p_t) \|^2_{p_t} +\frac{\eta^2 L}{2} \| \grad \gH(p_t) \|_{p_t}^2 \\
    &= - \frac{1}{2L}  \| \grad \gH(p_t) \|_{p_t}^2 \leq - \frac{\delta}{L} \gH(p_t),
\end{align*}
where the last inequality employs the Riemannian PL condition. This leads to $\gH(p_{t+1}) \leq (1- \frac{\delta}{L}) \gH(p_t)$. Applying this result recursively completes the proof. 
\end{proof}

Line-search methods are practically favourable because they adapt the stepsize without requiring the knowledge of the Lipschitz constant $L$. Here, we consider the backtracking line-search for choosing stepsize $\eta_t$ for Riemannian steepest descent, which is commonly used in practice. Given an initial stepsize $\bar{\eta}$, the backtracking line-search iteratively decreases the stepsize by a factor of $\varrho \in (0,1)$ until the Armijo-type sufficient decrease condition is satisfied, i.e., 
\begin{equation}
    \gH(p_t) - \gH({\rm Exp}_{p_t}(\eta_t \zeta(p_t)) \geq  r_1 \eta_t \langle -\grad \gH(p_t), \zeta(p_t) \rangle_{p_t}, \label{armijo_condition}
\end{equation}
for some update direction $\zeta(p_t)$. The complete procedure is included in Appendix \ref{appendix_linesearch}. We next present the convergence for RHM-SD with backtracking linesearch.

\begin{theorem}[Linear convergence of RHM-SD with backtracking line-search]
\label{theorem_convergence_linesearch}
Under the same setting as in Theorem \ref{theorem_convergence_RHGD}, consider Algorithm \ref{RHM} using the steepest descent direction with backtracking line-search, parameters $r_1, \rho \in (0,1)$, and an initial stepsize $\bar{\eta}$. Then, the iterates $p_t$ satisfy
$$\| \grad f(p_t) \|^2_{p_t} \leq  \left( 1  - 2\min \left\{ \bar{\eta} r_1, \frac{2\varrho(1-r_1)r_1}{L} \right\} \delta \right)^t \|\grad f(p_0) \|^2_{p_0}.$$
\end{theorem}
\begin{proof}
Given $\gH$ is $L$-smooth, the proof follows from \cite[Lemma~4.12]{boumal2020introduction} and the Riemannian PL condition.
\end{proof}

{
\begin{remark}
It should be highlighted that the convergence rates to global saddle points obtained in Theorems \ref{theorem_convergence_RHGD}, \ref{theorem_convergence_linesearch} are independent of the manifold curvature (which we achieve via solving a proxy Hamiltonian problem \eqref{MainRiemHamiltonian}). In contrast, the linear convergence rates shown in \cite{zhang2022minimax,jordan2022first} are curvature dependent. 
\end{remark}
}

\subsection{Important problem classes for RHM}
\label{sec:problem_classes}

We now discuss the instances of $f$ where the Riemannian Hamiltonian satisfies the Riemannian PL condition (Definition \ref{riemannian_PL_condition}). This allows RHM (Algorithm \ref{RHM}) to converge to global min-max saddle points of (\ref{main_riem_minmax}). 

From the expression of $\grad \gH(p)$ in Proposition \ref{gradient_hamiltonian_prop}, we observe that if all eigenvalues of $\hess f(p)$ are lower bounded in magnitude (i.e., $|\lambda| \geq \alpha > 0$), then the Riemannian Hamiltonian $\gH$ satisfies the Riemannian PL condition with $\delta = \alpha^2$. This is because
\begin{equation}\label{eq:eigvalue_bound_Hessian}
    \underbrace{\frac{1}{2}\| \grad \gH(p) \|^2_p \geq \alpha^2 \gH(p)}_{\text{Riemannian PL condition}}
    \Leftrightarrow \underbrace{\frac{1}{2}\| \hess f(p)[\grad f(p)] \|^2_p \geq \frac{\alpha^2}{2} \| \grad f(p) \|^2_p}_{\text{Required eigenvalue bound on }\hess f(p)}. 
\end{equation}
Our aim, therefore, is to identify classes of $f$ that satisfy the right hand side of (\ref{eq:eigvalue_bound_Hessian}). We provide three cases where the Riemannian PL condition is naturally satisfied on the Riemannian Hamiltonian $\gH$, which generalize the results in \cite{abernethy2019last} to Riemannian manifolds. These include the cases when the objective $f$ is g-strongly-convex-concave and when $f$ is smooth with sufficient geodesic linearity.

In order to analyze function classes of $f$ that lead to (\ref{eq:eigvalue_bound_Hessian}), we require the following results on the Riemannian Hessian $\hess f(p)$ of the product manifold $\gM$ (which are of independent interest as well).
\begin{itemize}
    \item[1)] Decomposition of the Riemannian Hessian $\hess f(p)$ and adjoint property of the cross derivatives. This is shown in Appendix \ref{app:sec:productHessian}. 
    \item[2)] We establish general lower bounds on the eigenvalue magnitude of the Riemannian Hessian, which we include in Appendix \ref{essential_lemma_appendix}. 
\end{itemize}
The above results help to bound the eigenvalues of $\hess f(p)$ in terms of the spectrum of $\hess_x f(x,y)$, $\hess_y f(x,y)$, and the cross derivatives $\grad_{xy}^2f(x,y), \grad_{yx}^2f(x,y)$. We now present the main results below.

\begin{proposition}[Geodesic strongly convex strongly concave]
\label{g_strongly_convex_concave_prop}
Let $f(x,y)$ be geodesic strongly convex in $x$ and geodesic strongly concave in $y$ with parameter $\mu > 0$. Then, $\gH$ satisfies the Riemannian PL condition (\ref{eq:eigvalue_bound_Hessian}) with $\delta = \mu^2$. 
\end{proposition}
\begin{proof}

We show that if there exists an eigenpair $(\lambda, \xi)$ of $\hess f(p)$ such that $|\lambda| < \mu$ with $p = (x,y), \xi = (u,v)$, then it leads to a contradiction. From the expression of the Riemannian Hessian in Proposition \ref{riem_hess_prop}, we have
\begin{align*}
     \hess_x f(x,y)[u] + \grad^2_{yx}f(x,y)[v] &= \lambda u   \\
     \hess_y f(x,y)[v] + \grad^2_{xy}f(x,y)[u] &= \lambda v.
\end{align*}
This can be equivalently written as
\begin{align}
    \langle \hess_x f(x,y)[u], u \rangle_x + \langle \grad^2_{yx}f(x,y)[v], u\rangle_x &= \lambda \| u\|^2_x \label{strong_cc_eq1}\\
    \langle \hess_y f(x,y)[v], v \rangle_y + \langle \grad^2_{xy}f(x,y)[u], v\rangle_y &= \lambda \| v\|^2_y \label{strong_cc_eq2}.
\end{align}
From \eqref{strong_cc_eq1}, we obtain 
\begin{equation}
    \langle \grad^2_{yx} f(x,y)[v], u\rangle_x =  -\langle u, (\hess_x f(x,y) - \lambda \, \id )[u] \rangle_x, \label{strong_cc_eq3}
\end{equation}
where $\id$ is the identity operator. From the symmetry of the Riemannian cross derivatives (Proposition \ref{symmetry_cross_derivative_prop}), we can substitute \eqref{strong_cc_eq3} into \eqref{strong_cc_eq2}, which gives
\begin{equation}
    \langle \hess_y f(x,y)[v], v\rangle_y -\langle u, (\hess_x f(x,y) - \lambda \, \id )[u] \rangle_x = \lambda \| v\|^2_y. \label{strong_cc_eq4}
\end{equation}
The geodesic strong convexity in $x$ and geodesic strong concavity in $y$ leads to $\hess_x f(x,y) \succeq \mu \,\id$ and $\hess_y f(x,y) \preceq -\mu \,\id$ respectively. Thus, the LHS of \eqref{strong_cc_eq4} is smaller than $-\mu$, which contradicts $|\lambda| < \mu$. Thus, all eigenvalues of $\hess f(p)$ satisfies $|\lambda| \geq \mu$. 
\end{proof}

\begin{proposition}[Smooth and geodesic linear]
\label{nonconvex_g_linear_prop}
Let $\sigma_{\min}(\grad^2_{xy} f(x,y) ) \allowbreak \geq \tau > 0$ and let $f(x,y)$ be geodesic linear in one variable and has $L_1$-Lipschitz Riemannian gradient in another variable. Then, $\gH$ satisfies the Riemannian PL condition (\ref{eq:eigvalue_bound_Hessian}) with $\delta = \frac{\tau^4}{2\tau^2 +L_1^2}$.
\end{proposition}
\begin{proof}
Without loss of generality, we assume $f(x,y)$ has $L_1$-Lipschitz gradient in $x$ and geodesic linear in $y$. The geodesic linearity in $y$ implies that $\hess_y f(x,y) = 0$, and therefore, we can apply Lemma \ref{lemma_hess_bound_linear}, which shows 
$$\lambda_{\rm |min|}^2(\hess f(p)) \geq \frac{\sigma_{\min}^4(\grad^2_{xy} f(x,y))}{2\sigma_{\min}^2(\grad^2_{xy} f(x,y)) + \| \hess_x f(x,y) \|^2_x}.$$ 
Also, from Lemma \ref{lemma_smoothness_grad}, we have $\| \hess_x f(x,y) \|_x^2 \leq L_1^2$. Finally, from the assumption $\sigma_{\min}(\grad^2_{xy} f(x,y)) \allowbreak \geq \tau$, the proof is complete. 
\end{proof}

\begin{proposition}[Smooth and sufficiently geodesic-bilinear]
\label{sufficient_bilinear_prop}
Let $0 < \tau \leq \sigma(\grad^2_{xy} f(x,y) ) \leq \Upsilon$ and let $f(x,y)$ has $L_1$-Lipschitz Riemannian gradient for both $x$ and $y.$ Define $\mu = \lambda_{\rm |min|}(\hess_x f(x,y))$, $\rho = \lambda_{\rm |min|}(\hess_y f(x,y))$ and let the sufficient geodesic-bilinearity condition holds: $(\tau^2 + \mu^2)(\tau^2 + \rho^2) - 4 L_1^2 \Upsilon^2 > 0$. Then, $\gH$ satisfies the Riemannian PL condition (\ref{eq:eigvalue_bound_Hessian}) with $\delta = \frac{(\tau^2 + \mu^2)(\tau^2 + \rho^2) - 4 L_1^2 \Upsilon^2}{2 \tau^2 + \rho^2 + \mu^2}$.
%
%
\end{proposition}
\begin{proof}
We can directly apply Lemma \ref{lemma_sufficient_linear_H} and set $a = 2 \tau^2 + \rho^2 + \mu^2$ and $b = (\tau^2 + \mu^2)(\tau^2 + \rho^2) - 4 L_1^2 \Upsilon^2 >0$ by assumption. 
\end{proof}

It is worth noticing that the sufficient geodesic-bilinearity condition 
in Proposition \ref{sufficient_bilinear_prop} can be interpreted as requiring a sufficiently large weight on the geodesic-bilinear component in the objective function $f$. To see this, suppose $f(x,y) = c_l f_0(x,y) + f_1(x) + f_2(y)$ where $f_0$ is geodesic linear in each $x$ and $y$ (i.e. bilinear) with the weight $c_l > 0$ and $f_1, f_2$ have $L_1$-Lipschitz Riemannian gradient. Because by definition, Riemannian Hessian of a geodesic linear function is zero, $f$ has $2L_1$-Lipschitz Riemannian gradient (by Lemma \ref{lemma_smoothness_grad}). Let $\tau_0, \Upsilon_0$ be the minimum and maximum singular values of $\grad^2_{xy} f_0(x,y)$. Then, $\tau = c_l \tau_0, \Upsilon = c_l \Upsilon_0$. The sufficient geodesic bilinearity condition is satisfied for $c_l \geq 4L_1 \Upsilon_0/ \tau_0^2$. This is because $(\tau^2 + \mu^2)(\tau^2 + \rho^2) > \tau^4 = c_l^4 \tau_0^4 \geq 16 L_1^2 c_l^2 \Upsilon_0^2 = 16L_1^2 \Upsilon^2$. 

\begin{remark}
When $f_1(x) = f_2(y) = 0$, it should be noted that $f(x,y) = c_l f_0(x,y)$ is geodesic bilinear. Additionally, $\gH$ satisfies the Riemannian PL condition with $\delta = \frac{c_l^2\tau_0^2}{2}$. 
\end{remark}

{
\subsection{The geodesic-bilinear example}
\label{motivating_example_sect}

\begin{figure}[t]
    \centering
    \subfloat[Iteration]{\includegraphics[scale=0.28]{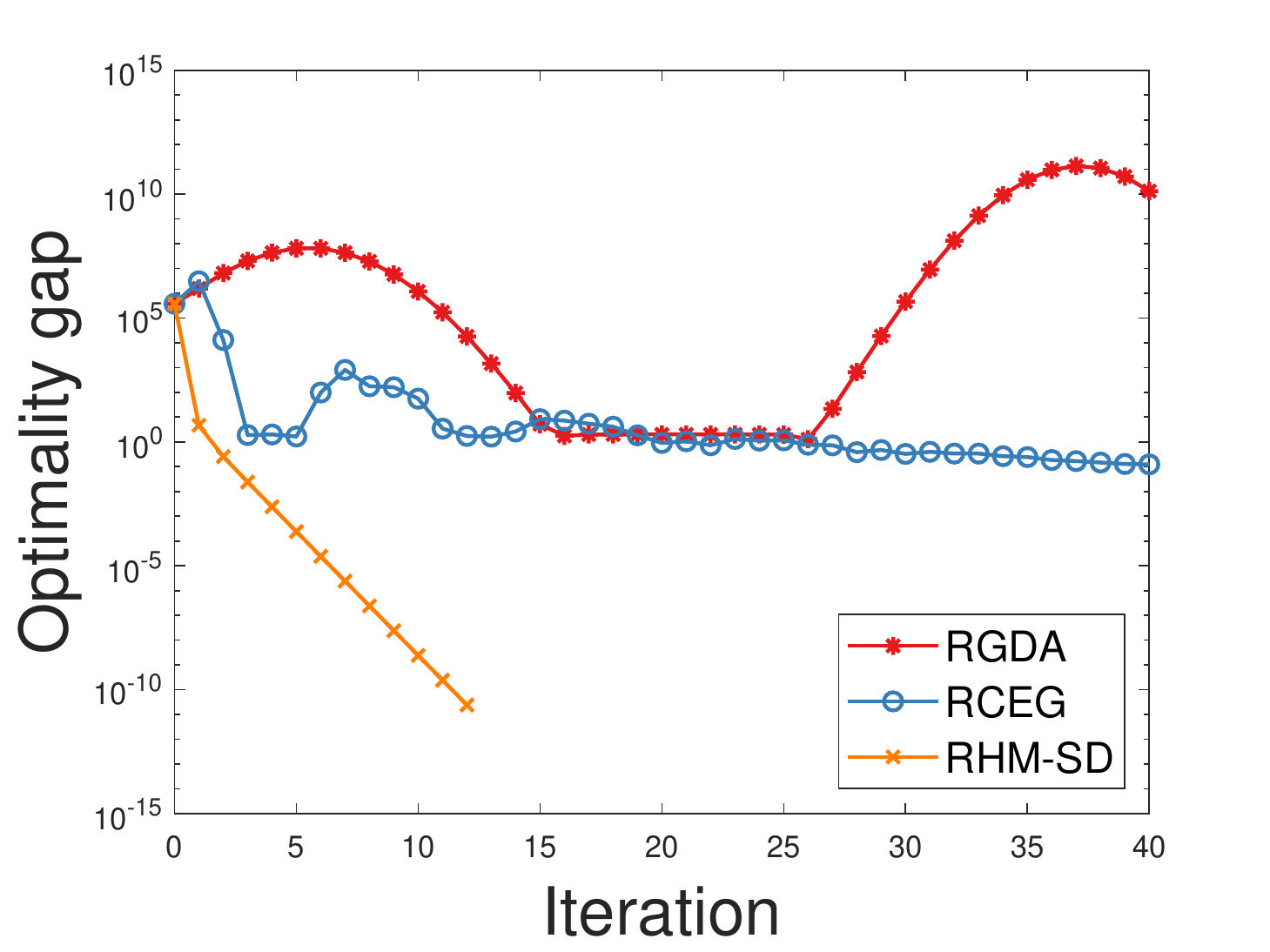}  \label{illu_iter}}
    \subfloat[Time]{\includegraphics[scale=0.28]{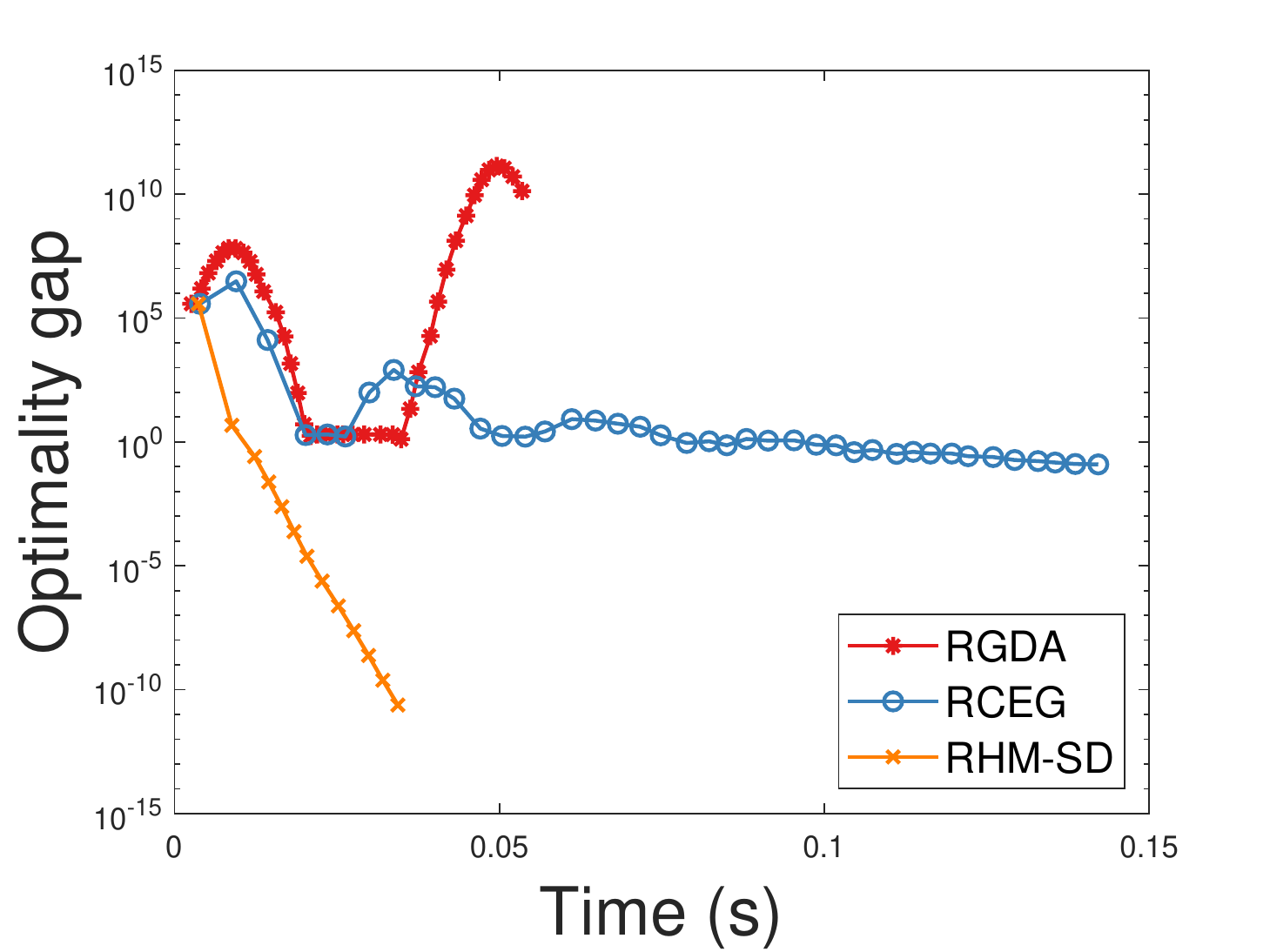}  \label{illu_time}} 
    \subfloat[Iteration (zoom out)]{\includegraphics[scale=0.28]{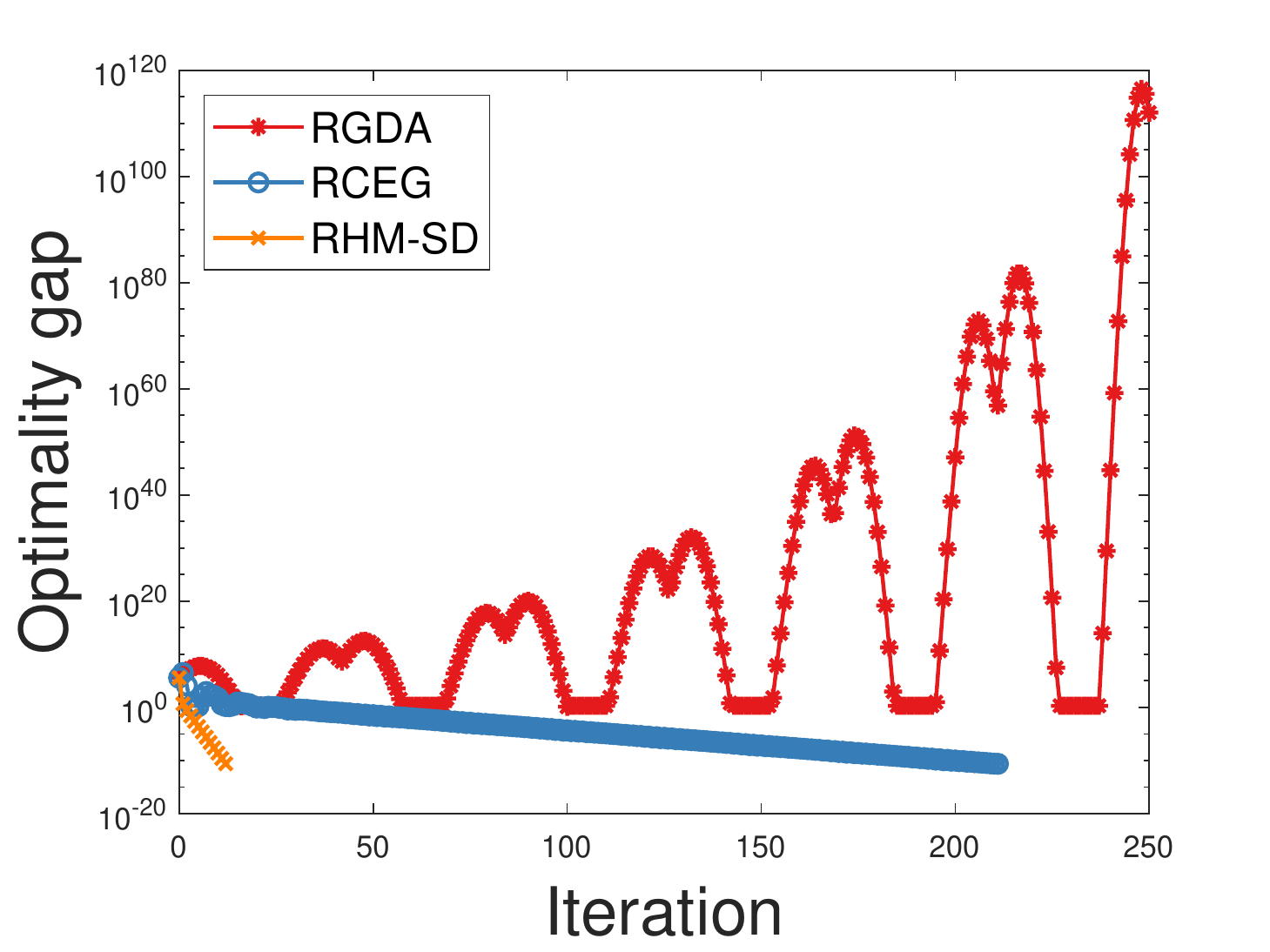}  \label{illu_iter_more}} 
\caption{{RGDA \cite{huang2020gradient} fails to converge on the geodesic bilinear problem $f(\bX, \bY) = \log\det(\bX) \log\det(\bY)$. In particular, RGDA suffers from cyclic behaviour. RCEG \cite{zhang2022minimax,jordan2022first} converges very slowly. In contrast to RGDA and RCEG, the proposed Riemannian gradient descent method on the Riemannian Hamiltonian function (RHM-SD) quickly converges on such challenging bilinear problems. Notably, the proposed RHM-SD achieves an optimality gap lower than $10^{-10}$ in just $12$ iterations while RCEG takes $220$ iterations.}}
\label{fig_illu_bilinear}
\end{figure}

Here, we give a motivating example to show how the Riemannian Hamiltonian approach achieves convergence to global saddle points. To this end, we consider the problem $f(\bX, \bY) = \log\det(\bX) \log\det(\bY)$ where $\bX, \bY \in \sS_{++}^d$, the set of $d \times d$ symmetric positive definite (SPD) matrices. When endowed with the affine-invariant metric \cite{bhatia2009positive}, i.e., $\langle \bU, \bV \rangle_\bX = \trace(\bX^{-1} \bU \bX^{-1} \bV)$ for any $\bU,\bV \in T_\bX \sS_{++}^d$, the set becomes a Riemannian manifold. Under this metric, one can show that the function is g-bilinear, i.e., geodesic linear in both $\bX, \bY$, but not g-strongly-convex-concave (Proposition \ref{g_convex_concave_prop}). However, the Riemannian Hamiltonian of the objective, i.e., $\gH(\bX, \bY) = \frac{1}{2} \big(\| \grad_\bX f(\bX, \bY) \|^2_{\bX} + \|\grad_\bY f(\bX, \bY)  \|^2_\bY \big)$ satisfies the Riemannian PL condition (Proposition \ref{PL_logdet_bilinear_quadratic}). This suggests that the vanilla Riemannian gradient descent method for minimizing the Riemannian Hamiltonian (\ref{MainRiemHamiltonian}) converges to the global saddle points of $f$.} {In Appendix \ref{geodesic_quad_bilinear_appendix}, we show that the geodesic-bilinear function $f(\bX, \bY)$ does not satisfy the min-max Riemannian PL condition on the original problem. This further justifies the merit of the proposed Hamiltonian proxy problem \eqref{MainRiemHamiltonian} of $f(\bX, \bY)$ that satisfies the Riemannian PL condition.}

{
On the other hand, the RGDA algorithm \cite{huang2020gradient} follows the negative of the min-max Riemannian gradient of $f$. Specifically, let $\M = \sS_{++}^d \times \sS_{++}^d$ and $P = (\bX, \bY) \in \M$ be the product manifold and its elements. The min-max gradient of $f$ is derived as $G(P) = \big( \grad_\bX f(\bX, \bY), -\grad_\bY f(\bX, \bY) \big) = ( \bX \log\det(\bY), - \bY \log\det(\bX) )$. We compare this expression with the gradient of the Riemannian Hamiltonian, which is $\grad \gH(P) = \big( \bX \log\det(\bX), \bY \log\det(\bY) \big)$. We observe that $\langle G(P), \grad \gH(P) \rangle_P = 0$, which implies that the min-max gradient of $f$ is always orthogonal to the gradient of its Riemannian Hamiltonian. In fact, such orthogonality holds for any g-bilinear objective (see Proposition \ref{orthogo_g_bilinear} in Appendix \ref{geodesic_quad_bilinear_appendix}). Given that the negative Hamiltonian gradient $-\grad \gH(P)$ points to the global saddle points, the orthogonality of its direction implies that RGDA provably cycles around the saddle points. For the RCEG algorithm \cite{zhang2022minimax,jordan2022first}, since it also makes use of the RGDA-style updates, we expect its slow convergence for g-bilinear problems.




We illustrate the above findings in Figure~\ref{fig_illu_bilinear}, where we compare the Riemannian Hamiltonian  steepest descent method (RHM-SD) against both RGDA \cite{huang2020gradient} and RCEG \cite{zhang2022minimax,jordan2022first} with properly tuned stepsize. The convergence is measured in optimality gap, i.e., $|\det(\bX) - 1| + |\det(\bY) - 1|$ given that the global saddle points of $f$ satisfy $\det(\bX^*) = \det(\bY^*) = 1$ (Proposition \ref{PL_logdet_bilinear_quadratic}). 
We observe that the proposed RHM-SD takes only $12$ iterations to obtain an optimality gap of lower than $10^{-10}$ for this challenging setup. 
However, RGDA experiences cyclic behaviour, which matches our analysis above. While RCEG incorporates a correction step for RGDA-style updates to address the cycling issue, it still exhibits a slight cyclic behavior during the initial phase but eventually converges (Figure~\ref{fig_illu_bilinear}(c)). Overall, RCEG takes $220$ iterations for the same optimality gap. 
From Figures~\ref{fig_illu_bilinear}(a) and~\ref{fig_illu_bilinear}(b), we also observe that per iteration runtime cost of RHM-SD is similar to RCEG. 

More details and discussions can be found in Section \ref{geodesic_quadratic_sect}, where we generalize the findings to include quadratic terms.
}

\section{Riemannian Hamiltonian consensus method}
\label{RHM_con_sect}
In the Euclidean setting, \cite{mescheder2017numerics} proposes the consensus method for solving min-max problems in the Euclidean space. The consensus method has also been viewed as a perturbation of the Euclidean Hamiltonian method \cite{abernethy2019last}. In this section, we propose an extension of RHM with steepest descent update, namely the Riemannian Hamiltonian consensus method (RHM-CON), by combining the Hamiltonian gradient direction with the min-max gradient direction. {In practice, particularly for some deep learning applications, Assumption \ref{stationary_assumption} may not satisfy. Thus solving the Hamiltonian proxy problem \eqref{MainRiemHamiltonian} may lead to undesired stationary points that are not saddle points. The consensus direction provides a regularization and is usually practically favourable for general nonconvex-nonconcave min-max problem. We show such an example in Section \ref{orthogonal_GAN_sect}.
}

The update of RHM-CON is given by
\begin{equation*}
    p_{t+1} = {\rm Exp}_{p_t} \big(- \eta_t \, \zeta(p_t) \big) =  {\rm Exp}_{p_t} \Big(- \eta_t \, \big( \gamma \, v(p_t)  +  \grad \gH(p_t) \big) \Big),
\end{equation*}
with $\gamma \geq 0$ and $v(p_t) := \big( \grad_x f(x_t,y_t), - \grad_y f(x_t, y_t) \big)$ is the min-max gradient 
When $\gamma = 0$, this reduces to RHM-SD. The RHM-CON method is formalized in Algorithm \ref{RHGDCO}. Below, we provide the convergence result for RHM-CON.

\begin{algorithm}[t]
 \caption{Riemannian Hamiltonian consensus (RHM-CON) method}
 \label{RHGDCO}
 \begin{algorithmic}[1]
  \STATE \textbf{Input:} Stepsize $\eta$ and regularization parameter $\gamma$.
  \STATE Initialize $p_0 = (x_0,y_0) \in \M$.
  \FOR{$t = 0,...,T$}
  \STATE Compute the min-max gradient $v(p_t) = \big( \grad_x f(x_t,y_t), - \grad_y f(x_t, y_t) \big)$. 
  \STATE Compute the update direction $\zeta(p_t) = \gamma \, v(p_t) +  \hess f(p_t)[\grad f(p_t)]$. 
  \STATE Update $p_{t+1} = {\rm Exp}_{p_t}\big(-\eta_t \, \zeta(p_t) \big)$.
  \ENDFOR
  \STATE \textbf{Output:} $p_T$.
 \end{algorithmic} 
\end{algorithm}

\begin{theorem}[Linear convergence of RHM-CON]
\label{linear_convergence_PL_general}
Under Assumption \ref{smoothness_assumption} with $L = L_0L_2 + L_1^2$, suppose that the Riemannian Hamiltonian $\gH$ satisfies the PL condition. Let $c > 0$ such that $\| \zeta(p_t) \|^2 = \| \gamma \, v(p_t) + \grad \gH(p_t) \|^2 \geq c \| \grad \gH(p_t) \|^2$ 
for all the iterates $p_t$. Set $\gamma < \sqrt{\delta}$, $\eta_t = \eta \leq \frac{1}{L}$, then Algorithm \ref{RHGDCO} converges with
\begin{equation*}
    \| \grad f(p_t)\|^2_{p_t} \leq \Big(1 - \nu \Big)^t \| \grad f(p_0) \|^2_{p_0},
\end{equation*}
where $\nu =  (c\delta + \delta - \gamma^2) \eta - L c \delta \eta^2 > 0$.
\end{theorem}
\begin{proof}
First, we highlight that 
\begin{equation*}
    \frac{1}{2}\|v(p) \|^2_p = \frac{1}{2} \| \grad f(p) \|^2_{p}  = \gH(p).
\end{equation*}
From the smoothness of Riemannian Hamiltonian (Proposition \ref{prop_ham_smooth}, Lemma \ref{lemma_smoothness_grad}) and the update in Algorithm \ref{RHGDCO}, we have
\begin{align*}
    &\gH(p_{t+1}) - \gH(p_t) \\
    &\leq - \eta \langle \grad \gH(p_t), \zeta(p_t)  \rangle_{p_t} +\frac{\eta^2 L}{2} \| \zeta(p_t) \|_{p_t}^2 \\
    &= -\frac{\eta}{2} \| \grad \gH(p_t) \|^2_{p_t} + \frac{\eta}{2} \| \zeta(p_t) - \grad \gH(p_t) \|^2_{p_t} - \Big( \frac{\eta}{2} - \frac{\eta^2 L}{2} \Big) \| \zeta(p_t) \|^2_{p_t} \\
    &\leq  \Big( - \frac{\eta}{2} - \frac{\eta c}{2} + \frac{\eta^2 L c}{2} \Big) \| \grad \gH(p_t) \|^2_{p_t} + \frac{\eta \gamma^2}{2} \| v(p_t) \|^2_{p_t} \\
    &\leq ( -\eta - \eta c + \eta^2 L c) \delta \gH(p_t) + \eta\gamma^2 \gH(p_t) \\
    &= \Big( L c\delta \eta^2 - c\delta \eta - \delta \eta + \eta \gamma^2 \Big) \gH(p_t),
\end{align*}
where the second inequality follows from $\eta \leq \frac{1}{L}$ (which gives $\frac{\eta}{2} - \frac{\eta^2 L}{2} \geq 0$) and the lower bound on $\| \zeta(p_t)\|^2_{p_t}$. The last inequality uses the PL condition and $\eta \leq \frac{1}{L} < \frac{1+c}{Lc}$, which ensures $-\frac{\eta}{2} - \frac{\eta c}{2} + \frac{\eta^2Lc}{2} < 0$. From the choice of $\eta$ and $\gamma$ as well as the definition of $\nu$, we have $\nu > 0$. This is because $\nu = \eta (c \delta + \delta - \gamma^2 - Lc\delta \eta) \geq \eta (\delta - \gamma^2) > 0$. Thus, $\gH(p_{t+1}) = ( 1- \nu) \gH(p_t)$ ensuring linear convergence. Applying this result recursively completes the proof. 
\end{proof}

From Theorem \ref{linear_convergence_PL_general}, we see that linear convergence is achieved provided that the weight $\gamma$ on min-max gradient direction is sufficiently small. Also, we highlight that a uniform parameter $c > 0$ always exists in a compact set as long as $\gamma v(p_t) \neq - \grad \gH(p_t)$. This can be ensured by choosing a small value for $\gamma$.

\section{Stochastic min-max optimization}
\label{stochastic_hm_sect}
Applications such as domain generalization, robust training, and generative adversarial networks yield a min-max problem with a stochastic function $f$, e.g., with a finite sum structure of the function \cite{loizou2020stochastic}. Under the stochastic setting, the objective function in \eqref{main_riem_minmax} can be expressed as an expectation, i.e.,
\begin{equation*}
    \min_{x \in \M_x} \max_{y \in \M_y} \Big\{f(x,y) = \sE_\omega[f (x,y; \omega)] \Big\},
\end{equation*}
where $\omega \in \gD$ is a random variable following a certain distribution $\gD$. This implies an expectation structure on the Riemannian Hamiltonian as
\begin{equation*}
    \gH(p) = \frac{1}{2} \Big\| \sE_\omega \big[\grad f(p ; \omega) \big] \Big\|^2_p = \frac{1}{2} \sE_{\omega} \sE_{\varphi} \langle \grad f (p; \omega), \grad f (p; \varphi) \rangle_p,
\end{equation*}
for $\omega, \varphi \in \gD$. Modifying Proposition \ref{gradient_hamiltonian_prop} for the stochastic setting leads to 
\begin{align*}
    \grad \gH(p) &= \frac{1}{2}
    \sE_{\omega,\varphi} \Big[  \hess f(p; \omega)[\grad f(p; \varphi)] + \hess f(p; \varphi)[\grad f(p; \omega)] \Big].
\end{align*}
Let $\grad \gH_{\omega,\varphi}(p) := \frac{1}{2} \hess f(p; \omega)[\grad f(p; \varphi)] + \frac{1}{2} \hess f(p; \varphi)[\grad f(p; \omega)]$. We can modify RHM-SD by replacing the gradient of Hamiltonian with its stochastic version (which we call RHM-SGD) as 
\begin{align}
    \grad \gH_{\gS, \gS'}(p_t) &:= \frac{1}{|\gS||\gS'|} \sum_{\omega \in \gS, \varphi \in \gS'} \grad \gH_{\omega,\varphi} (p), \label{stochastic_HG_formula}
\end{align}
where $\gS = \{ \omega_1, ...,\omega_{|\gS|} \}, \gS' = \{ \varphi_1, ..., \varphi_{|\gS'|}\}$ are randomly selected subsets with $\omega_i, \varphi_j \in \gD$. The stochastic Hamiltonian gradient provides an unbiased estimate of the full gradient, i.e., $\sE_{\gS, \gS'} [\grad \gH_{\gS, \gS'}(p)] = \grad \gH (p)$. We now show the convergence result of RHM-SGD.

\begin{theorem}[Convergence of RHM-SGD with fixed and decaying stepsize]
\label{RSHGD_fixed_var}
Let Assumption \ref{smoothness_assumption} hold with $L = L_0 L_2 + L_1^2$, and let the Riemannian Hamiltonian $\gH$ satisfy the PL condition with parameter $\delta$. Assume also that the variance of the stochastic gradient is bounded, i.e., $\sE_{\omega, \varphi} \| \grad \gH_{\omega, \varphi} (p_t) \|^2_{p_t} \leq G$. 
Then, RHM-SGD with fixed stepsize $\eta_t = \eta < \frac{1}{2\delta}$ converges with 
$\sE\| \grad f(p_t)\|^2_{p_t} \leq (1- 2\eta\delta)^t \sE \| \grad \gH (p_0) \|^2_{p_0} + \frac{\eta L G}{4}.$ Also, RHM-SGD with decaying stepsize $\eta_t = \frac{2t+1}{2\delta (t+1)^2}$, converges with $\sE\| \grad f(p_t)\|^2_{p_t} \leq \frac{L G}{2\delta^2 t}$.
\end{theorem}
\begin{proof}
The proof follows from \cite[Theorem~4]{karimi2016linear} and can be easily adapted to the Riemannian manifold setting, and therefore, is omitted. 
\end{proof}

We can similarly consider the stochastic version of RHM-CON, which we denote as RHM-SCON, with the update step as
\begin{equation*}
    \zeta_{\gS, \gS'}(p_t) = \gamma (v_{\gS}(p_t) + v_{\gS'}(p_t))/2 + \grad \gH_{\gS, \gS'}(p_t),
\end{equation*}
where $v_{\gS}(p_t)$ is the stochastic min-max gradient on sample set $\gS$. Theorem \ref{RSHGD_fixed_var} can be adapted to prove the convergence of RHM-SCON following similar assumptions and analysis in Theorem \ref{linear_convergence_PL_general}.


\section{Convergence under retraction}
\label{retraction_sect}
Existing algorithms for solving \eqref{main_riem_minmax}, such as RCEG \cite{zhang2022minimax}, employs the exponential map to update iterates on the manifolds. However, in many cases, the computational cost of implementing the exponential map for many Riemannian manifolds is prohibitive. An alternative is to consider the more general retraction operation \cite[Chapter~4]{absil2009optimization}. In this section, we show that the use of retraction (instead of the exponential map) in RHM algorithms guarantees similar convergence under an additional mild assumption.

Retraction $R_p: T_p\M \xrightarrow{}\M$ is a map that satisfies for all $p \in \M$, (1) $R_p(0) = p$ (2) $\D R_p (0)[\xi] = \xi$ for all $\xi \in T_p\M$. From the definition, we observe that the exponential map is a special case of retraction. In practice, when an efficient retraction is available, the Hamiltonian gradient update can be performed via retraction, i.e., $p_{t+1} = R_{p_t}(- \eta \, \grad \gH(p_t))$. To analyze the convergence, we make the following additional assumption that bound the differential operator of the retraction map. 

\begin{assumption}
\label{bounded_diffretr_assumption}
There exists constants $\theta_1, \theta_2 > 0$ such that the retraction curve $c(t) := R_p (t \xi)$ with $\| \xi\|_p = 1$ satisfies $\| c'(t) \|_{c(t)}  \leq \theta_1$ and $\| c''(t) \|_{c(t)} \leq \theta_2$ for all $t$ where $c(t) \in \gU$, where $\gU$ is a compact subset of $\M$.
\end{assumption}

This assumption is always satisfied for a compact manifold $\gM$. The compactness appears to be necessary for retraction-based analysis for first-order algorithms \cite{han2021improved,sato2019riemannian,kasai2018riemannian,boumal2020introduction}. We remark that for the case of the exponential map, the retraction curve coincides with the geodesic curve. Then, $\theta_1 = 1$ because $\| c'(t) \|_{c(t)} = \| \Gamma_p^{c(t)} \xi \|_{c(t)} = 1$ by isometric property of parallel transport. Also, $\theta_2 = 0$ from the definition of the geodesic.

\begin{proposition}
\label{prop_smoothness_retraction}
Under Assumptions~\ref{smoothness_assumption} and \ref{bounded_diffretr_assumption}, the Riemannian Hamiltonian $\gH$ is retraction $L_R$-smooth with $L_R = \theta_1^2 L  + \theta_2 L_1 L_0$, i.e., for any $p \in \M$, $q = R_p(\xi) \in \gU$, we have $\gH(q) \leq \gH(p) + \langle \grad\gH(p), \xi \rangle_p + \frac{L_R}{2} \| \xi \|^2_p$.
\end{proposition}
\begin{proof}
For any retraction curve $c(t) = R_p(t\xi)$ with $\| \xi \|_p = 1$ and $t \geq 0$ such that $c(t) \in \gU$, we obtain
\begin{align}
    \frac{d^2}{dt^2} \gH(c(t)) &= \langle \hess \gH(c(t))[c'(t)] , c'(t)  \rangle_{c(t)} + \langle \grad\gH(c(t)), c''(t) \rangle_{c(t)} \nonumber\\
    &\leq L \theta_1^2 + \theta_2 \| \hess f(c(t)) [\grad f(c(t))] \|_{c(t)} \nonumber\\
    &\leq L \theta_1^2 + \theta_2 L_1 L_0 = L_R, \label{smooth_retr_eq1}
\end{align}
where the second inequality applies the gradient of Hamiltonian is $L$-Lipschitz (Proposition \ref{prop_ham_smooth}, Lemma \ref{lemma_smoothness_grad}) and Assumption \ref{bounded_diffretr_assumption}. The last inequality follows from Assumption \ref{smoothness_assumption}. The proof from \eqref{smooth_retr_eq1} to $L_R$-smoothness of $\gH$ is due to \cite[Lemma~3.2]{huang2015riemannian}, which we include here for completeness. 

For any $\xi \in T_p\M$ such that $R_p(\xi) \in \gU$, let $\alpha = \| \xi\|_p$, $\zeta = \xi/\| \xi\|_p$ and hence $\xi = \alpha \zeta$ with $\| \zeta\|_p =1$. Applying Taylor's Theorem on $\gH \circ R_p$ gives
\begin{align*}
    \gH(R_p(\xi)) - \gH(p) &= \gH(R_p(\alpha \zeta)) - \gH(R_p(0)) \\
    &= \alpha \frac{d}{dt} \gH(R_p(t \zeta)) \Big \vert_{t = 0} + \frac{\alpha^2}{2} \frac{d^2}{dt^2} \gH (R_p(t\zeta)) \Big\vert_{t = \tilde{t}} \\
    &\leq \alpha \langle \grad  \gH(p), \zeta \rangle_{p} + \frac{\alpha^2 L_R}{2} \\
    &= \langle \grad \gH(p), \xi\rangle + \frac{L_R}{2} \| \xi \|^2_p,
\end{align*}
where $\tilde{t} \in [0, \alpha]$. Thus, the proof is complete.
\end{proof}

Using Proposition \ref{prop_smoothness_retraction}, we show below that RHM-SD attains a linear convergence rate with retraction. 

\begin{theorem}[Linear convergence of RHM-SD under retraction]
Under same settings as in Theorem \ref{theorem_convergence_RHGD}, suppose Assumption \ref{bounded_diffretr_assumption} holds, and the iterates stay in the compact set $\gU$. Then, RHM-SD with retraction and $\eta = 1/L_R$ converges with $\| \grad f(p_t)\|^2_{p_t} \leq (1 - \frac{\delta}{L_R})^t \| \grad f(p_0) \|^2_{p_0}$.
\end{theorem}

The proof is similar to the proof of Theorem \ref{theorem_convergence_RHGD} and is omitted. A
similar analysis with the retraction operation can be performed for other variants of RHM including RHM-CON, RHM-SGD, and RHM-SCON.

\section{Experiments}
\label{application_sect}
In this section, we discuss empirical performance of the proposed Riemannian Hamiltonian methods for various min-max optimization problems on manifolds. The algorithms are implemented in Matlab using the Manopt package \cite{boumal2014manopt} except for Section \ref{robust_training_sect}, \ref{orthogonal_GAN_sect} where we use Pytorch with the Geoopt package \cite{kochurov2020geoopt}. We highlight that there exist many other manifold optimization packages, such as ROPTLIB \cite{huang2018roptlib}, Manopt.jl \cite{bergmann2022manopt}, Pymanopt \cite{townsend2016pymanopt}, McTorch \cite{meghwanshi2018mctorch}, and RiemOpt \cite{smirnov2021tensorflow}, where RHM can also be implemented efficiently. 
We use the following acronyms for the various RHM algorithms considered in this section.
\begin{itemize}
    \item RHM-SD-F: RHM with steepest descent direction with fixed stepsize.
    \item RHM-SD: RHM with steepest descent direction with backtracking line search.
    \item RHM-CON: RHM consensus method with fixed stepsize (Section \ref{RHM_con_sect}).
    \item RHM-CG: RHM with the conjugate gradient method.
    \item RHM-TR: RHM with the trust-region method where we use Hessian approximation with finite differentiation \cite{boumal2015riemannian}. 
    \item RHM-SGD: RHM with stochastic gradient (Section \ref{stochastic_hm_sect}).
    \item RHM-SCON: RHM with stochastic consensus method (Section \ref{stochastic_hm_sect}).
\end{itemize}

We compare the proposed Riemannian Hamiltonian methods with the Riemannian gradient descent ascent (RGDA) \cite{huang2020gradient} and the Riemannian corrected extra-gradient (RCEG) \cite{zhang2022minimax}. As discussed previously, RGDA has not been studied and analyzed for solving the general min-max problem (\ref{main_riem_minmax}), but when $\M_y$ is a convex subset of the Euclidean space \cite{huang2020gradient}. In our experiments, however, we extend RGDA to solve (\ref{main_riem_minmax}).

For all the experiments, we implement the algorithms with exponential map for comparability with RCEG, except for the applications of subspace robust Wasserstein distance (Section \ref{prwd_sect}), robust training (Section \ref{robust_training_sect}) and generative adversarial networks (Section \ref{orthogonal_GAN_sect}) where we implement with retraction map because the manifolds considered do not have a well-defined logarithm map. Hence, for these applications, RCEG is excluded for comparison. In robust training and generative adversarial network experiments, we also test stochastic algorithms for RGDA and RHM. The codes are available at \url{https://github.com/andyjm3}.

\subsection{Geodesic quadratic bilinear optimization}
\label{geodesic_quadratic_sect}
The first example we consider is 
\begin{equation}
    f(\bX, \bY) = c_q (\log\det(\bX))^2 + c_l \log\det(\bX) \log\det(\bY) - c_q (\log\det(\bY))^2, \label{g_bilinear_quadratic}
\end{equation}
where $\bX, \bY \in \sS_{++}^d$, the set of $d \times d$ symmetric positive definite (SPD) matrices. The weights $c_l, c_q \geq 0$ control the balance between the linear and quadratic terms. 

For $\bX \in \sS_{++}^d$, the tangent space $T_\bX \sS_{++}^d$ is the set of symmetric matrices. When endowed with the affine-invariant (AI) metric, i.e., $\langle \bU, \bV \rangle_\bX = \trace(\bX^{-1} \bU \bX^{-1} \bV)$, for any $\bU, \bV \in T_\bX \sS_{++}^d$, one can derive the geodesic, exponential map, and other Riemannian optimization ingredients \cite{han2021riemannian,bhatia2009positive,pennec2020manifold}. We include the expressions in Appendix \ref{app:sec:matrix_manifolds}. Here, we use $\M_{\rm SPD}$ to represent the SPD manifold with the AI metric. It is worth noticing that the function (\ref{g_bilinear_quadratic}) is nonconvex-nonconcave in the Euclidean space (with details included in Appendix \ref{geodesic_quad_bilinear_appendix}).

However, the log-det function is geodesic linear on SPD manifold with the AI metric \cite{sra2015conic} and we show in the following proposition that $f(\bX, \bY)$ is g-convex-concave, although not necessarily g-strongly-convex-concave. 

\begin{proposition}
\label{g_convex_concave_prop}
The function \eqref{g_bilinear_quadratic} is g-convex-concave on $\M_{\rm SPD}$ but not g-strongly-convex-concave.
\end{proposition}

We next prove that the Riemannian Hamiltonian $\gH$ of the objective \eqref{g_bilinear_quadratic} satisfies the PL condition, which allows linear convergence of the proposed RHM algorithms. 
\begin{proposition}
\label{PL_logdet_bilinear_quadratic}
The Riemannian Hamiltonian of \eqref{g_bilinear_quadratic} satisfies the PL condition with $\delta=(4c_q^2 + c_l^2)d^2$. A point $(\bX^*, \bY^*)$ is a global saddle point of \eqref{g_bilinear_quadratic} if and only if it satisfies $\det(\bX^*) = \det(\bY^*) = 1$.
\end{proposition}
In Proposition \ref{PL_logdet_bilinear_quadratic}, we see that there exist a continuum of global saddle points. Consequently, we define an optimality gap criterion as $|\det(\bX) - 1| + |\det(\bY) - 1|$ for a candidate point $(\bX, \bY)$.

\subsection*{Experiment settings and results}
We consider $d = 30$ and discuss results on various combinations of $c_q, c_l$. We compare our RHM with RGDA~\cite{huang2020gradient} and RCEG~\cite{zhang2022minimax}. 
All the choices of stepsize are tunned to reflect the best performance except for RHM-SD, RHM-CG, RHM-TR where the stepsizes are selected adaptively by the algorithms. For RHM-CON, we set $\gamma = 0.5$. Convergence of an algorithm is measured in terms of $\| \grad f(p_t) \|_{p_t}$, which is equivalent to $\sqrt{2 \gH(p_t)}$. This measure of convergence has also been considered in \cite{zhang2022minimax} for min-max problems on manifolds. Algorithms are stopped either when gradient norm falls below $10^{-10}$ or the max iteration has been reached. Results are reported in Fig.~\ref{geodesic_quadratic_plot}. 

From Fig.~\ref{geodesic_quadratic_plot}, we observe rapid convergence of RHM algorithms in all the settings. 
The convergence for RGDA varies across different choices of $c_q, c_l$ where it converges faster when the weight on the quadratic term ($c_q$) is relatively higher and is not able to converge when $c_l$ increases. We also observe convergence for RCEG in all cases but the rate is slower compared to RHM algorithms. In Fig. \ref{disttoopt_g_quadratic}, we further compare the optimality gap where we observe all the proposed RHM algorithms reach below $10^{-10}$ at a faster rate than the baselines. The slopes of RHM-SD-F and RHM-CON are steeper than that of RCEG (indicating better theoretical rates for RHM). Additional results on optimality gap comparisons are in Fig. \ref{geodesic_quad_linear_disttoopt} in Appendix \ref{app:sec:experiments}. Finally, Fig.~\ref{runtime_g_quadratic} shows the runtime performance of various algorithms, with the markers indicating the progress of respective algorithms per iteration. We observe that the per-iteration computational cost of RHM is higher than RGDA. This is because RHM exploits second-order information of $f$ to compute the gradient of $\gH$. Also, we see that RCEG can be costly because it requires evaluation of the exponential map twice and the logarithm map once per iteration.

\begin{figure*}[!t]
    \centering
    \subfloat[\texttt{$c_q = 0, c_l = 1$} ]{\includegraphics[scale=0.29]{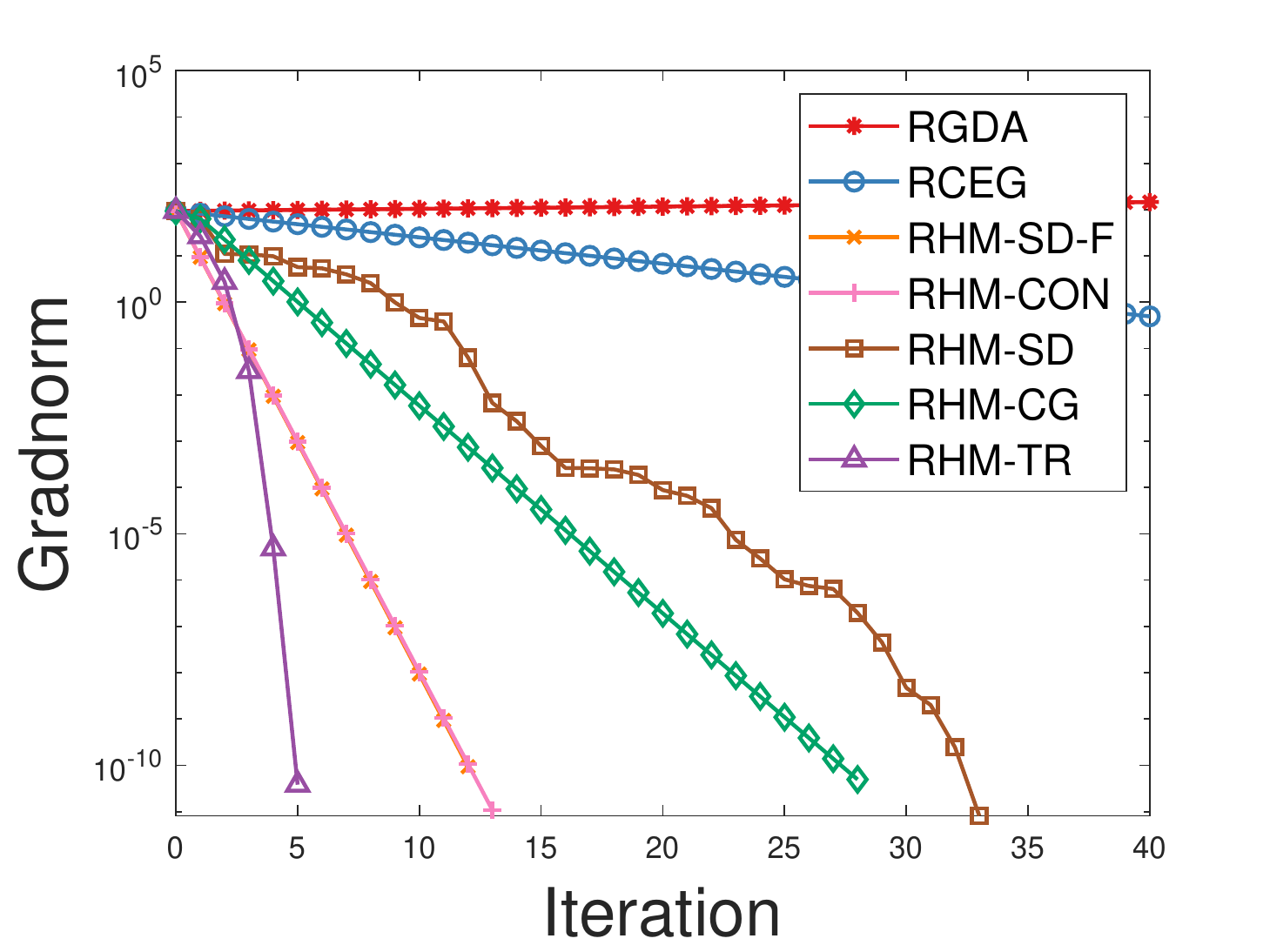}  \label{bilinear_plot}}
    \subfloat[\texttt{$c_q = 1, c_l = 0$} ]{\includegraphics[scale=0.29]{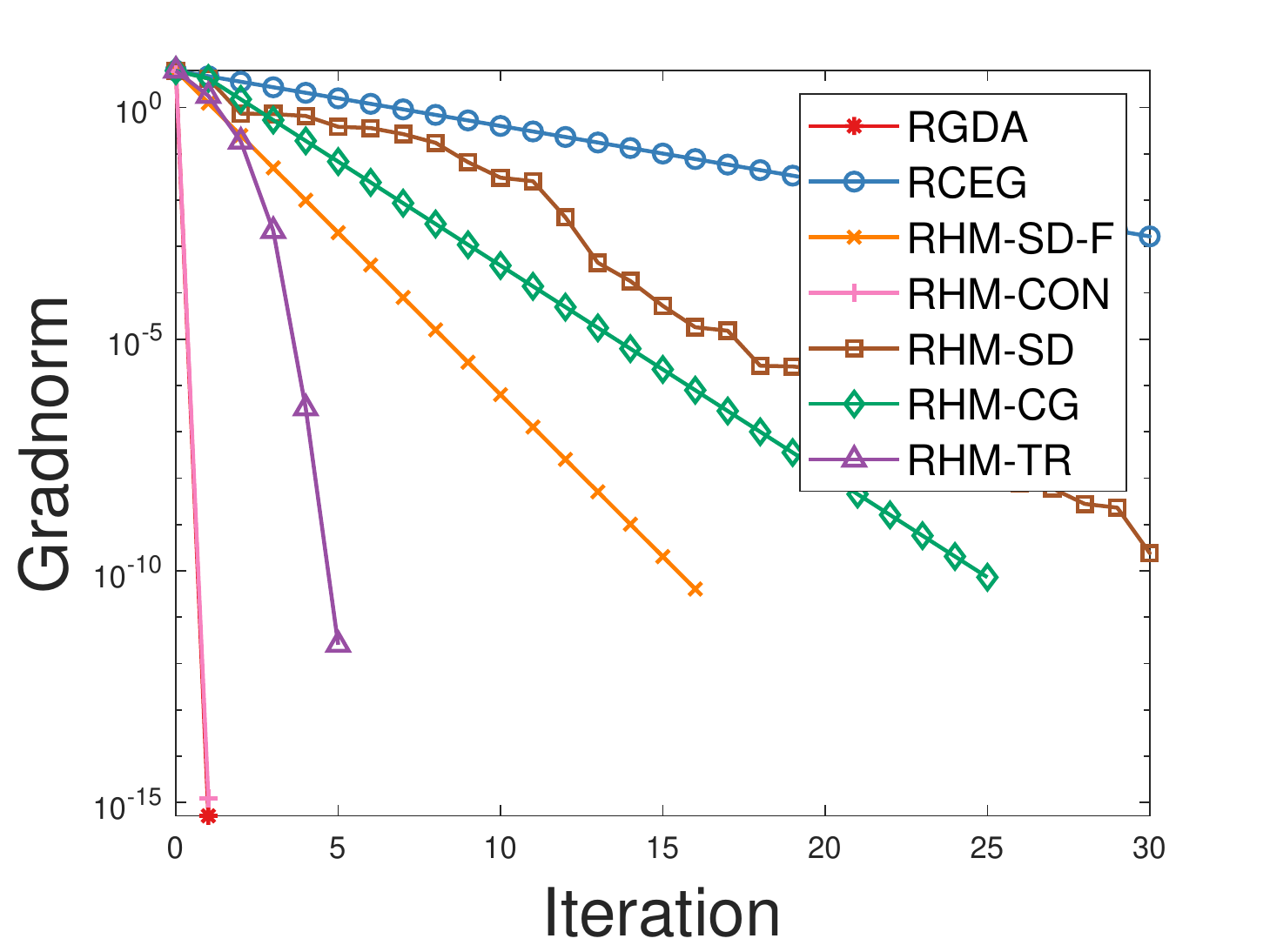}}
    \subfloat[\texttt{$c_q = 1, c_l = 0.1$} ]{\includegraphics[scale=0.29]{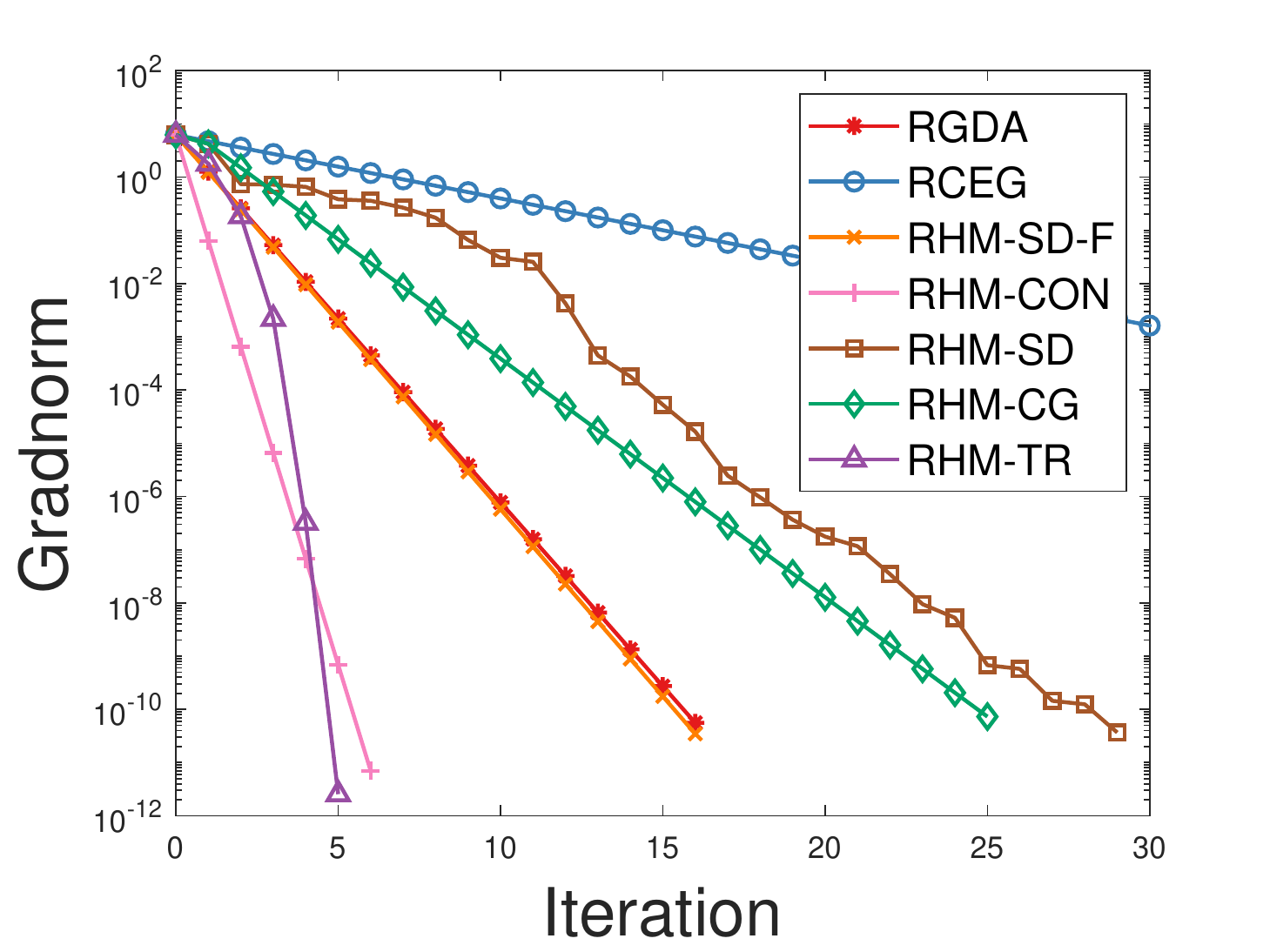}}\\
    \subfloat[\texttt{$c_q = 1, c_l = 1$} ]{\includegraphics[scale=0.29]{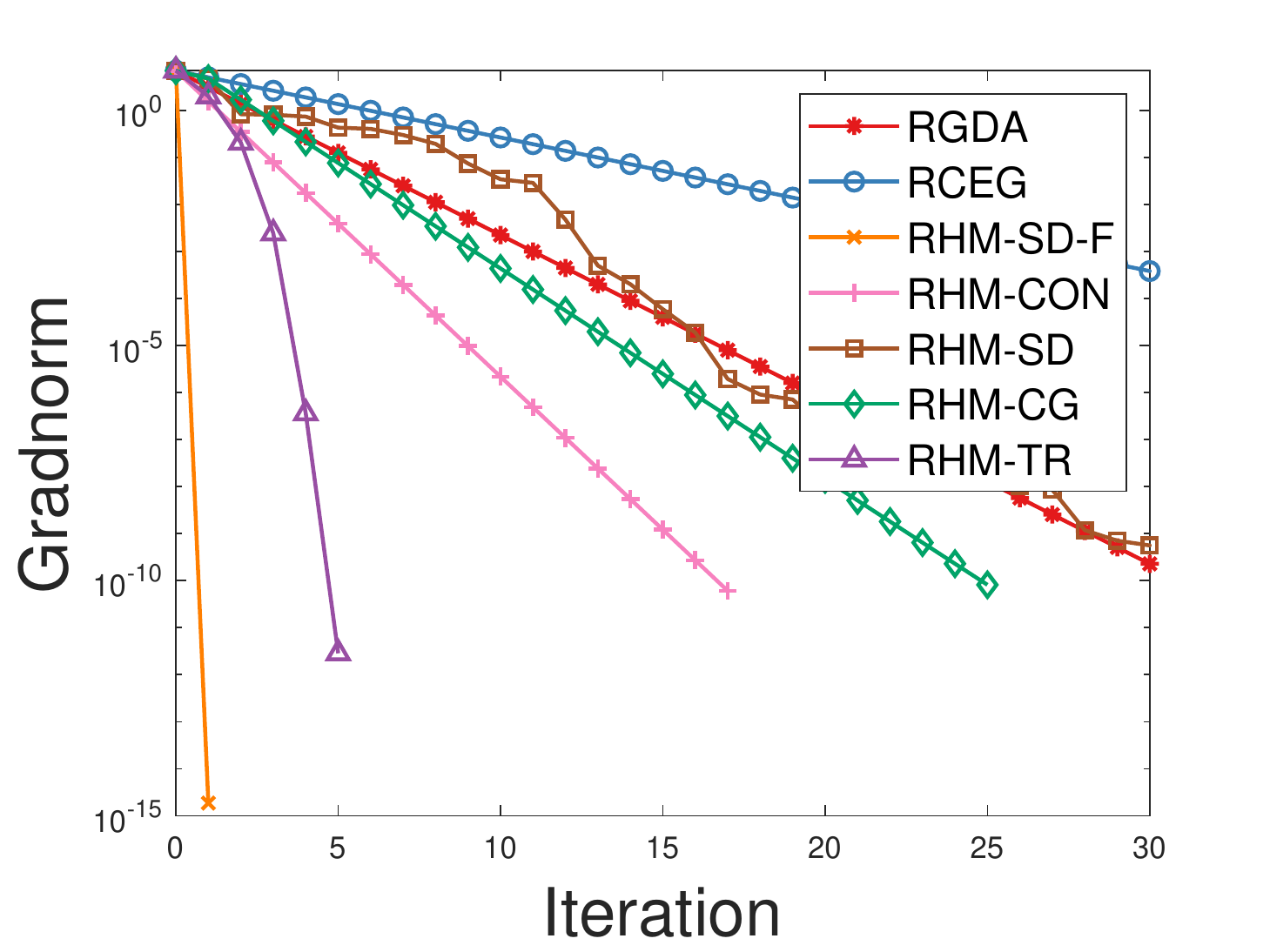}}
    \subfloat[\texttt{$c_q = 1, c_l = 10$} ]{\includegraphics[scale=0.29]{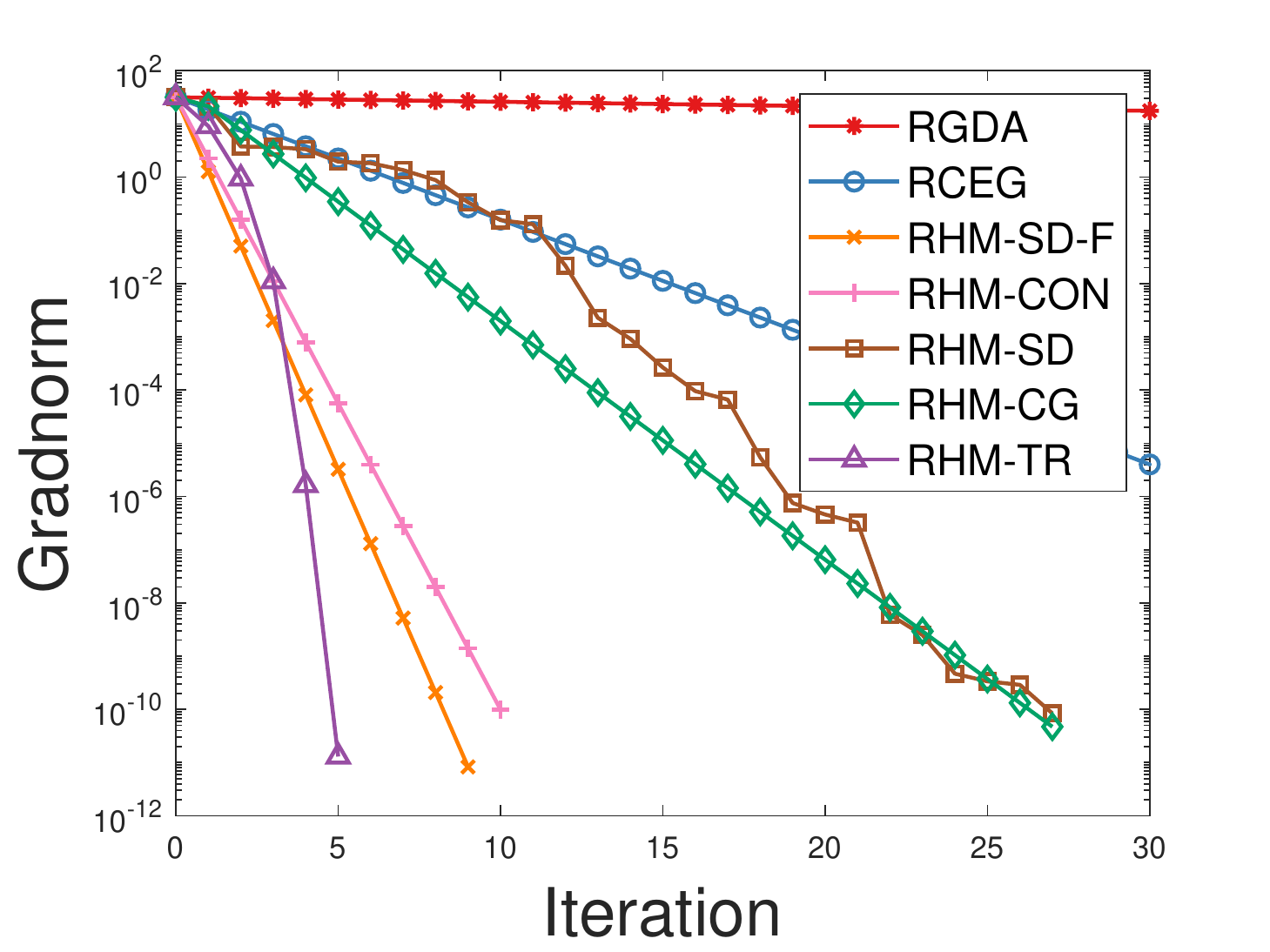}}
    \subfloat[\texttt{$c_q = 1, c_l = 10$ (opt gap)} ]{\includegraphics[scale=0.29]{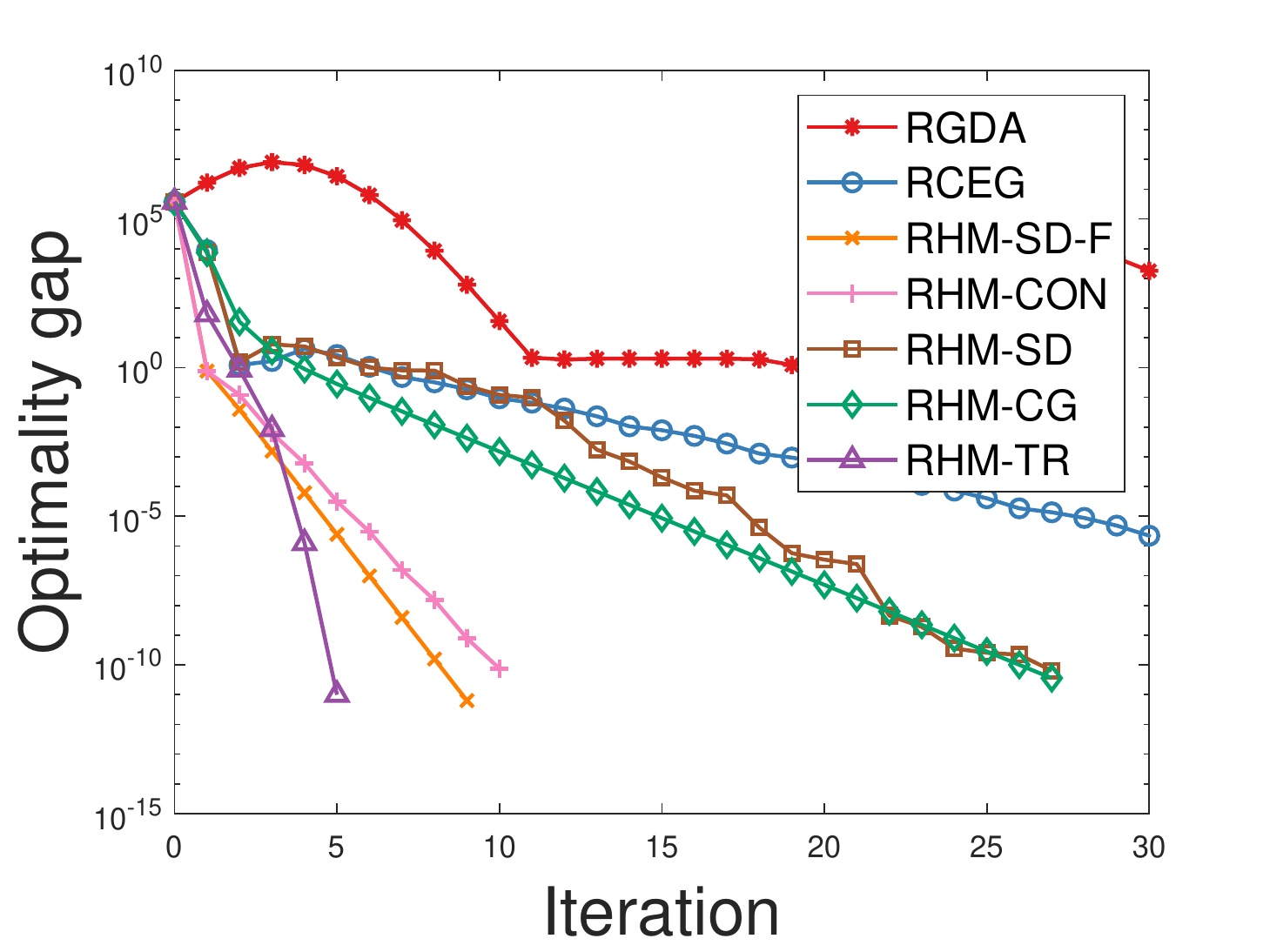} \label{disttoopt_g_quadratic}}\\
    \subfloat[\texttt{$c_q = 1, c_l = 10$ (time)} ]{\includegraphics[scale=0.29]{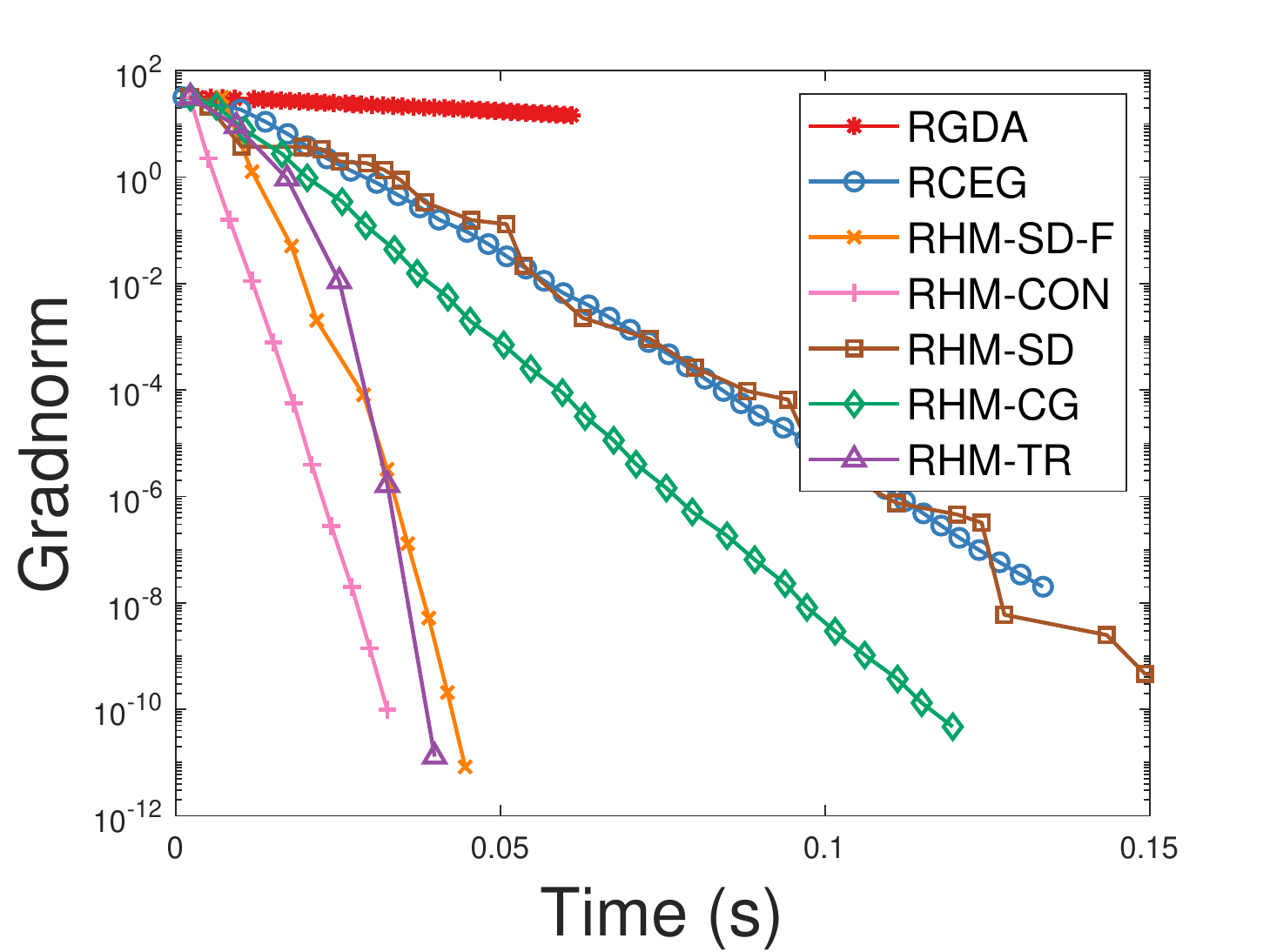} \label{runtime_g_quadratic}}
    \caption{Experiments on the geodesic quadratic bilinear problem \eqref{g_bilinear_quadratic} with $d = 30$, under varying weights $c_q, c_l$. We observe that our RHM algorithms converge quickly in all settings while baselines such as RGDA \cite{huang2020gradient} and RCEG \cite{zhang2022minimax}. The performance of RGDA varies greatly with the settings where it converges only for a few settings and for the others RGDA fails to converge. RCEG presents a relatively more stable convergence behavior than RGDA but with a rate that is slower than our proposed RHM algorithms.}
    \label{geodesic_quadratic_plot}
\end{figure*}

\subsection{Robust geometry-aware PCA}
Geometry-aware principal component analysis (PCA) on $\M_{\rm SPD}$ \cite{horev2016geometry} concerns dimensionality reduction for SPD matrices while preserving geometric structures on the manifold. The robust PCA (or robust Fr\'echet mean) on SPD manifolds has been considered in \cite{zhang2022minimax}. For a set of SPD matrices $\bM_i \in \sS_{++}^d$, $i = 1,...,n$, the aim is to find the Fr\'echet mean $\bM \in \sS_{++}^d$ that is bounded away from zero, i.e.,
\begin{equation}
    \min_{\bM \in \M_{\rm SPD}} \max_{\bx \in \gS^d} \, \bx^\top \bM \bx + \frac{\alpha}{n} \sum_{i=1}^n {\rm dist}^2(\bM, \bM_i), \label{rgpca_def}
\end{equation}
where $\alpha > 0$ and $\gS^{d-1}: = \{ \bx\in \sR^d: \| \bx \|_2 = 1 \}$ denotes the sphere manifold and $ {\rm dist}: \sS_{++}^d \times \sS_{++}^d$ is the Riemannian distance on $\M_{\rm SPD}$. 

We first note that the function in (\ref{rgpca_def}) is geodesic strongly convex in $\bM$ and geodesic nonconcave in $\bx$. Also, it is difficult to verify the Riemannian PL condition on the Hamiltonian of (\ref{rgpca_def}). Hence, this is a challenging problem instance as it does not fall into the studied settings of the existing works \cite{huang2020gradient,zhang2022minimax} including ours. 

\subsection*{Experiment settings and results}
For this problem, we follow the same settings as discussed in \cite{zhang2022minimax} for generating the SPD matrices $\bM_i$ with the eigenvalues bounded in $[\mu_0, \mu_1]$. Following \cite{zhang2022minimax}, we choose $d = 50$, $n = 40$, $\mu_0 = 0.2$, and $\mu_1 = 4.5$. The convergence results are presented in Figs.~\ref{rgpca_01} and \ref{rgpca_3}, where we only include RHM-SD-F, RHM-CON ($\gamma = 0.5$), and RHM-CG for clarity (RHM-TR performs similar to RHM-CG). We observe that although RGDA and RCEG converge faster than RHM when $\alpha = 3$, they fail to converge when $\alpha = 0.1$. The latter finding is not surprising as both RGDA and RCEG seem to perform poorly on approximately bilinear problems (as also observed in Section \ref{geodesic_quadratic_sect}). In contrast, we observe that RHM algorithms converge in both the settings, which is also validated by our analysis in Section \ref{sec:problem_classes}. It is known that the conjugate gradient based methods outperforms steepest descent methods on more challenging optimization problems. This explains the faster convergence of RHM-CG over RHM-SD-F and RHM-CON. Overall, the results in Fig. \ref{rgpca_tracelog_plots} show the benefit of the Riemannian Hamiltonian modeling in non standard settings.

\begin{figure*}[!t]
    \centering
    \subfloat[\texttt{\texttt{RGPCA} ($\alpha = 0.1$)} ]{\includegraphics[scale=0.28]{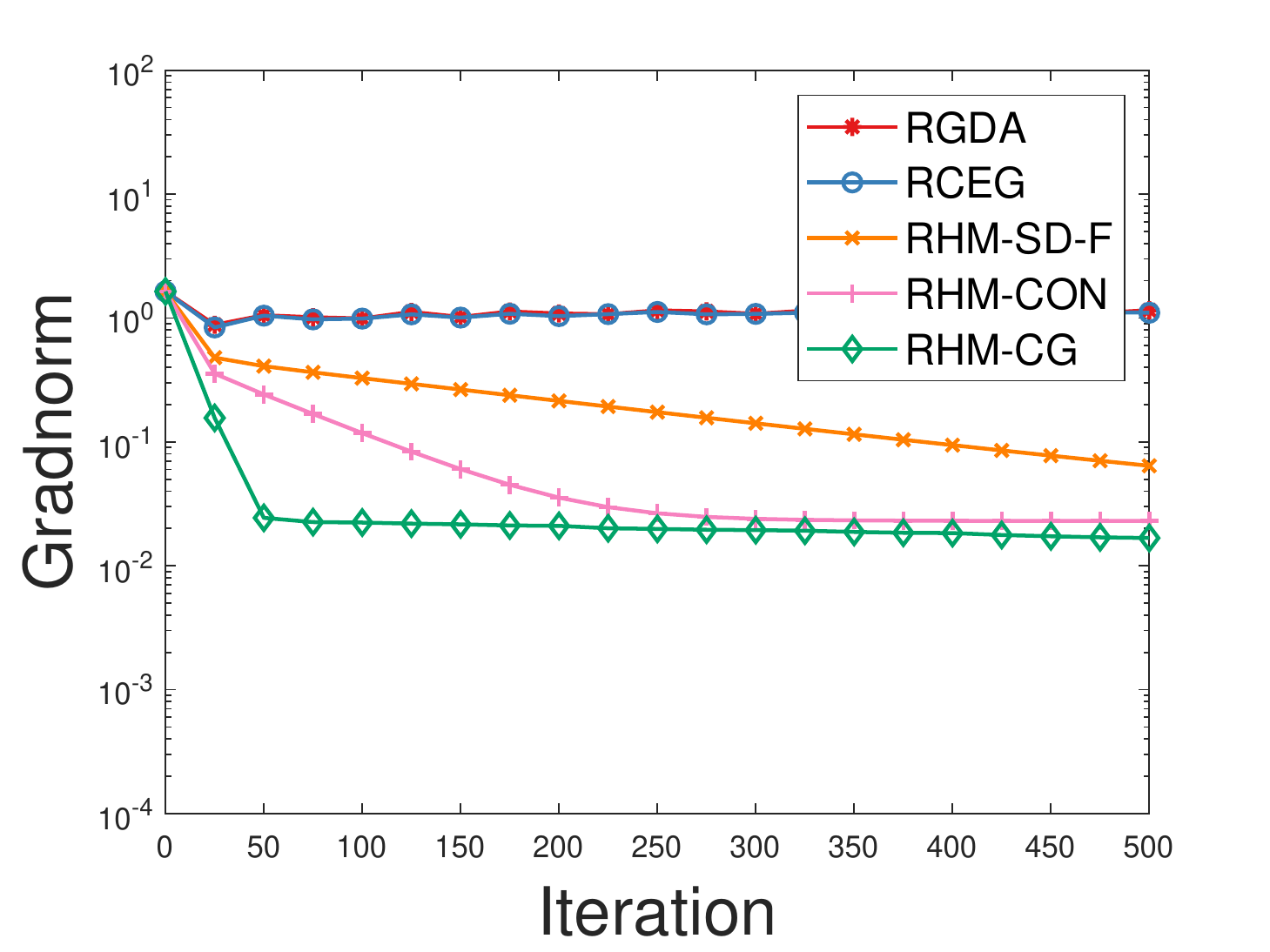}  \label{rgpca_01}}
    \subfloat[\texttt{\texttt{RGPCA} ($\alpha = 3$)} ]{\includegraphics[scale=0.28]{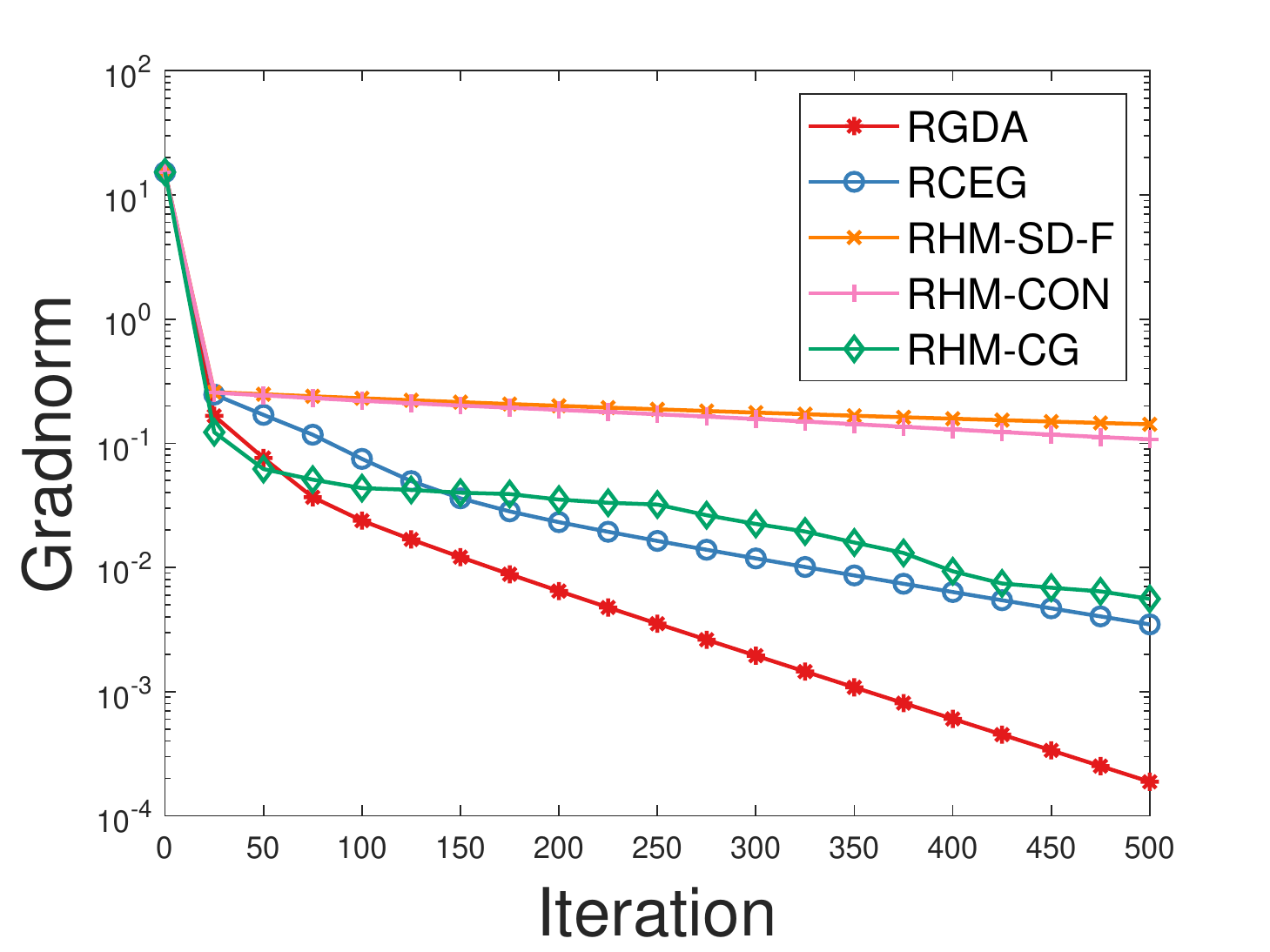}  \label{rgpca_3}} 
    \subfloat[\texttt{SRWD}]{\includegraphics[scale=0.28]{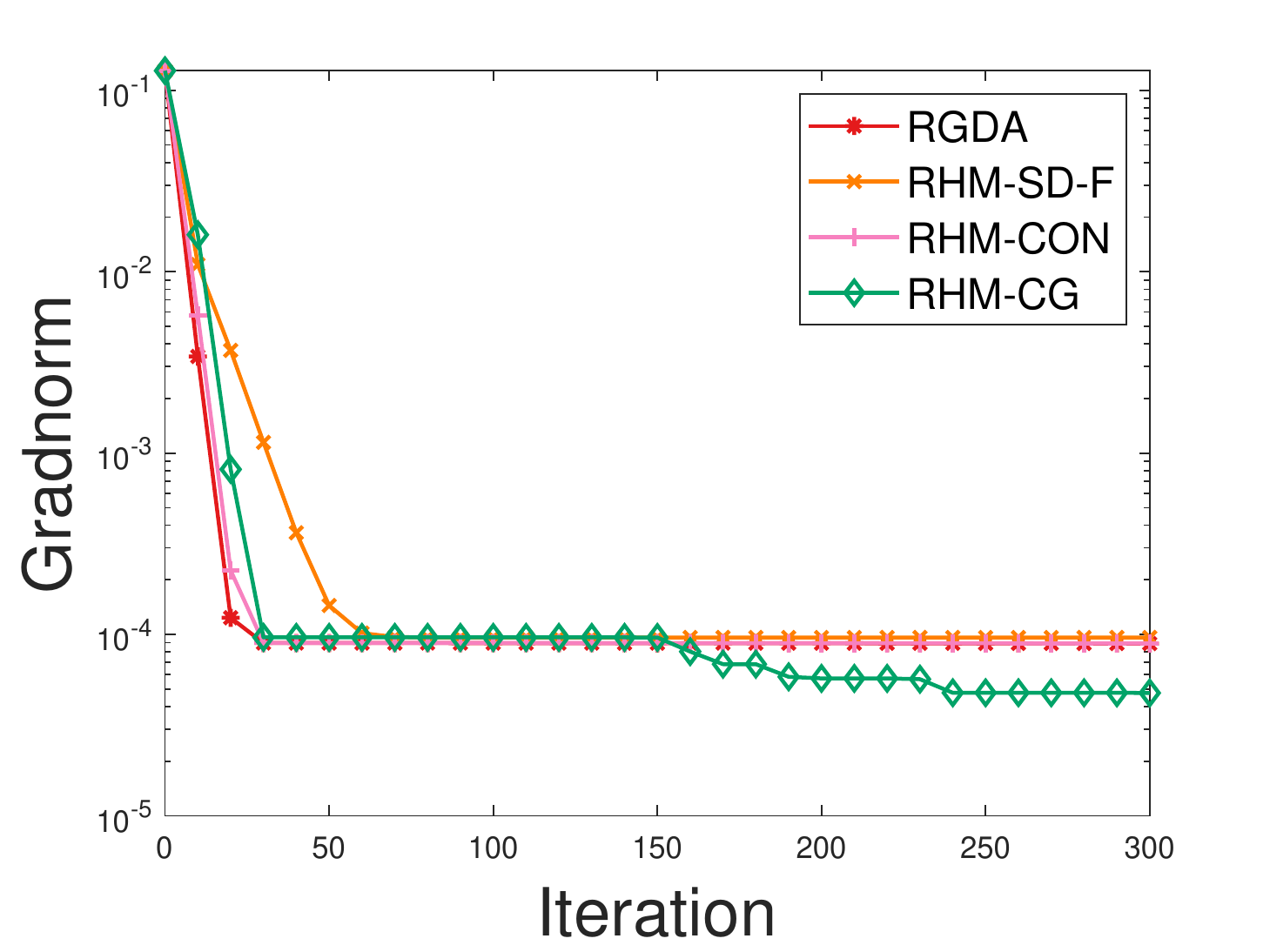}  \label{prwd_plot}}
    \caption{Convergence on the robust geometry-aware PCA (\texttt{RGPCA}) problem \eqref{rgpca_def} with $d = 50, n = 40, \mu = 0.2, L = 4.5$, and subspace robust Wasserstein distance (\texttt{SRWD}) problem on the example of fragmented hypercube \cite{lin2020projection}. We observe that the baselines RGDA and RCEG fail to converge for $\alpha = 0.1$ (approximately bilinear setting), whereas the proposed RHM algorithms show convergence for $\alpha = 0.1$ and $\alpha = 3$. } 
    \label{rgpca_tracelog_plots}
\end{figure*}

\subsection{Subspace robust Wasserstein distance}
\label{prwd_sect}
We next consider the problem of learning subspace robust Wasserstein distance \cite{paty2019subspace,lin2020projection,huang2021riemannian}, where the aim is to compute the Wasserstein distance over the worst-case optimal transport cost on a low-dimensional space. Given two discrete measures on $\sR^d$, $\mu = \sum_{i = 1}^m a_i \delta_{\bx_i}, \nu = \sum_{j=1}^n b_j \delta_{\by_j}$ where $\delta_\bx$ is the Dirac at location $\bx$. The weights $a_i, b_j$ belong to the probability simplex, i.e., $\sum_i a_i = \sum_j b_j = 1$. The objective (with entropy regularization) is then given as 
\begin{equation}
    \min_{\bGamma \in \Pi(\mu, \nu) } \max_{\bU: \bU \in {\rm St}(d, r)} \, \sum_{i,j} \Big( \Gamma_{i,j} \| \bU^\top \bx_i - \bU^\top \by_j \|_2^2 + \epsilon \, \pi_{i,j} \big( \log(\pi_{i,j}) - 1 \big) \Big), \label{prwd_formula}
\end{equation}
where ${\rm St}(d, r) := \{ \bU \in \sR^{d \times r} : \bU^\top \bU = \bI \}$ is the set of column orthonormal matrices ($d \geq r$), known as the Stiefel manifold. $\Pi(\mu, \nu) := \{ \Gamma \in \sR^{m \times n} : \Gamma_{i,j} > 0, \sum_i \Gamma_{i,j} = b_j, \sum_j \Gamma_{i,j} = a_i, \forall\, i,j \}$ is the set of couplings, which forms the so-called doubly stochastic manifold (or coupling manifold) \cite{douik2019manifold,shi2021coupling,mishra2021manifold}. 

\subsection*{Experiment settings and results}
We follow the same experiment settings as in \cite{lin2020projection,huang2021riemannian} and consider a uniform distribution over hypercube $[-1, 1]^d$ and a pushforward map defined as $T(\bx) = \bx + 2 \, {\rm sign}(\bx) \odot (\sum_{i=1}^k \be_i)$, where ${\rm sign}(\bx)$ extracts the sign of $\bx$ elementwise and $\{\be_i\}_{i=1}^d$ are the canonical basis of $\sR^d$. 

We choose $d = 30, r = 5, k = 2, n = 100, \epsilon = 0.2$ and compare the proposed RHM-SD-F, RHM-CON ($\gamma = 0.5$), RHM-CG with RGDA in Fig.~\ref{prwd_plot}. RCEG cannot be implemented to solve \eqref{prwd_formula} because the doubly stochastic manifold does not have a well-defined logarithm map. From the results, we see similar convergence speed of all methods while due to the inbuilt line-search algorithm of RHM-CG, it converges to a point with a smaller gradient norm.

\subsection{Robust training of neural networks with orthonormal weights}
\label{robust_training_sect}

We next consider adversarial robust training of deep neural networks with orthonormal weights  \cite{huang2020gradient}. Adversarial training of neural networks provide robust prediction against small data perturbations. 
Orthonormality on parameters has shown to improve generalization accuracy as well as accelerate and stabilize convergence of neural network models \cite{bansal2018can,cogswell2015reducing,wang2020orthogonal,huang2018orthogonal}. This corresponds to optimization over the Stiefel manifold.

In particular, we consider the adversarial training to defend against a universal perturbation $\bp$ proposed in \cite{moosavi2017universal}. The perturbation set we consider is the sphere manifold $\gS^{d-1}(r):= \{ \bp\in \sR^d : \| \bp \|_2 = r \}$ with radius $r$. This requires the perturbed samples to stay a certain distance away from the original ones, a strategy also applied in \cite{li2020stochastic}. Given a set of data-target pairs $\{(\bx_i, y_i)\}_{i=1}^n$ where $\bx_i \in \sR^d$ are the feature vectors. The objective of adversarial training is 
\begin{equation*}
    \min_{ \{\bW_\ell\}_{\ell =1 }^L : \bW_\ell \in {\rm St}(d_\ell, d_{\ell +1}) }\ \max_{\bp \in \gS^{d-1}(r) } \,  \frac{1}{n}\sum_{i = 1}^n \mathcal{L} \big( h(\bx_i + \bp ; \{\bW_\ell \}_{\ell=1}^L ), y_i \big),
\end{equation*}
where $\mathcal{L}(\cdot, \cdot)$ is a loss function and $h(\cdot)$ represents the forward function of a neural network. 

\begin{figure*}[!t]
    \centering
    \begin{subfigure}[c]{.3\textwidth}
      \centering
      \includegraphics[scale=0.23]{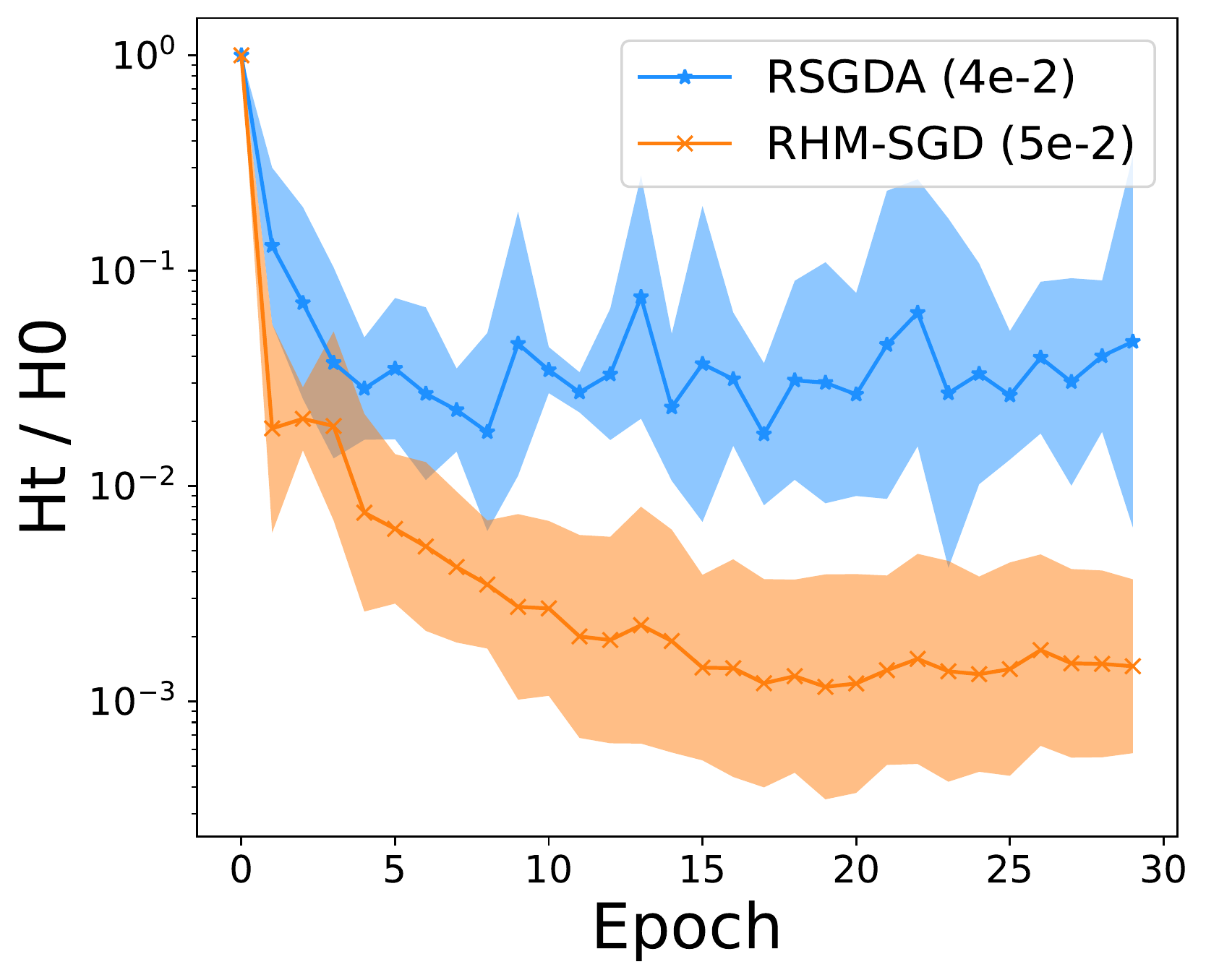}
      \caption{\texttt{RTNN}: {convergence}}
      \label{rtnn_plot} 
    \end{subfigure} 
    \begin{subfigure}[c]{.3\textwidth}
      \centering
      \includegraphics[scale=0.23]{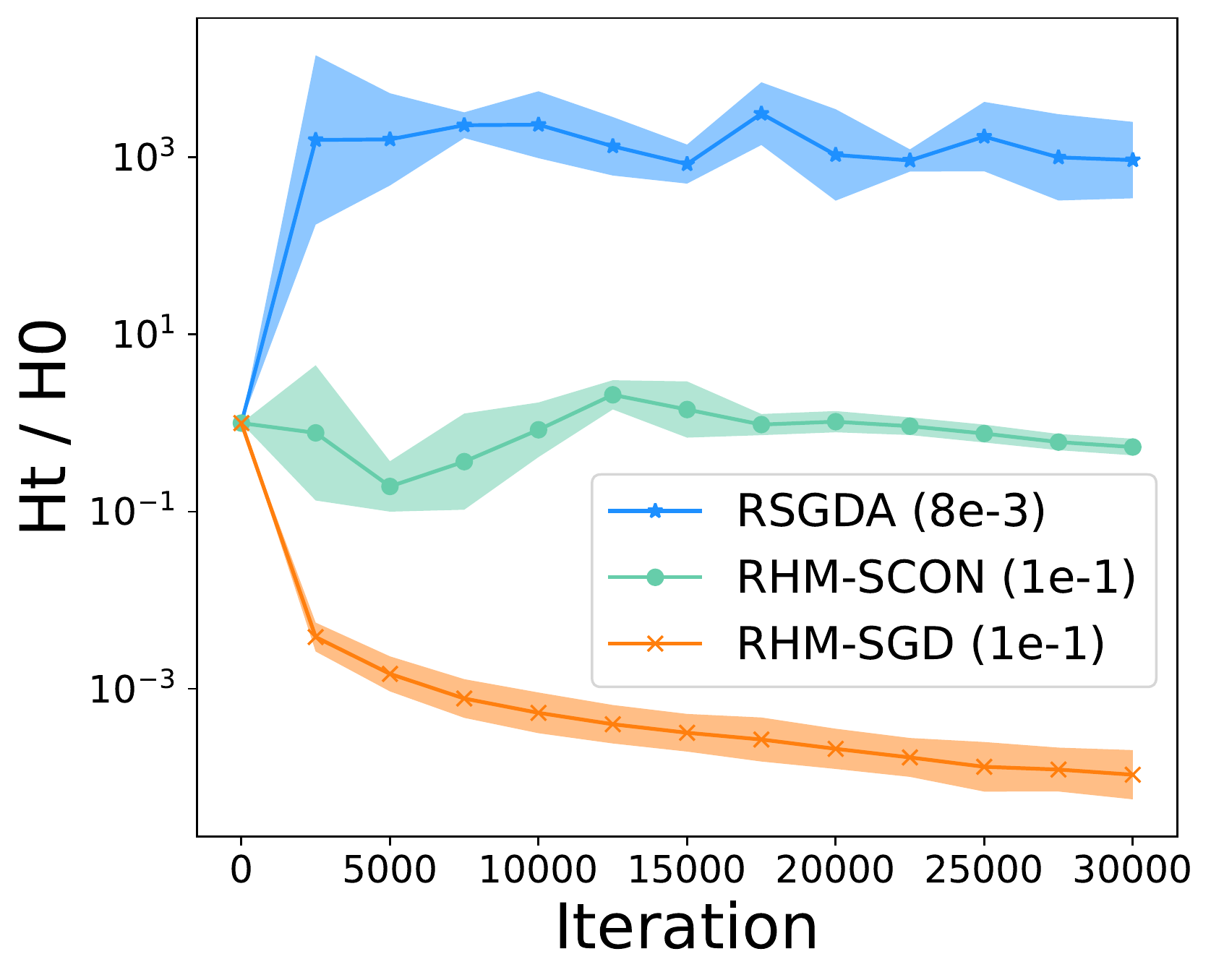}
      \caption{\texttt{OGAN}: {convergence}}
      \label{gan_hamit_plot} 
    \end{subfigure} 
    \begin{subfigure}[c]{.3\textwidth}
      \centering
      \includegraphics[scale=0.34]{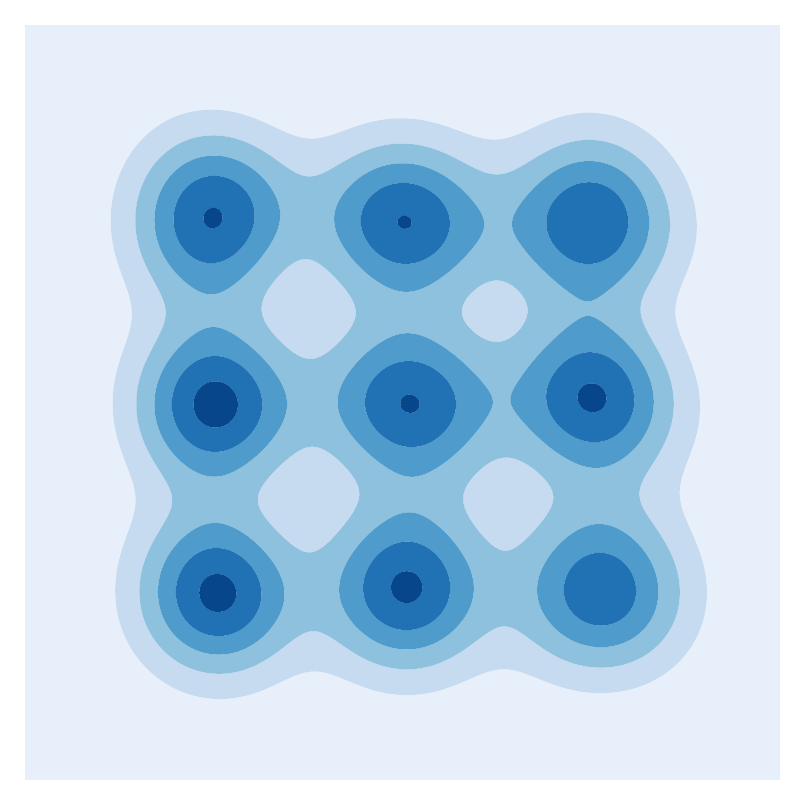}
      \\[11pt]
      \caption{\texttt{OGAN}: {ground truth}}
      \label{gan_groundtrusth} 
    \end{subfigure} 
    \\
    \begin{subfigure}[c]{.48\textwidth}
      \centering
      \includegraphics[scale=0.24]{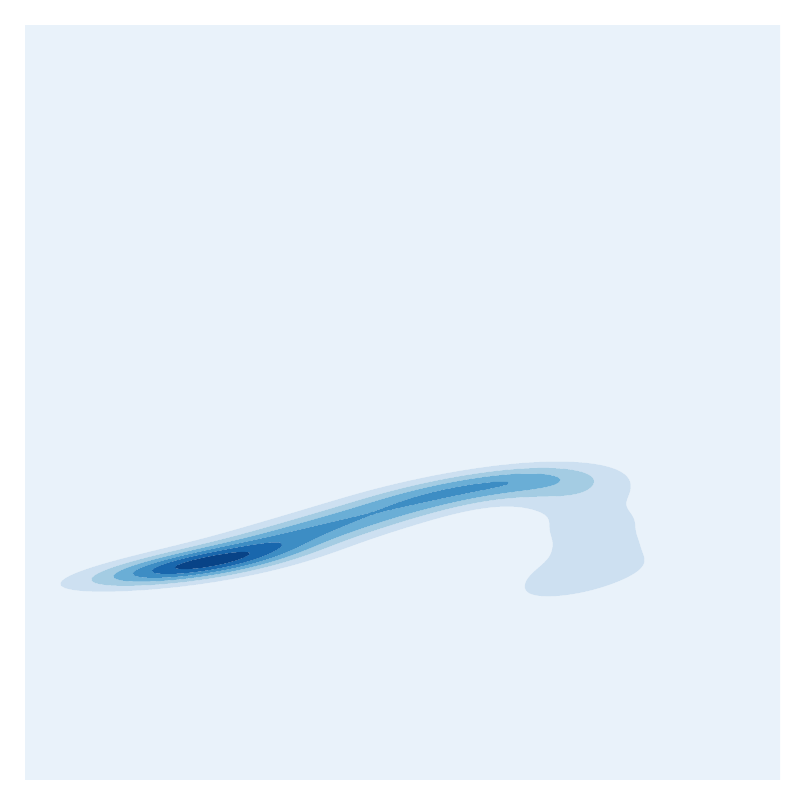}
       \includegraphics[scale=0.24]{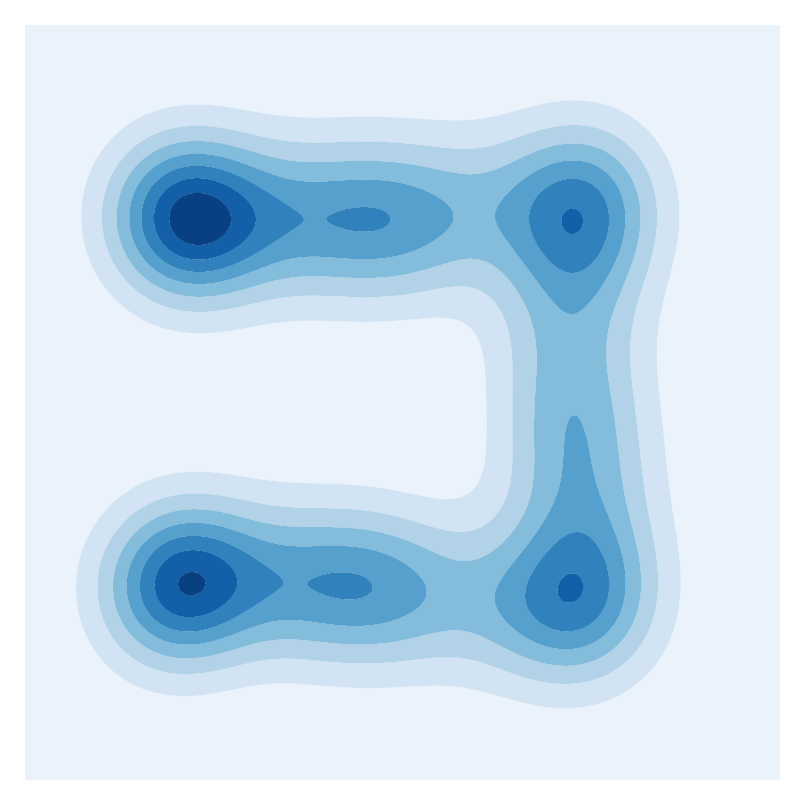}
       \includegraphics[scale=0.24]{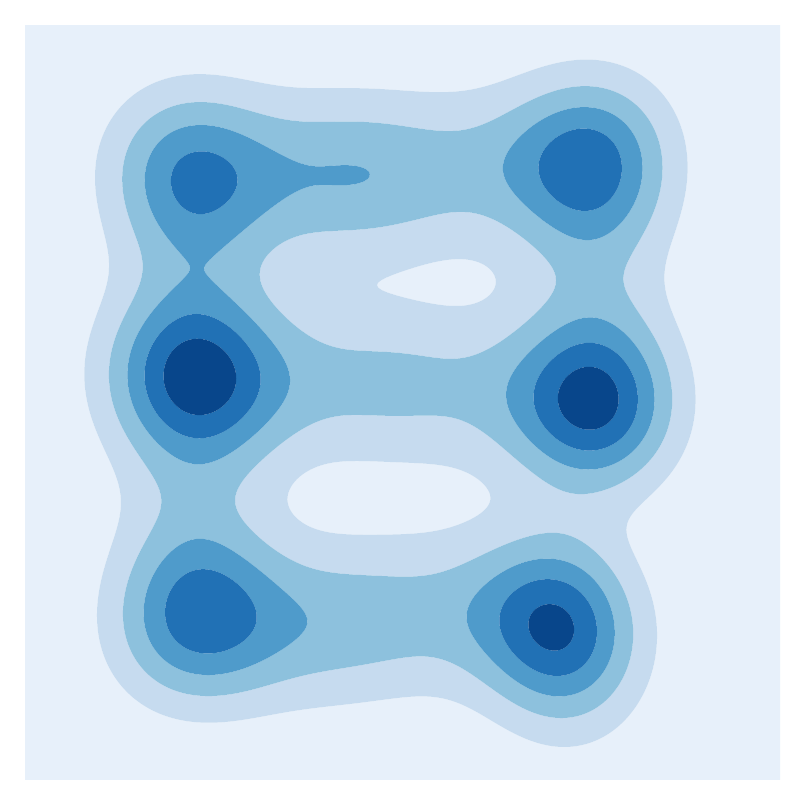}
      \caption{\texttt{RSGDA}: generated samples at $1, 2, 3 \times 10^4$ iterations from left to right} 
      \label{rsgda_generate_gan} 
    \end{subfigure} 
    \hfill
    \begin{subfigure}[c]{.48\textwidth}
      \centering
      \includegraphics[scale=0.24]{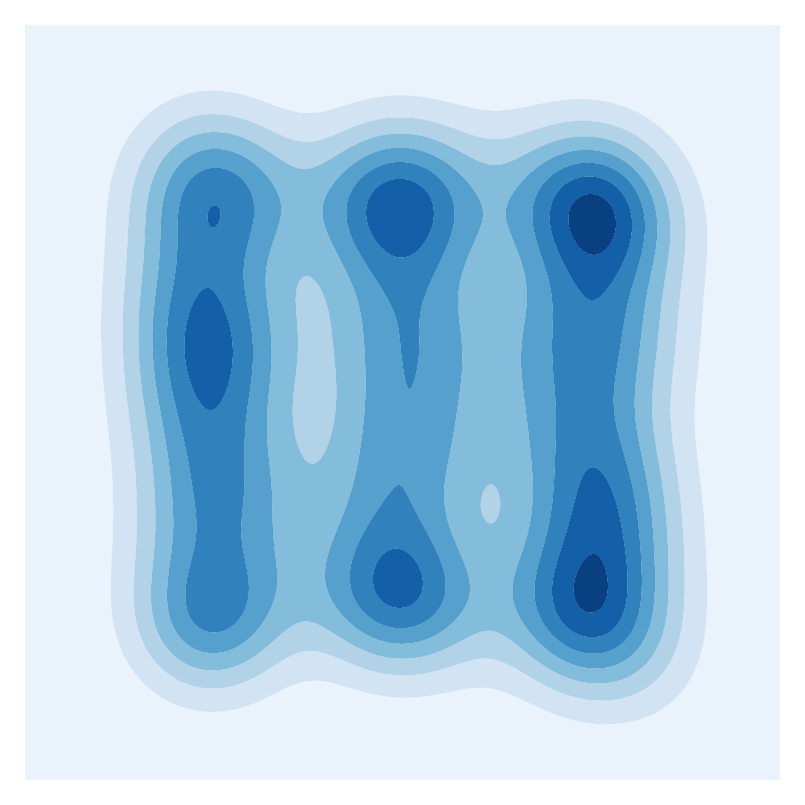}
      \includegraphics[scale=0.24]{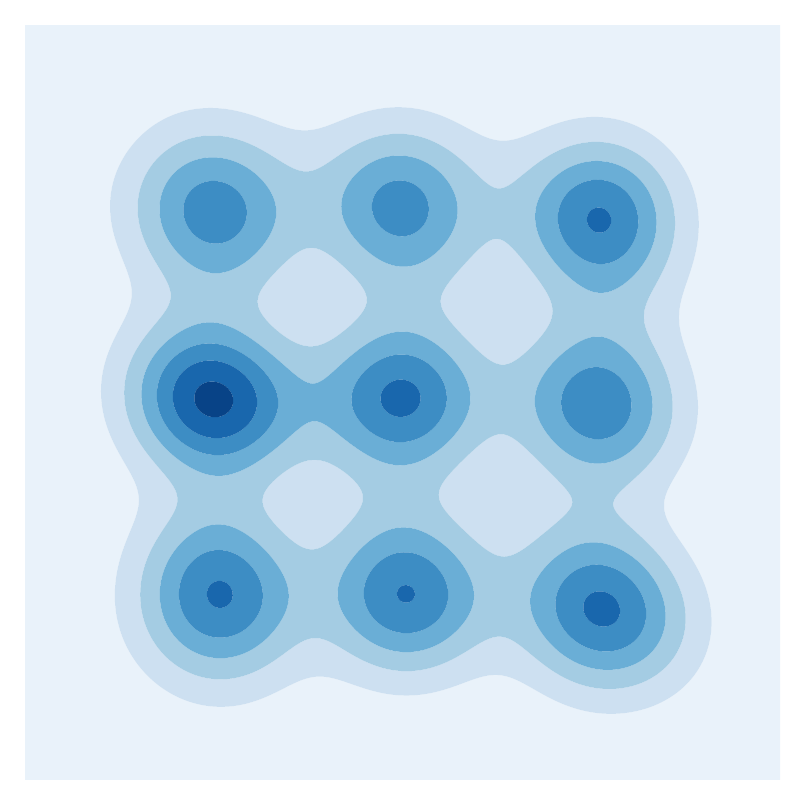}
      \includegraphics[scale=0.24]{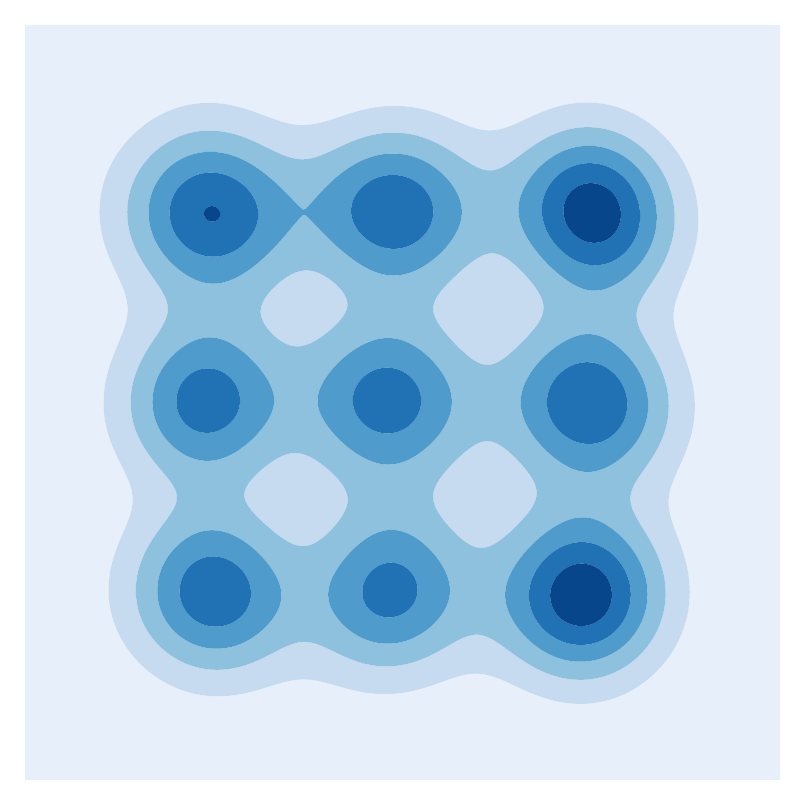}
      \caption{\texttt{RHM-SCON}: generated samples at $1, 2, 3 \times 10^4$ iterations from left to right} 
      \label{rhmsgd_generate_gan} 
    \end{subfigure} 
    \caption{\eqref{rtnn_plot}: Convergence on adversarial robust training of neural network (\texttt{RTNN}). \eqref{gan_hamit_plot}: generative adversarial networks with orthonormal weights (\texttt{OGAN}). \eqref{gan_groundtrusth}: Ground truth distribution. \eqref{rsgda_generate_gan}, \eqref{rhmsgd_generate_gan}: generated samples from RSGDA, RHM-SCON respectively where we see RHM-SCON quickly converge to the ground truth distribution while RSGDA suffers from mode collapse. The numbers in the parentheses indicate the best tuned stepsizes for different algorithms.}
    \label{rtnn_plots}
\end{figure*}

\subsection*{Experiment settings and results}
The adversarial training is implemented for classification tasks on MNIST images \cite{lecun1998gradient} where we include two hidden layers of size $16$ with the orthonormality constraint. We compare the proposed stochastic version of RHM (RHM-SGD), detailed in Section \ref{stochastic_hm_sect}, with Riemannian stochastic gradient descent ascent (RSGDA) algorithm \cite{huang2020gradient}. We highlight that RHM-SCON performs similarly to RHM-SGD, and thus, we exclude its result for clarity. Because we require dual sampling per-iteration to compute the stochastic Hamiltonian gradient $\grad \gH_{\gS, \gS'}(p_t) = \frac{1}{|\gS||\gS'|} \sum_{\omega \in \gS, \varphi \in \gS'} \grad \gH_{\omega,\varphi} (p)$, we choose the batch size to be $32$ for both $\gS, \gS'$ and $64$ for RSGDA. Hence, the per-iteration sampling cost is identical. We measure convergence in terms of the relative Hamiltonian ${\gH(p_t)}/{\gH(p_0)}$, where the Hamiltonian is evaluated on the full training set. The stepsize is fixed for both the algorithms.

We plot the convergence results (with the best tuned stepsize) in Fig. \ref{rtnn_plot}, which are averaged over five different runs. We see a clear advantage of RHM-SGD compared to RSGDA with faster and more stable convergence.

\subsection{Orthonormal generative adversarial networks}
\label{orthogonal_GAN_sect}
Generative adversarial networks (GAN) \cite{goodfellow2014generative,balduzzi2018mechanics} are popular in generating synthetic samples by optimizing a min-max game between a generator and a discriminator. The orthonormality constraint on weight parameters of the discriminator has shown to benefit the training of GANs \cite{brock2018large,muller2019orthogonal}. In particular, given samples $\{\bx_i\}_{i=1}^n$ we consider the following min-max problem
\begin{equation*}
    \min_{\{\bW^G_\ell\}}\  \max_{\{\bW^D_\ell\}: \bW^D_\ell \in {\rm St(d_\ell, d_{\ell +1})}} \, \frac{1}{n} \sum_{i=1}^n  \Big( \log (\sigma(D(\bx_i))) + \log(1 - \sigma(D(G(\bz_i))))  \Big),
\end{equation*}
where $D(\cdot), G(\cdot)$ represent the discriminator and generator with $\{ \bW^D_\ell\}, \{ \bW^G_\ell \}$ denoting their network weight parameters respectively. Here, $\sigma(\cdot)$ is the sigmoid function and the prior $\bz_i$ is sampled from the standard normal distribution.

\subsection*{Experiment settings and results}
Following \cite{balduzzi2018mechanics}, we train the GAN model on 2-d samples from a multimodal mixture of Gaussian distribution. The ground truth is shown in Fig. \ref{gan_groundtrusth}. Both the generator and discriminator have 5 hidden layers with 128 units and ReLU activation. The dimension of the prior $\bz_i$ is 64. For simplicity, we add the orthonormal constraint only for the penultimate layer of the discriminator model. For this experiment, we apply RHM-SCON 
with $\gamma = 0.5$ and compare against RSGDA, both with fixed stepsize. The batch size is chosen to be 128 for RHM-SCON and 256 for RSGDA. Similarly, the best choices of stepsize are reported, 
and the results are averaged over five different runs. 

The convergence in terms of the relative Hamiltonian are shown in Fig.~\ref{gan_hamit_plot}, where we see RSGDA diverges while RHM-SCON is more stable. We also examine the solution quality by providing the generated samples from both algorithms at iteration $10^4, 2\times 10^4,$ and $3\times 10^4$ in Figs~\ref{rsgda_generate_gan} and \ref{rhmsgd_generate_gan} respectively. We note that RSGDA results in undesired mode collapse, an observation also made in \cite{balduzzi2018mechanics} for training SGDA on the Euclidean space. In contrast, RHM-SCON quickly converges and recovers the ground truth distribution. Even though RHM-SGD converges to a lower Hamiltonian value, its performance in recovery of the ground truth is poor, as shown in Fig. \ref{rhm_puresgd_gan} in Appendix \ref{app:sec:experiments} where the generated samples collapse to a single point. It indicates that RHM-SGD converges to a stationary point which is not a saddle point (not surprising as Assumption \ref{stationary_assumption} may not be satisfied). This also highlights the practical benefit of consensus regularization for RHM (Section \ref{RHM_con_sect}), as evidenced in the good performance of RHM-SCON.


\section{Concluding remarks}\label{sec:conclusion}
Building on the success of the Hamiltonian methods for solving min-max problems in the Euclidean space, we have considered a more general problem on manifolds, and proposed a Riemannian Hamiltonian function $\gH$ that respects the manifold geometry. This leads to a gradient expression (in Proposition \ref{gradient_hamiltonian_prop}) that allows simple analysis for the resulting optimization methods. Adapting the proofs from the Euclidean space to Riemannian manifolds requires to forgo the matrix structure of the ingredients, which includes addressing a varying inner product (Riemannian metric). The proposed Riemannian Hamiltonian methods (RHM) come with convergence guarantees and various extensions. The experiments validate the good performance of RHM in different applications. As future work, one direction is to explore the utility of RHM for more general nonconvex nonconcave problems without the Riemannian PL assumption.
In addition, the current convergence analysis is measured in the Riemannian Hamiltonian, which is the gradient norm squared of the original objective $f$. It remains a question whether linear convergence can be maintained in terms of the optimality gap on function value of $f$.



\appendix



\section{Riemannian geometries of the considered manifolds}\label{app:sec:matrix_manifolds}

In this section, we review the Riemannian optimization-related ingredients of several manifolds that are considered in the experiments section. The expressions are from the works \cite{absil2009optimization,boumal2020introduction,sra2015conic,shi2021coupling,douik2019manifold,mishra2021manifold}.



\subsection{Symmetric positive definite manifold}
Consider the set of the symmetric positive definite matrices of size $d\times d$, $\sS_{++}^d := \{ \bX: \sR^{d \times d}: \bX^\top = \bX, \bX \succeq \bzero \}$, equipped with the affine-invariant Riemannian metric. The geodesic from $\bX$ to $\bY$ is given by $\gamma(t) = \bX^{1/2}( \bX^{-1/2} \bY \bX^{-1/2} )^t \bX^{1/2}$. At $\bX \in \sS_{++}^d$, the exponential map is derived as ${\rm Exp}_{\bX}(\bU) = \bX \exp(\bX^{-1}\bU)$ for any $\bU \in T_\bX\sS_{++}^d$. The logarithm map is ${\rm Log}_{\bX}(\bY) = \bX \log(\bX^{-1} \bY)$. The Riemannian gradient of a function $f: \sS_{++}^d: \xrightarrow{} \sR$ is given by $\grad f(\bX) = \bX \nabla f(\bX) \bX$, where $\nabla f(\bX)$ is the Euclidean partial derivative of $f$ at $\bX$.

\subsection{Sphere manifold} It is defined as
$\gS^{d-1} = \{ \bx \in \sR^d : \| \bx \|_2 = 1 \}$, which is an embedded submanifold of $\mathbb{R}^d$ with the tangent space expression $T_\bx \gS^{d-1} = \{ \bu \in \sR^d: \bx^\top \bu = 0\}$. It can be endowed with the standard inner product at the Riemannian metric, i.e., $\langle \bu, \bv \rangle_\bx = \langle \bu, \bv \rangle_2$, for $\bu,\bv \in T_\bx \gS^{d-1}$. The orthogonal projection of any $\bv \in \sR^d$ to $T_\bx \gS^{d-1}$ is derived as ${\rm Proj}_\bx(\bv) = \bv - (\bx^\top \bv) \bx$. The exponential map along $u\in T_\bx \gS^{d-1}$ is ${\rm Exp}_{\bx}(\bu) = \cos(\| \bv \|_2) \bx + \sin (\| \bv\|_2) \frac{\bv}{\| \bv\|}$ and the logarithm map is ${\rm Log}_{\bx}(\by) =  \arccos(\bx^T\by)\frac{{\rm Proj}_\bx(\by - \bx)}{\|{\rm Proj}_\bx(\by - \bx) \|_2}$. The Riemannian gradient of $f$ is ${\rm Proj}_{\bx}(\nabla f(\bx))$, where $\nabla f(\bx)$ is the Euclidean partial derivative of $f$ at $\bx$.

\subsection{Stiefel manifold} It is the set ${\rm St}(d, r) = \{ \bX \in \sR^{d \times r} : \bX^\top \bX = \bI \}$. It is a generalization of the sphere manifold to higher dimensions and can be similarly endowed with the standard inner product as metric $\langle \bU, \bV \rangle_\bX = \langle \bU, \bV \rangle_2$ For the experiments, we consider the popular QR-based retraction for approximating the exponential map, i.e., $R_{\bX}(\bU) = {\rm qf}(\bX + \bU)$, where ${\rm qf}(\cdot)$ returns the Q-factor from the QR decomposition for any tangent vector $\bU$.

\subsection{Doubly stochastic manifold}
The doubly stochastic manifold (or coupling manifold) between two discrete probability measures $\mu = \sum_{i = 1}^m a_i \delta_{\bx_i}, \nu = \sum_{j=1}^n b_j \delta_{\by_j}$ is the set of couplings $\Pi(\mu, \nu) := \{ \bGamma \in \sR^{m \times n} : \bGamma_{i,j} > 0, \sum_i \bGamma_{i,j} = b_j, \sum_j \bGamma_{i,j} = a_i, \forall \, i,j \}$ endowed with the Fisher information Riemnnanian metric. The geometry has been developed in \cite{douik2019manifold,shi2021coupling,mishra2021manifold}.

Without loss of any generality, we assume $\sum_i a_i = \sum_j b_j = 1$. The tangent space at $\bGamma \in \Pi(\mu, \nu)  $ is given by $T_\bGamma \Pi(\mu, \nu) = \{ \bU \in \sR^{m \times n} : \sum_i \bU_{i,j} = \sum_j \bU_{i,j} =  0, \forall \, i,j\}$. The Fisher information metric is defined as for $\bU, \bV \in T_\bGamma \Pi(\mu, \nu)$, $\langle \bU, \bV \rangle_\bGamma = \sum_{i,j} (\bU_{i,j} \bV_{i,j})/\bGamma_{i,j}$. For the experiments, we consider the Sinkhorn-based retraction. The Sinkhorn-Knopp algorithm \cite{sinkhorn1964relationship} is a popular approach for balancing non-negative matrices to satisfy the row-sum and column sum constraint and later adapted to solve the optimal transport problem efficiently \cite{peyre2019computational}. Let $\bA \in \sR^{m \times n}, \bA_{i,j} > 0$, and denote ${\rm Sinkhorn}(\bA)$ as the output of applying the Sinkhorn-Knopp algorithm on $\bA$ with constraint defined by $\Pi(\mu, \nu)$, i.e., ${\rm Sinkhorn}(\bA) \in \Pi(\mu, \nu)$. Subsequently, the retraction is given by $R_\bGamma( \bU ) = {\rm Sinkhorn}( \bGamma \odot \exp( \bU \oslash \bGamma ) )$, where $\exp$, $\odot$, and $\oslash$ are elementwise exponential, product, and division operations, respectively.

\section{Line-search methods and Wolfe conditions on Riemannian manifolds}
\label{appendix_linesearch}
In this section, we present the Riemannian versions of the Armijo, Wolfe, and strong Wolfe conditions \cite{sato2021riemannian}.
\begin{definition}
\label{def_linesearch_condition}
Consider an iterative algorithm for minimizing $h:\M \xrightarrow{} \sR$, producing $p_{t+1} = {\rm Exp}_{p_{t}}(\vartheta_t \xi_t)$ for some direction $\xi_t \in T_{p_t}\M$ and stepsize $\vartheta_t \in \sR$. The Armijo condition is $h(p_{t}) - h(p_{t+1}) \geq r_1 \vartheta_t \langle -\grad h(p_t), \xi_t \rangle$, for some $r_1 \in (0,1)$. The (weak) Wolfe condition is the Armijo condition together with \eqref{weak_wolfe} and the strong Wolfe condition is the Armijo condition with \eqref{strong_wolfe}, where
\begin{align}
    \langle \grad h(p_{t+1}), \D {\rm Exp}_{p_t}(\vartheta_t \xi_t)[\xi_t] \rangle_{p_{t+1}} &\geq r_2 \langle \grad h(p_t), \xi_t \rangle_{p_t} \label{weak_wolfe} \\
    |\langle \grad h(p_{t+1}), \D {\rm Exp}_{p_t}(\vartheta_t \xi_t)[\xi_t] \rangle_{p_{t+1}}| &\leq r_2 |\langle \grad h(p_t), \xi_t \rangle_{p_t}| \label{strong_wolfe}
\end{align}
for some $r_2 \in (r_1, 1)$. Here, $\D {\rm Exp}$ is the differential of the exponential operation.
\end{definition}
The backtracking line-search for satisfying the Armijo condition has been used in Riemannian steepest descent method \cite{boumal2019global}.

One can generalize the analysis from the Euclidean space to show that there exists a stepsize that satisfy the three conditions for arbitrary direction $\xi_t$. The backtracking line-search for satisfying the Armijo condition is in Algorithm \ref{line_search_algorithm}. This has been used in Riemannian steepest descent method \cite{boumal2019global}. The procedures that return stepsizes satisfying the Wolfe conditions are in \cite{sato2016dai,nocedal1999numerical}. 

\begin{algorithm}[t]
 \caption{Backtracking line-search}
 \label{line_search_algorithm}
 \begin{algorithmic}[1]
  \STATE \textbf{Input:} Current iterate $p_t \in \M$, search direction $\xi_t \in T_{p_t}\M$, initial stepsize $\bar{\vartheta}$ and $r_1, \varrho \in (0,1)$.
  \STATE Initialize $\vartheta \xleftarrow{} \bar{\vartheta}$.
  \WHILE{$h(p_t) - h({\rm Exp}_{p_t}(\vartheta \xi_t)) < r_1 \vartheta \langle -\grad h(p_t), \xi_t \rangle_{p_t}$}
  \STATE Set $\vartheta \xleftarrow{} \varrho \bar{\vartheta}$. 
  \ENDWHILE
  \STATE \textbf{Output:} $\vartheta$.
 \end{algorithmic} 
\end{algorithm}

\section{Review of RGDA and RCEG}\label{app:sec:rgda_rceg}
In this section, we provide the details of the Riemannian gradient descent ascent \cite{huang2020gradient} and Riemannian corrected extra-gradient \cite{zhang2022minimax} algorithms for min-max optimization on manifolds. 

RGDA simultaneously updates the variables in the direction of the min-max Riemannian gradient, i.e.,
\begin{equation*}
    x_{t+1} = {\rm Exp}_{x_t}(-\eta_t \ \grad_x f(x_t, y_t)), \quad y_{t+1} = {\rm Exp}_{y_t}(\eta_t \ \grad_y f(x_t, y_t)).
\end{equation*}

RCEG first updates the variables to the point $(w_t, z_t)$ along the min-max Riemannian gradient. It then uses the obtained point to generate the final update, i.e.,
\begin{align*}
    w_t &= {\rm Exp}_{x_t}(-\eta_t \ \grad_x f(x_t, y_t)), \\
     z_t &= {\rm Exp}_{y_t}(\eta_t \ \grad_y f(x_t, y_t)), \\
    x_{t+1} &= {\rm Exp}_{w_t} (-\eta_t \ \grad_x f(w_t, y_t) + {\rm Log}_{w_t}(x_t)), \\
    y_{t+1} &= {\rm Exp}_{z_t} (\eta_t \ \grad_y f(w_t, y_t) + {\rm Log}_{z_t}(y_t)).
\end{align*}

{
In \cite{zhang2022minimax}, only  convergence for g-convex-concave functions is analyzed, where the authors show that RCEG converges sublinearly with {averaged iterate} under the fixed stepsize $\eta \leq \frac{1}{2L_1 \tau_{\zeta, D}}$ where $\tau_{\zeta, D} > 1$ depends on the curvature and diameter of the domain. Thus, the analysis is only local with domain-dependent rate of convergence.  
The recent work \cite{jordan2022first} starts by showing average-iterate convergence of RCEG under g-convex-concave functions and last-iterate convergence under g-strongly-convex-concave functions. Nevertheless, similar assumptions on the bounded domain (and also the curvature) is required. The stepsize also requires to be carefully selected, which depends on the curvature and diameter bound. In addition, \cite{jordan2022first} proves convergence for RGDA under similar settings. 
For g-strongly-convex-concave functions, the last-iterate convergence of RGDA requires a diminishing stepsize, and for g-convex-concave functions, the average-iterate convergence of RGDA require a stepsize that again depends on the curvature and diameter bound. 
}

\section{Key propositions}
\label{app:sec:productHessian}

In this section, we derive the explicit expression for the Riemannian Hessian on the product manifold $\M = \M_x \times \M_y$ and show that the cross derivatives are adjoint with respect to the Riemannian metric.



\begin{proposition}[Riemannian Hessian of product manifold]
\label{riem_hess_prop}
Consider a product Riemannian manifold $\M = \M_x \times \M_y$ and $f : \M \xrightarrow{} \sR$. For any $p = (x, y) \in \M$ and $\xi = (u,v) \in T_x\M$, the Riemannian Hessian $\hess f(p)[\xi]$ is derived as 
\begin{equation*}
    \hess f(p)[\xi] = \begin{pmatrix} \hess_x f(x,y)[u] + \grad^2_{yx} f(x,y)[v] \\[2pt] \grad_{xy}^2 f(x,y)[u]  + \hess_y f(x,y)[v] \end{pmatrix}.
\end{equation*}
\end{proposition}

\begin{proof}
From standard analysis, the Levi-Civita connection on a product manifold $\M = \M_x \times \M_y$ (e.g., in \cite[Exercise~5.4]{boumal2020introduction}) is given by 
\begin{equation*}
    \bnabla_{(U_x, U_y)}(V_x, V_y) = \Big( \bnabla^{(x)}_{U_x}V_x +  \D_y V_x [U_y], \, \D_x V_y [U_x] + \bnabla^{(y)}_{U_y} V_y \Big),
\end{equation*}
where $V_x \in \mathfrak{X}(\M_x), V_y \in \mathfrak{X}(\M_y)$ are vector fields on respective manifolds and $\D$ is the directional derivative. Further, $\D_y V_x: \mathfrak{X}(\M_y) \xrightarrow{} \mathfrak{X}(\M_x)$ and when evaluating at $(x,y)$, this is equivalently defined as $\D_y V_x(x, \cdot)(y) : T_y\M_y \xrightarrow{} T_x\M_x$, which is the directional derivative. $\bnabla^{(x)}, \bnabla^{(y)}$ are the Levi-Civita connections on $\M_x, \M_y$, respectively. Applying the definition of the Riemannian Hessian, $\hess f(p)[\xi] = \bnabla_\xi \grad f(p)$, we obtain the desired result.
\end{proof}


\begin{proposition}
\label{symmetry_cross_derivative_prop}
For any $(x,y) \in \M_x \times \M_y$ and $(u,v) \in T_{x}\M_x \times T_y\M_y$, we have $\langle \grad^2_{yx}f(x,y) [v], u \rangle_x \allowbreak = \langle \grad_{xy}^2 f(x,y)[u], v\rangle_y$. Equivalently, $\grad^2_{yx} f(x,y)$ is the adjoint operator of $\grad^2_{xy} f(x,y)$. 
\end{proposition}

\begin{proof}
Let $p = (x,y)$ and $\xi = (u,v), \zeta = (w,z)$ for any $(u,v), \allowbreak (w,z) \in T_x\M_x \times T_y\M_y$. Then, from the self-adjoint property (symmetry) of the Riemannian Hessian, we have 
\begin{equation}
    \langle \hess f(p)[\xi], \zeta \rangle_p = \langle \hess f(p)[\zeta], \xi \rangle_p, \label{symmetric_prop_eq1}
\end{equation}
for any $\xi, \zeta$. Combining with Proposition \ref{riem_hess_prop}, the result \eqref{symmetric_prop_eq1} is equivalent to 
\begin{align*}
    &\langle \hess_x f(x,y)[u], w \rangle_x + \langle \grad^2_{yx} f(x,y)[v], w \rangle_x +  \langle \grad^2_{xy} f(x,y) [u], z  \rangle_y \\
    &\qquad + \langle \hess_y f(x,y)[v], z \rangle_y \\
    = &\langle \hess_x f(x,y)[w], u \rangle_x + \langle \grad^2_{yx} f(x,y)[z], u\rangle_x + \langle \grad^2_{xy} f(x,y)[w], v \rangle_y \\
    &\qquad + \langle \hess_y f(x,y)[z], v \rangle_y.
\end{align*}
Given that $\hess_x$ and $\hess_y$ satisfy the self-adjoint property, we obtain 
\begin{align}
    &\langle \grad^2_{yx} f(x,y)[v], w \rangle_x + \langle \grad^2_{xy} f(x,y) [u], z  \rangle_y \nonumber\\
    = &\langle \grad^2_{yx} f(x,y)[z], u\rangle_x + \langle \grad^2_{xy} f(x,y)[w], v \rangle_y. \label{sym_prop_eq2}
\end{align}
We can see \eqref{sym_prop_eq2} holds for any choice of $(u,v), (w,z)$ and this only happens when $\langle \grad^2_{yx}f(x,y) [v], u \rangle_x = \langle \grad_{xy}^2 f(x,y)[u], v\rangle_y$ holds for any $(u,v)$. To see this, consider the vectorization of the tangent vectors as $\bu, \bv, \bw, \bz$. We also denote $\bB_{xy}, \bB_{yx}$ as the matrix representation of the linear operators $\grad^2_{xy}f(x,y), \grad^2_{yx}f(x,y)$ at $(x,y)$ respectively. Then \eqref{sym_prop_eq2} can be rewritten as 
\begin{align*}
    \bw^\top \bG_x \bB_{yx} \bv + \bz^\top \bG_y \bB_{xy} \bu = \bu^\top \bG_x \bB_{yx} \bz + \bv^\top \bG_y \bB_{xy} \bw,
\end{align*}
where $\bG_x, \bG_y$ are the (symmetric positive definite) metric tensors at $x,y$. This is equivalent to 
\begin{equation*}
    \bz^\top \big( \bG_y \bB_{xy} - \bB_{yx}^\top \bG_x  \big) \bu = \bv^\top \big( \bG_y \bB_{xy} - \bB_{yx}^\top \bG_x \big) \bw,
\end{equation*}
which is satisfied for any $\bu, \bv, \bw, \bz$ and any $\bG_x, \bG_y$ as metric tensors. Hence, $\bG_y \bB_{xy} = \bB_{yx}^\top \bG_x$ and the proof is complete.
\end{proof}

\begin{remark}
Proposition \ref{symmetry_cross_derivative_prop} shows that the Riemannian cross derivatives are symmetric with respect to Riemannian metric on respective manifolds. When $\M_x$, $\M_y$ are the Euclidean spaces, then Proposition \ref{symmetry_cross_derivative_prop} is equivalent to the Schwarz's theorem of symmetric second-order derivatives.
\end{remark}

\section{Essential lemmas}
\label{essential_lemma_appendix}

The following lemmas generalize \cite[Lemmas 17, 28]{abernethy2019last} to linear operators, specifically in terms of the Riemannian Hessian operator. We first highlight that for two operators $T$, $T^*$ that are adjoint, we have $\lambda(T \circ T^*) = \lambda(T^* \circ T ) = \sigma^2(T) =\sigma^2(T^*)$. 


\begin{lemma}
\label{lemma_hess_bound_linear}
Consider the Riemannian Hessian $\hess f(p)$ where $p = (x,y) \in \M_x \times \M_y$. Suppose $\hess_y f(x,y) = 0$. Then, $\lambda_{\rm |min|}(\hess f(p)) \allowbreak \geq \frac{\sigma^2_{\min}(B_{xy})}{ \sqrt{2\sigma^2_{\min} (B_{xy}) + \| H_x\|^2_x}}$.
\end{lemma}

\begin{proof}
We consider the operator $\hess f(p) \circ \hess f(p)$ and study its eigenvalue. First, we see {that} for any $p =(x,y) \in \M_x \times \M_y$ and $\xi = (u, v) \in T_x\M_x \times T_y \M_y$, we have 
%
%
\begin{align*}
    \hess f(p) [\xi] &= \begin{pmatrix}
                \hess_x f(x,y)[u] + \grad_{yx}^2 f(x,y)[v]  \\
                \grad_{xy}^2 f(x,y)[u]
                \end{pmatrix}, 
\end{align*}
and therefore,
\begin{align*}
     &\hess f(p) [\hess f(p)[\xi]]  \\
     &= \begin{pmatrix} \hess_x f(x,y)[ \hess_x f(x,y)[u]] + \hess_x f(x,y)[\grad^2_{yx}f(x,y)[v]] \\
      + \grad^2_{yx} f(x,y) [\grad^2_{xy} f(x,y)[u]] \\[5pt]
     \grad^2_{xy} f(x,y)[\hess_x f(x,y) [u]] + \grad^2_{xy}f(x,y) [\grad^2_{yx}f(x,y)[v]]
     \end{pmatrix}.
\end{align*}
%
%
Suppose $(\delta, \xi)$ is an eigenpair of the operator $\hess f(p) \circ \hess f(p)$, which gives
\begin{align}
    &\hess_x f(x,y)[ \hess_x f(x,y)[u]] + \hess_x f(x,y)[\grad^2_{yx}f(x,y)[v]]  \nonumber\\ 
     &\hspace{4.5cm} + \grad^2_{yx} f(x,y) [\grad^2_{xy} f(x,y)[u]]  = \delta  u, \label{lemma_2a_eq1} \\
    & \grad^2_{xy} f(x,y)[\hess_x f(x,y) [u]] + \grad^2_{xy}f(x,y) [\grad^2_{yx}f(x,y)[v]] = \delta v. \label{lemma_2a_eq2}
\end{align}
Let $B_{xy} = \grad^2_{xy} f(x,y), B_{yx} = \grad^2_{yx} f(x,y)$, and $H_x = \hess_x f(x,y)$. Suppose $\delta < \frac{\sigma^4_{\min} (B_{xy})}{2\sigma^2_{\min}(B_{xy}) + \| H_x\|^2_x} \allowbreak < \sigma^2_{\min} (B_{xy})$. Then, we have $B_{xy} \circ B_{yx} - \delta \, \id$ is invertible where we use the fact that $B_{xy}$ and $B_{yx}$ are adjoint. 
%
Hence, from \eqref{lemma_2a_eq2} we have $v = - (B_{xy} \circ B_{yx} - \delta \, \id)^{-1} \circ (B_{xy} \circ H_x) [u]$. Substituting the expression of $v$ 
%
%
into \eqref{lemma_2a_eq1} yields
\begin{equation}
    \Big(H_x \circ \big( \id -  B_{yx} \circ (B_{xy} \circ B_{yx} - \delta \, \id)^{-1} \circ B_{xy}  \big) \circ H_x   + B_{yx} \circ B_{xy} - \delta \, \id \Big) [u] = 0. \label{lemma_2a_eq3}
\end{equation}
We next show that when 
\begin{equation}
    \delta < \frac{\sigma^4_{\min} (B_{xy})}{2\sigma^2_{\min}(B_{xy}) + \| H_x\|^2_x} < \sigma^2_{\min} (B_{xy}), \label{delta_condition}
\end{equation}
then \eqref{lemma_2a_eq3} does not have a nontrivial solution in $u$ (i.e., $u \neq 0$), which leads to a contradiction that $\xi$ is an eigenvector.
%
%
It suffices to show that for any $\delta$ satisfying the condition \eqref{delta_condition}, 
%
%
the following inequality 
\begin{equation}
     \frac{- \delta \| H_x\|_x^2}{\sigma_{\min}^2(B_{xy}) - \delta}+ \sigma_{\min}^2 (B_{xy}) - \delta > 0, \label{lemma_2a_eq4}
\end{equation}
holds, which violates \eqref{lemma_2a_eq3}. 
%
%
Here, we highlight $B_{xy}$ is the adjoint of $B_{yx}$, and therefore, the eigenvalues $\lambda_i(\id - B_{yx} \circ (B_{xy} \circ B_{yx} - \delta \, \id)^{-1} \circ B_{xy}) =  \frac{- \delta}{\sigma^2_i(B_{xy}) - \delta} < 0$ from the singular value decomposition of $B_{xy}$. The roots of \eqref{lemma_2a_eq4} are
\begin{align*}
    r_1 &= \sigma^2_{\min}(B_{xy}) + \frac{1}{2}  \|H_x \|_x^2 - \sqrt{(\sigma^2_{\min}(B_{xy}) + \frac{1}{2}  \|H_x \|_x^2)^2 - \sigma^4_{\min}(B_{xy})} \\
    r_2 &= \sigma^2_{\min}(B_{xy}) + \frac{1}{2}  \|H_x \|_x^2 + \sqrt{(\sigma^2_{\min}(B_{xy}) + \frac{1}{2}  \|H_x \|_x^2)^2 - \sigma^4_{\min}(B_{xy})}.
\end{align*}
One can show for any $c_1 > 0$, $4c_2 < c_1^2$, then $\frac{2c_2}{c_1} < c_1 - \sqrt{c_1^2 - 4c_2}$. Let $c_1 =  \sigma^2_{\min} (B_{xy}) + \frac{1}{2}\| H_x\|^2_x$, $c_2 = \frac{1}{4} \sigma^4_{\min}(B_{xy})$, we have the smaller root satisfies $r_1 > \frac{\sigma^4_{\min}(B_{xy})}{2\sigma^2_{\min} (B_{xy}) + \| H_x\|^2_x} > \delta$, 
%
%
%
Hence, there 
%
%
does not exist $u \neq 0$ that satisfies \eqref{lemma_2a_eq3}, which implies $\delta \geq \frac{\sigma^4_{\min}(B_{xy})}{2\sigma^2_{\min} (B_{xy}) + \| H_x\|^2_x}$. This completes the proof.
\end{proof}

\begin{lemma}
\label{lemma_sufficient_linear_H}
Consider the Riemannian Hessian $\hess f(p)$, where $p = (x,y) \in \M_x \times \M_y$. Let $H_x := \hess_x f(x,y), H_y := \hess_y f(x,y)$, $B_{xy} := \grad^2_{xy} f(x,y)$, and
\begin{align*}
         a &= 2 \sigma_{\min}^2(B_{xy}) + \lambda^2_{{\rm |min|}}(H_x) + \lambda^2_{{\rm |min|}}(H_y), \\ 
         b &= \Big(\sigma_{\min}^2(B_{xy}) + \lambda^2_{{\rm |min|}}(H_x) \Big) \Big(\sigma_{\min}^2(B_{xy})  + \lambda^2_{{\rm |min|}}(H_y) \Big)  \\
         &\qquad \qquad - \sigma_{\max}^2(B_{xy})(\| H_x \|_x + \| H_y \|_y)^2.
\end{align*}
Suppose that $b > 0$. Then, $\lambda_{\rm |min|}(\hess f(p)) \geq \sqrt{\frac{b}{a}}$.
\end{lemma}

\begin{proof}
Similarly to Lemma \ref{lemma_hess_bound_linear}, we consider the operator $\hess f(p) \circ \hess f(p)$, i.e.,
\begin{align*}
    \hess f(p) [\xi] &= \begin{pmatrix}
                \hess_x f(x,y)[u] + \grad_{yx}^2 f(x,y)[v]  \\
                \hess_y f(x,y)[v] + \grad_{xy}^2 f(x,y)[u]
                \end{pmatrix}, 
\end{align*}
and
\begin{align*}
     &\hess f(p) [\hess f(p)[\xi]]  \\
     &= \begin{pmatrix} \hess_x f(x,y)[ \hess_x f(x,y)[u]] + \hess_x f(x,y)[\grad^2_{yx}f(x,y)[v]] \\
      + \grad^2_{yx} f(x,y) [\hess_y f(x,y) [v]] + \grad^2_{yx} f(x,y) [\grad^2_{xy} f(x,y)[u]] \\[5pt]
     \hess_y f(x,y)[\hess_y f(x,y)[v]] + \hess_y f(x,y)[\grad^2_{xy} f(x,y)[u]] \\
     + \grad^2_{xy} f(x,y)[\hess_x f(x,y) [u]] + \grad^2_{xy}f(x,y) [\grad^2_{yx}f(x,y)[v]].
     \end{pmatrix}.
\end{align*}
Suppose $(\delta, \xi)$ is an eigenpair of the operator $\hess f(p) \circ \hess f(p)$, which gives
\begin{align}
     &\hess_x f(x,y)[ \hess_x f(x,y)[u]] + \hess_x f(x,y)[\grad^2_{yx}f(x,y)[v]]  \nonumber\\ 
     &\; + \grad^2_{yx} f(x,y) [\hess_y f(x,y) [v]] + \grad^2_{yx} f(x,y) [\grad^2_{xy} f(x,y)[u]]  = \delta  u, \label{lemma_la_eq1} \\
    &\hess_y f(x,y)[\hess_y f(x,y)[v]] + \hess_y f(x,y)[\grad^2_{xy} f(x,y)[u]] \nonumber\\
    &\; + \grad^2_{xy} f(x,y)[\hess_x f(x,y) [u]] + \grad^2_{xy}f(x,y) [\grad^2_{yx}f(x,y)[v]] = \delta v. \label{lemma_la_eq2}
\end{align}
Denote $T_x := H_x \circ H_x  + B_{yx} \circ B_{xy} - \delta \, \id$ and similarly for $T_y := H_y \circ H_y + B_{xy} \circ B_{yx} -\delta \, \id$, where $H_x = \hess_x f(x,y)$, $H_y = \hess_y f(x,y)$ and $B_{xy} = \grad^2_{xy} f(x,y)$, $B_{yx} = \grad^2_{yx} f(x,y)$. Then, we can simplify \eqref{lemma_la_eq1} and \eqref{lemma_la_eq2} as 
\begin{equation}
\begin{array}{ll}
     T_x [u] &= - (H_x \circ B_{yx} + B_{yx} \circ H_y) [v] \\
     T_y [v] &= - (H_y \circ B_{xy} + B_{xy} \circ H_x) [u]
\end{array}
\label{lemma_la_eq0}
\end{equation}
Suppose $\delta < \frac{b}{a}$. Then, we can show $T_y$ is invertible. This is because, for any $c_1 > 0$, $4c_2 < c_1^2$, we have $\frac{2c_2}{c_1} < c_1 - \sqrt{c_1^2 - 4c_2}$. From the definition of $a$ and $b$ and setting $c_1 = a, c_2 = b$, we have 
\begin{align*}
    \frac{2b}{a} &< 2 \sigma^2_{\min} (B_{xy}) + \lambda_{\min}(H_x \circ H_x) + \lambda_{\min}(H_y \circ H_y) \\
    &\quad - \sqrt{(\lambda_{\min}(H_x \circ H_x) - \lambda_{\min}(H_y \circ H_y))^2 + 4 \sigma_{\max}^2(B_{xy})(\| H_x \|_x + \| H_y \|_y)^2 } \\
    &< 2 \sigma^2_{\min} (B_{xy}) + \lambda_{\min}(H_x \circ H_x) + \lambda_{\min}(H_y \circ H_y) \\
    &\quad - | \lambda_{\min}(H_x \circ H_x) - \lambda_{\min}(H_y \circ H_y) | \\
    &\leq 2 \sigma^2_{\min} (B_{xy}) + 2 \lambda_{\min}(H_y \circ H_y),
\end{align*}
where we emphasize that $B_{yx}$ is the adjoint to $B_{xy}$ and hence $\lambda(B_{yx} \circ B_{xy}) = \lambda(B_{xy} \circ B_{yx} ) = \sigma^2(B_{xy}) =\sigma^2(B_{yx})$.

Hence, $\delta < \frac{b}{a}<  \sigma^2_{\min} (B_{xy}) + \lambda_{\min}(H_y \circ H_y)$ and $T_y = H_y \circ H_y + B_{xy} \circ B_{yx} - \delta \, \id$ is invertible, because $\lambda_{\min}(T_y) \geq \sigma^2_{\min} (B_{xy}) + \lambda_{\min}(H_y \circ H_y) - \delta >0$ by Weyl's inequality.
%
%
Thus, \eqref{lemma_la_eq0} gives $v = - T_y^{-1} \circ (H_y \circ B_{xy} + B_{xy} \circ H_x) [u]$. Substituting this expression for $v$ into the first equation of \eqref{lemma_la_eq0} yields
\begin{equation}
    \Big(T_x - \big(  H_x \circ B_{yx} + B_{yx} \circ H_y \big) \circ T_y^{-1} \circ \big( H_y \circ B_{xy} + B_{xy} \circ H_x \big) \Big) [u] = 0. \label{lemma_la_eq3}
\end{equation}
Nevertheless, we can verify when $\delta < \frac{b}{a}$, \eqref{lemma_la_eq3} does not have any nontrivial solution for $u$, which gives a contradiction. Specifically, we show the following inequality is always satisfied under the condition on $\delta$, 
\begin{align}
    (\lambda_{\min}(H_y \circ H_y) + \sigma_{\min}^2(B_{xy}) - \delta )^{-1} &\sigma_{\max}^2(B_{xy})  (\| H_x\|_x + \| H_y\|_y)^2 \nonumber\\
    &< \lambda_{\min}(H_x \circ H_x) + \sigma_{\min}^2(B_{xy}) - \delta, \label{lemma_la_eq4}
\end{align}
which violates \eqref{lemma_la_eq3} for any $u \neq 0,$ because \eqref{lemma_la_eq4} would imply that 
\begin{align*}
    \lambda_{\min}\left(\Big(T_x - \big(  H_x \circ B_{yx} + B_{yx} \circ H_y \big) \circ T_y^{-1} \circ \big( H_y \circ B_{xy} + B_{xy} \circ H_x \big) \Big)\right) > 0,
\end{align*}
subsequently \eqref{lemma_la_eq3} implies $u=0,$ hence, $\xi=0,$ a contradiction.
It remains to show that under $\delta < \frac{b}{a}$, \eqref{lemma_la_eq4} is satisfied. That is, the roots of \eqref{lemma_la_eq4} are given by $\frac{1}{2}(a \pm \sqrt{a^2 - 4b})$. We have shown that $\delta < \frac{b}{a} < \frac{1}{2}(a - \sqrt{a^2 -4b})$. This implies \eqref{lemma_la_eq4} is always satisfied and results in a contradiction. Hence, $\delta \geq \frac{b}{a}$, which completes the proof. 
\end{proof}

\section{Analysis of RHM with conjugate gradient and trust-region update steps}
\label{RHM_CG_TR}
We provide the details on convergence analysis of minimizing the Riemannian Hamiltonian with the Riemannian conjugate gradient and trust-region methods, i.e., we consider Algorithm \ref{RHM} with the update step $\xi(p_t)$ computed as conjugate gradient direction and trust-region step.

\subsection{RHM with conjugate gradient (RHM-CG)}

\begin{theorem}[Linear convergence of RHM-CG]
\label{linear_convergence_RHM_CG}
Under the same settings as in Theorem \ref{theorem_convergence_RHGD}, consider Algorithm \ref{RHM} with conjugate gradient direction $\xi(p_t)$ where $\beta_t$ (used in update) and $\eta_t$ are chosen such that $\langle \xi(p_t), -\grad \gH(p_t) \rangle \geq c \| \grad\gH(p_t)\|^2_{p_t}$ for some $c > 0$ and the Armijo condition (Definition \ref{def_linesearch_condition}) is satisfied. Let $\tilde{\eta} = \min_{i = 0,...,t} \eta_i$. Then, iterates $p_t$ satisfy $\| \grad f(p_t) \|^2_{p_t} \leq (1 - 2 r_1 \tilde{\eta} c\delta)^t \| \grad f(p_0) \|^2_{p_0}$.
\end{theorem}

\begin{proof}
From the Armijo condition, we have for the stepsize $\eta_t$,
\begin{align*}
    \gH(p_{t+1}) - \gH(p_t) &\leq r_1 \eta_t \langle \grad \gH(p_t), \zeta(p_t) \rangle \\
    &\leq -r_1 \eta_t c \| \grad \gH(p_t) \|^2_{p_t} \leq -2 r_1 \eta_t c \delta \gH(p_t) \leq -2r_1 \tilde{\eta} c\delta \gH(p_t),
\end{align*}
where the last inequality follows from the definition of $\tilde{\eta}$ and $\gH(p_t) \geq 0$ for all $p_t$. Applying the result recursively completes the proof.
\end{proof}

We notice that the bound only requires a descent direction and a sufficient function decrease. Hence, we suspect a tighter bound exists when analyzing specific types of conjugate gradient (with different $\beta_t$ types).

We also highlight that most, if not all, types of conjugate gradient methods satisfy the conditions in Theorem \ref{linear_convergence_RHM_CG}. See more discussions in \cite{sato2021RCGnew}. As an example, consider the \textit{Fletcher-Reeves-type CG} \cite{fletcher1964function} with $\beta_t = \frac{\| \grad \gH(p_t)\|^2_{p_t}}{\| \grad \gH(p_{t-1}) \|^2_{p_{t-1}}}$. If the stepsize $\eta_t$ is chosen to satisfy the strong Wolfe conditions (Definition \ref{def_linesearch_condition}) with $0 < r_1< r_2 < 1/2$, then from \cite[Lemma~4.1]{sato2021riemannian}, the conditions in Theorem \ref{linear_convergence_RHM_CG} are satisfied with $\langle \xi(p_t), -\grad \gH(p_t) \rangle \geq \frac{1-2r_2}{1-r_2} \| \grad \gH(p_t) \|^2$.

\subsection{RHM with trust-region (RHM-TR)}
For the Riemannian trust-region (TR) method, the update step $\xi(p_t)$ is computed by (approximately) solving the trust-region subproblem on the tangent space \cite{absil2009optimization}, i.e.,
\begin{equation}
    \xi(p_t) = \argmin_{\xi \in T_{p_t}\M: \| \xi \|_{p_t}\leq \Delta_t} \widehat{m}_{p_t}(\xi) = \gH(p_t) + \langle \grad \gH(p_t), \xi \rangle_{p_t} + \frac{1}{2} \langle H_t[\xi], \xi \rangle_{p_t}, \label{trust_region_subproblem}
\end{equation}
where $H_t: T_{p_t}\M \xrightarrow{} T_{p_t}\M$ is a self-adjoint linear operator that approximates the Hessian $\hess \gH(p_t)$. Depending on how much decrease is provided by the obtained direction, we either accept or reject the trust-region step and modify the radius $\Delta_t$.

\begin{theorem}[Convergence of RHM-TR]
Under the same settings as in Theorem \ref{theorem_convergence_RHGD} with $L = L_0L_1 + L_2^2$, consider Algorithm \ref{RHM} with $\xi(p_t)$ given by solving \eqref{trust_region_subproblem} with truncated conjugate gradient. Assume further that $\| H_t - \hess \gH(p_t) \|_{p_t} \leq L_H\| \grad \gH(p_t) \|_{p_t}$. 
Let $c = \min_{i=0,...,t} \frac{\Delta_i}{L_0L_1}$ and $\widetilde{L} = L_H L_0 L_1 + L$. Then, the iterates $p_t$ satisfy $\| \grad f(p_t) \|^2_{p_t} \leq \big(1 - \frac{1}{2}\min\{c, {1}/{\widetilde{L}} \} \rho'\delta \big)^t \| \grad f(p_0) \|^2_{p_0}$. 

Under an additional Lipschitzness condition on $\nabla^2 \widehat{\gH}_{p}$, we can show around the global minima $p^*$, there exists $\theta > 0, T > 0$ such that for all $t > T$, the convergence is superlinear with $d(p_{t+1}, p^*) \leq \theta d^2(p_t, p^*)$. 
\end{theorem}

\begin{proof}
First from Assumption \ref{smoothness_assumption}, $\| \grad \gH(p_t) \|_{p_t} = \| \hess f(p_t) [\grad f(p_t) ]\|_{p_t} \leq L_1 L_0$ and the operator norm of $H_t$ is bounded as 
\begin{equation*}
    \| H_t \|_{p_t} \leq \| H_t - \hess \gH(p_t) \|_{p_t}  + \| \hess \gH(p_t) \|_{p_t} \leq L_H  L_0 L_1 + L.
\end{equation*}
Also, the trust-region direction $\xi(p_t)$ returned by the truncated conjugate gradient method satisfies a so-called \textit{Cauchy decrease inequality} \cite[eq.~(7.14)]{absil2009optimization}, which gives
\begin{align*}
    \widehat{m}_{p_t}(0) - \widehat{m}_{p_t}(\xi(p_t)) &\geq \frac{1}{2} \| \grad \gH(p_t) \|_{p_t} \min \Big\{ \Delta_t, \frac{\| \grad \gH(p_t) \|_{p_t}}{\| H_t \|_{p_t}} \Big\} \\
    &\geq \frac{1}{2} \| \grad \gH(p_t) \|_{p_t} \min \Big\{ c \| \grad \gH(p_t) \|_{p_t}, \frac{\| \grad \gH(p_t) \|_{p_t}}{\| H_t \|_{p_t}} \Big\} \\
    &\geq \frac{1}{2} \min \Big\{ c, \frac{1}{L_H  L_0 L_1 + L} \Big\} \|\grad \gH(p_t) \|^2_{p_t} \\
    &\geq \frac{1}{2} \min \Big\{ c, \frac{1}{L_H  L_0 L_1 + L} \Big\}  \delta \gH(p_t).
\end{align*}
where the second inequality follows from the definition of $c$ and Assumption \ref{smoothness_assumption} where 
Furthermore, from the acceptance rule,
\begin{align*}
    \gH(p_{t+1}) - \gH(p_t) \leq \rho' \big( \widehat{m}_{p_t}(\xi(p_t)) - \widehat{m}_{p_t}(0) \big) \leq - \frac{1}{2} \min \Big\{ c, \frac{1}{L_H  L_0 L_1 + L} \Big\} \rho' \delta  \gH(p_t).
\end{align*}
Hence, the linear convergence is proved by recursively applying the result. The superlinear convergence simply follows from \cite[Theorem~7.4.11]{absil2009optimization} around any local minima.
\end{proof}

\section{On geodesic quadratic bilinear optimization}
\label{geodesic_quad_bilinear_appendix}

{
We first show an important result on the orthogonality of the min-max Riemannian gradient and Riemannian gradient of the Riemannian Hamiltonian for any g-bilinear function on arbitrary manifolds. 

\begin{proposition}
\label{orthogo_g_bilinear}
Let $f(x, y)$ be a g-bilinear function on $\M = \M_x \times \M_y$. Denote $G(p) = (\grad_x f(x,y), - \grad_y f(x,y)) \in T_p\M$ for $p = (x,y) \in \M$ as the min-max Riemannian gradient. Then for any $p \in \M$, we have $\langle G(p) , \grad \gH(p) \rangle_p = 0$ where $\gH(p) = \frac{1}{2} \| \grad f(p) \|_p^2$ is the Riemannian Hamiltonian of $f$. 
\end{proposition}
\begin{proof}
First it is aware that for any g-bilinear function, we have $\hess_x f(x,y) = \hess_y f(x,y) = 0$. Hence, from Proposition \ref{gradient_hamiltonian_prop} and \ref{riem_hess_prop}, we show
\begin{equation*}
    \grad \gH(p) = \hess f(p) [\grad f(p)] = 
    \begin{pmatrix}
    \grad_{yx}^2 f(x,y)[\grad_y f(x,y)] \\
    \grad_{xy}^2 f(x,y)[\grad_x f(x,y)]
    \end{pmatrix}.
\end{equation*}
Finally, we have
\begin{align*}
    \langle G(p), \grad \gH(p) \rangle_p &= \langle \grad_x f(x,y), \grad_{yx}^2 f(x,y)[\grad_y f(x,y)] \rangle_x \\
    &\quad + \langle - \grad_y f(x,y)) , \grad_{xy}^2 f(x,y)[\grad_x f(x,y)] \rangle_y \\
    &= \langle \grad_y f(x,y), \grad_{xy}^2 f(x,y)[\grad_x f(x,y)] \rangle_y \\
    &\quad + \langle - \grad_y f(x,y)) , \grad_{xy}^2 f(x,y)[\grad_x f(x,y)] \rangle_y = 0
\end{align*}
where we apply Proposition \ref{symmetry_cross_derivative_prop}. 
\end{proof}
}

\begin{proof}[Proof of Proposition \ref{g_convex_concave_prop}]
First, the expression of geodesic curve connecting any $\bX_0, \bX_1 \in \M_{\rm SPD}$ is given by $\gamma(t) = \bX_0^{1/2} (\bX_0^{-1/2} \bX_1 \bX_0^{-1/2})^t \bX_0^{1/2}$. From \cite[Proposition 5.7]{vishnoi2018geodesic}, we see $\log\det(\bX)$ is geodesic linear. That is, for the geodesic $\gamma(t)$ joining $\bX_0, \bX_1$ with $\gamma(0) = \bX_0, \gamma(1) = \bX_1$, it can be shown that $\log\det(\gamma(t)) = (1-t) \log\det(\bX_0) + t \log\det(\bX_1)$. It remains to show $(\log\det(\bX))^2$ is geodesic convex, which is equivalent to show $\frac{d^2 (\log\det(\gamma(t)))^2}{dt^2} \geq 0$ for all $t \in [0,1]$ (second order characterization of geodesic convexity \cite{vishnoi2018geodesic}). Specifically, we show 
\begin{align}\label{eq:double_derivative_logdet}
    \frac{d^2 (\log\det(\gamma(t)))^2}{dt^2} = 2 (\log\det(\bX_1) - \log\det(\bX_0))^2 \geq 0.
\end{align}
The equality in (\ref{eq:double_derivative_logdet})
holds when $\bX_0 \neq \bX_1$ while $\det(\bX_0) = \det(\bX_1)$ and hence $\frac{d^2 (\log\det(\gamma(t)))^2}{dt^2} > 0$ is not always satisfied. Similar arguments hold for g-concavity with respect to $\bY$. 
\end{proof}

\begin{proof}[Proof of Proposition \ref{PL_logdet_bilinear_quadratic}]
The Riemannian gradient of $f$ is derived as 
\begin{align*}
    \grad_\bX f(\bX, \bY) &= \big( c_l \log\det(\bY) + 2 c_q \log\det(\bX) \big) \bX\\
    \grad_\bY f(\bX, \bY) &= \big( c_l \log\det(\bX) - 2 c_q \log\det(\bY) \big) \bY. 
\end{align*}
Under the affine-invariant metric, the Hamiltonian is given by 
$$\gH(\bX, \bY) = \frac{(4c_q^2 + c_l^2)d}{2} \Big( (\log\det(\bX))^2 + (\log\det(\bY))^2 \Big).$$
The gradient of Hamiltonian is given by $\grad_\bX \gH(\bX, \bY) = (4c_q^2 + c_l^2)d \allowbreak\log\det(\bX) \bX$ and $\grad_\bY \gH(\bX, \bY) \allowbreak = (4c_q^2 + c_l^2) d \allowbreak \log\det(\bY) \bY$. Next, we verify
\begin{align*}
    &\frac{1}{2}\Big( \| \grad_\bX \gH(\bX, \bY) \|^2_\bX + \| \grad_\bY \gH(\bX, \bY) \|^2_\bY \Big) \\
    &= \frac{(4c_q^2 + c_l^2)^2d^3}{2} \Big( (\log\det(\bX))^2 +  (\log\det(\bY))^2\Big) \\
    &= (4c_q^2 + c_l^2)d^2 \gH(\bX, \bY).
\end{align*}
In addition, from the definition of global saddle point in \eqref{global_saddle_pt}, the pair $(\bX^*, \bY^*)$ where $\det(\bX^*) = \det(\bY^*) = 1$, satisfies $f(\bX^*, \bY^*) = 0$. Thus, we have 
$$f(\bX^*, \bY) = -c_q(\log\det(\bY))^2  \leq f(\bX^*, \bY^*) \leq c_q (\log\det(\bX))^2  = f(\bX, \bY^*)$$
for all $\bX, \bY \in \sS_{++}^d$. Hence, the proof is complete. 
\end{proof}

{

Finally, we show that the geodesic-bilinear problem does not satisfy the min-max Riemannian PL condition on the function $f$. To this end, we first need to define the Riemannian min-max PL condition below.

\begin{definition}[Riemannian min-max PL condition]
\label{RPL_minmax}
For a min-max problem $\min_{x\in \M_x} \max_{y \in \M_y} f(x,y)$, the objective satisfies the Riemannian min-max PL condition if for a global saddle point $(x^*, y^*)$, there exists a constant $\delta > 0$ such that
\begin{align*}
    \frac{1}{2} \| \grad_x f(x', y) \|_{x'}^2 &\geq \delta \big( f(x', y) - f(x^*, y) \big), \quad \forall y \in \M, \\
    \frac{1}{2} \| \grad_y f(x, y') \|_{y'}^2 &\geq \delta \big( f(x, y^*) - f(x, y') \big), \quad \forall x \in \M.
\end{align*}
\end{definition}

Definition \ref{RPL_minmax} is equivalent to stating that the objective $f(x,y)$ satisfies the Riemannian PL in $x$ and $-f(x,y)$ satisfies the Riemannian PL in $y$. Such definition is natural as it includes geodesic strongly convex strongly concave functions as special cases.

\begin{lemma}
\label{lemma_RPL}
The g-bilinear function $f(\bX, \bY) = \log\det(\bX) \log\det(\bY)$ does not satisfy Definition \ref{RPL_minmax}.
\end{lemma}
\begin{proof}
We show  the case for $\bX$. A similar statement also holds for $\bY$. As the global saddle point $(\bX^*, \bY^*)$ satisfies $\det(\bX^*) = \det(\bY^*) = 1$, we have $f(\bX^*, \bY) = 0$. In addition, the Riemannian gradient is $\grad_\bX f(\bX', \bY) = \bX' \log\det(\bY)$ with $\| \grad_\bX f(\bX', \bY) \|_{\bX'}^2 = (\log\det(\bY))^2$. On the other hand, the right-hand-side in Definition \ref{RPL_minmax} is $f(\bX', \bY) -  f(\bX^*, \bY) = \log\det(\bX') \log\det(\bY)$. It is clear that $\frac{1}{2}(\log\det(\bY))^2$ is not necessarily larger than $\delta \log\det(\bX') \log\det(\bY)$ for $\delta > 0$ and for all $\bY \in \sS_{++}^d$. Hence, the claim follows.
\end{proof}
}

\section{Additional experiment results}\label{app:sec:experiments}

\subsection{Optimality gap for geodesic quadratic bilinear optimization}
We include additional convergence results in Fig. \ref{geodesic_quad_linear_disttoopt} on the optimality gap for the geodesic quadratic bilinear optimization problem in Section \ref{geodesic_quadratic_sect}.

\begin{figure*}[!t]
    \centering
    \subfloat[\texttt{$c_q = 0, c_l = 1$} ]{\includegraphics[scale=0.29]{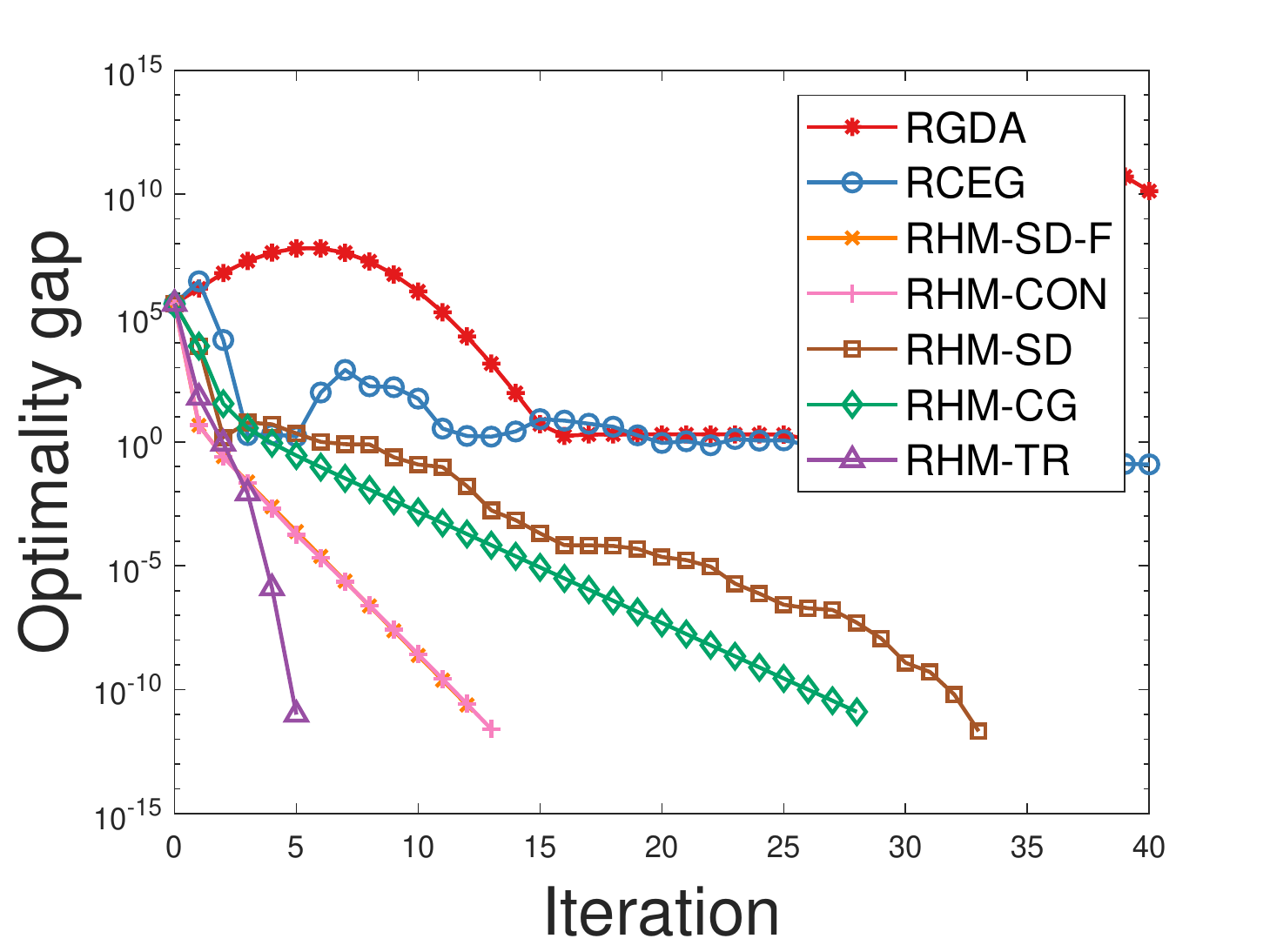}}
    \subfloat[\texttt{$c_q = 1, c_l = 0$} ]{\includegraphics[scale=0.29]{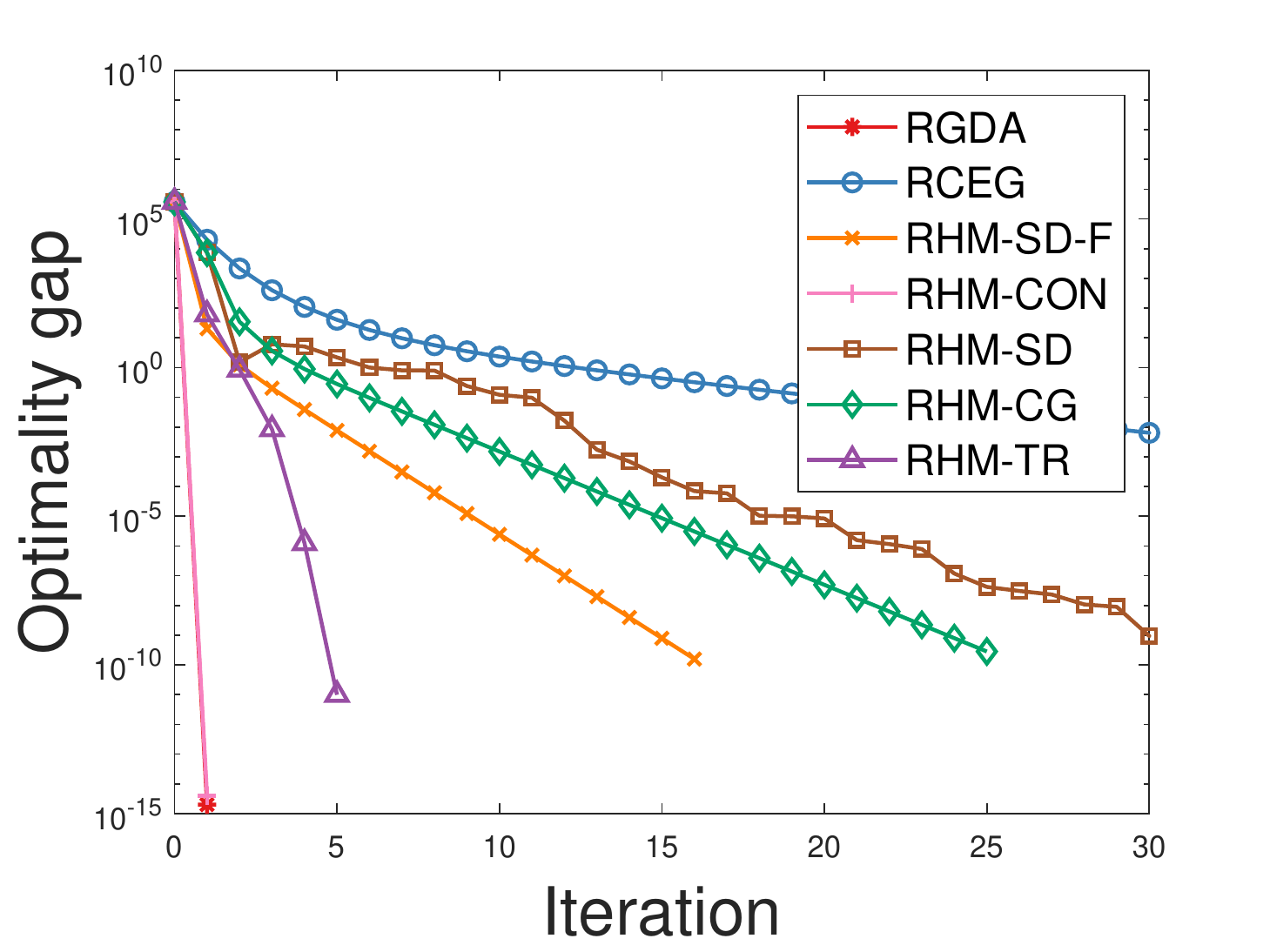}}
    \subfloat[\texttt{$c_q = 1, c_l = 0.1$} ]{\includegraphics[scale=0.29]{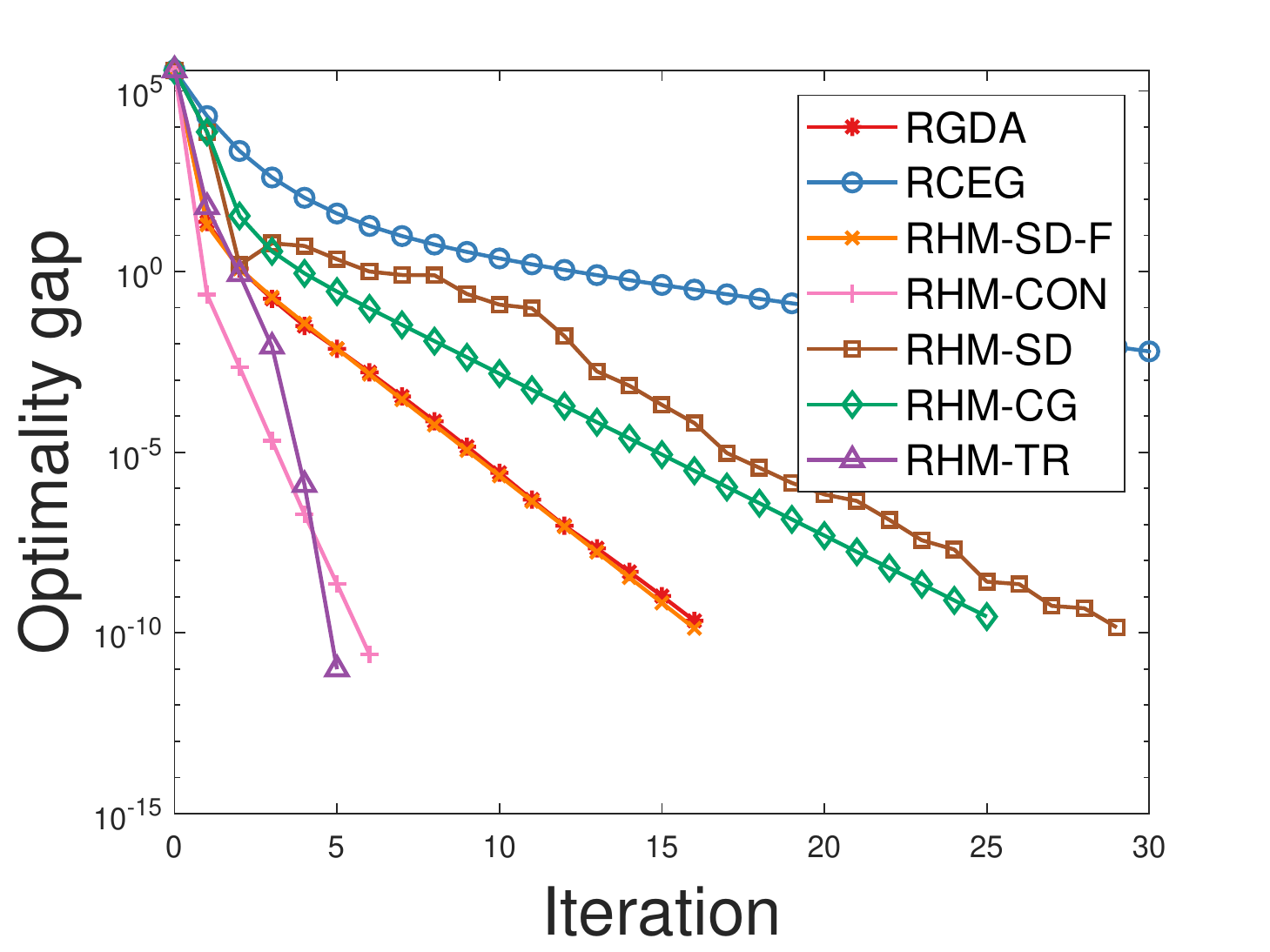}}\\
    \subfloat[\texttt{$c_q = 1, c_l = 1$} ]{\includegraphics[scale=0.29]{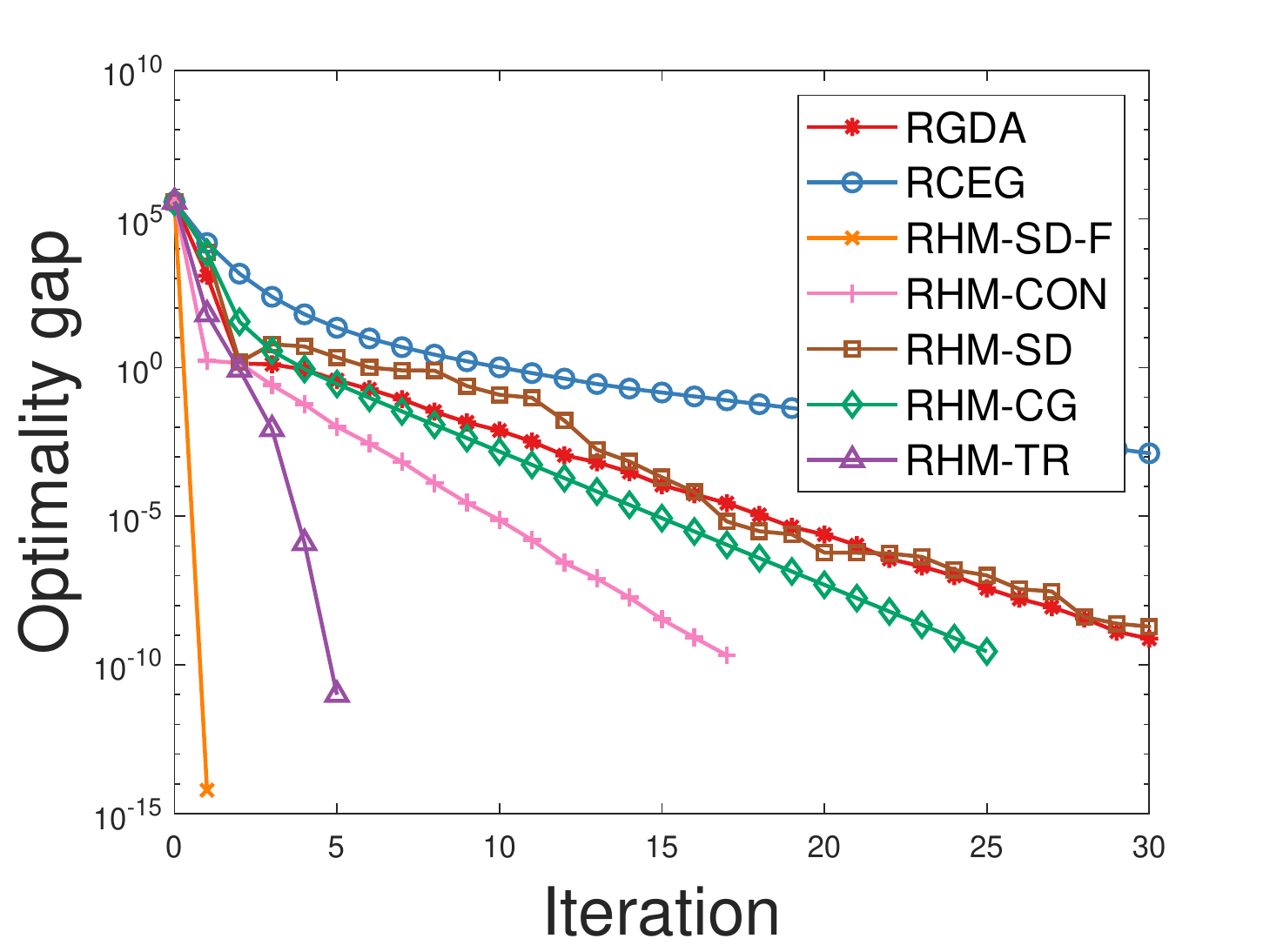}}
    \subfloat[\texttt{$c_q = 1, c_l = 10$} ]{\includegraphics[scale=0.29]{Figures/spd_quadratic_disttoopt_10.pdf}}
    \caption{Experiments comparing optimality gap on the geodesic quadratic blinear problem \eqref{g_bilinear_quadratic} with $d = 30$, under different weights $c_q, c_l$. We observe that the RHM algorithms show a good rate of convergence in all the settings. In particular, RHM-SD-F and RHM-CON significantly outperforms RCEG in all the settings indicating better theoretical rates.}
    \label{geodesic_quad_linear_disttoopt}
\end{figure*}

\subsection{Results of RHM-SGD for orthonormal GAN}
We show the sample collapse of RHM-SGD in Fig. \ref{rhm_puresgd_gan}.

\begin{figure*}[!t]
    \centering
    \subfloat{\includegraphics[scale=0.29]{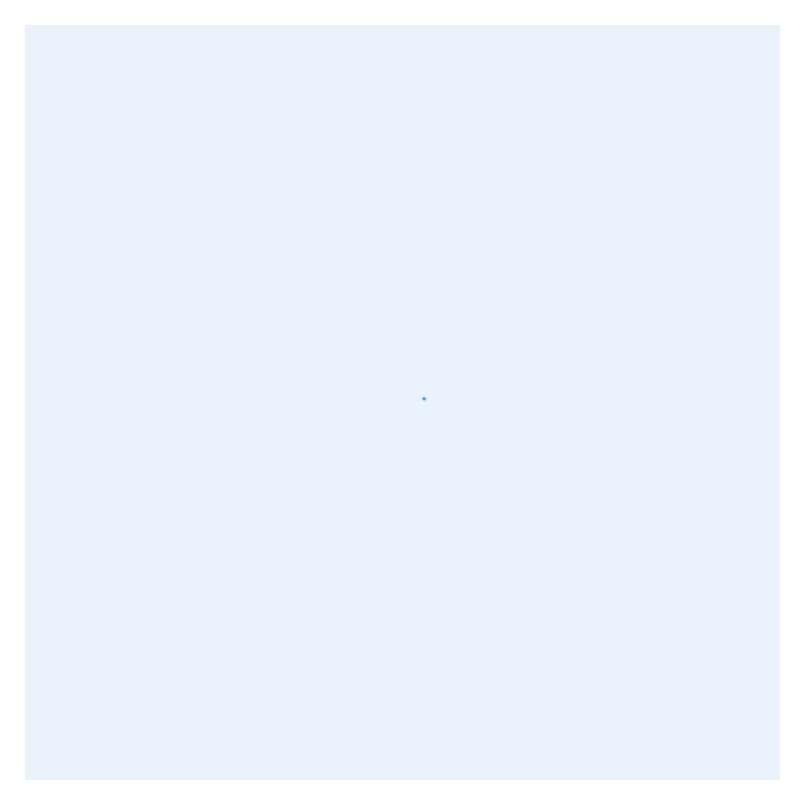}}
    \subfloat{\includegraphics[scale=0.29]{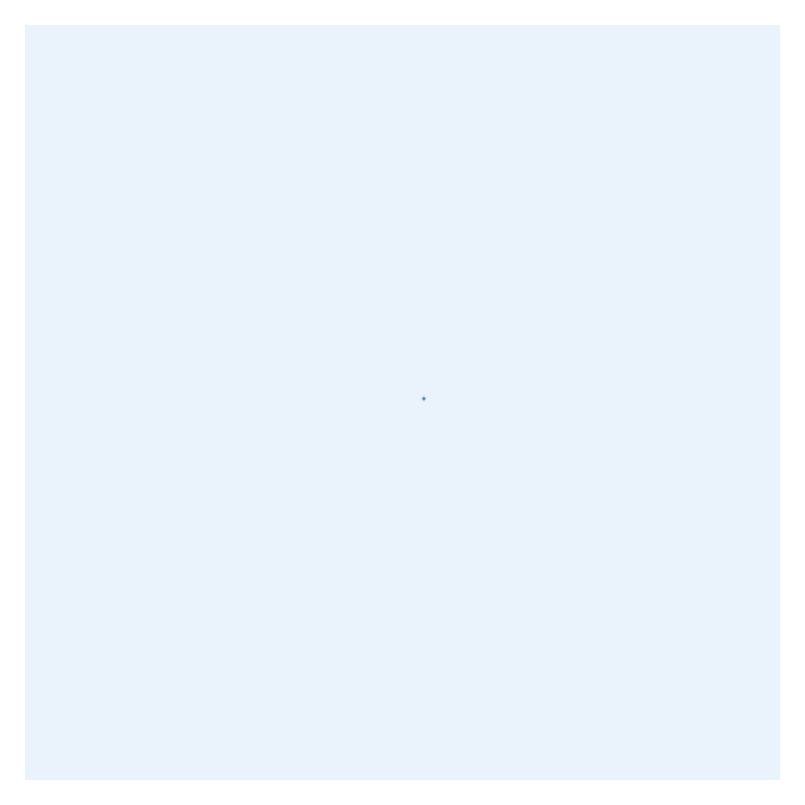}}
    \subfloat{\includegraphics[scale=0.29]{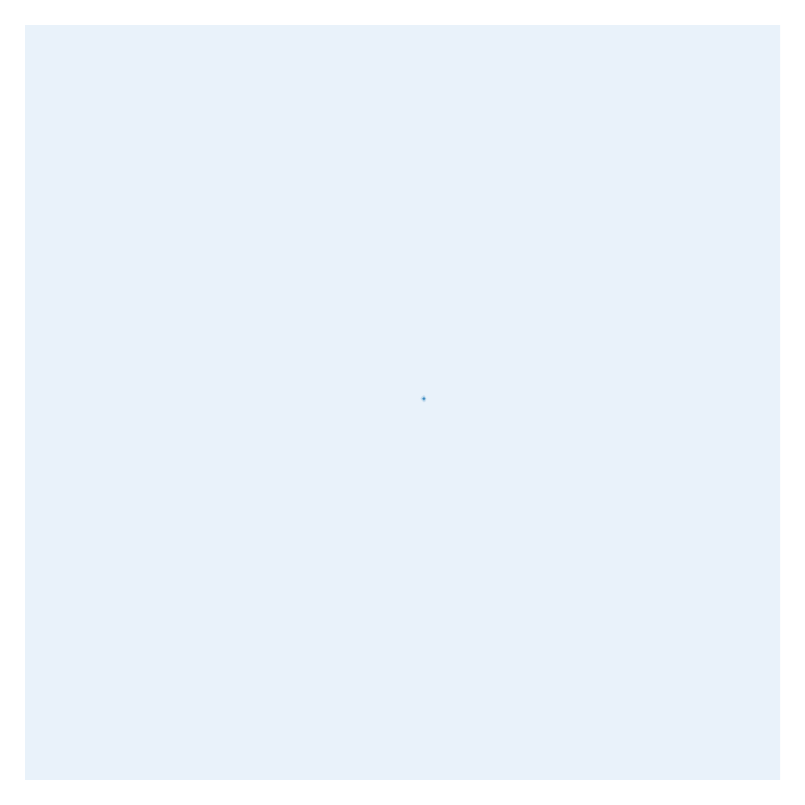}}
    \caption{Generated samples from \texttt{RHM-SGD} at $1, 2, 3 \times 10^4$ iterations from left to right. We see although RHM-SGD converges in Hamiltonian, the generated samples collapse to a single point (zoom the figures to see the single point).}
    \label{rhm_puresgd_gan}
\end{figure*}

\subsection{Trace-logarithm bilinear optimization}
We consider the `bilinear' example of \cite{zhang2022minimax} on the symmetric positive definite (SPD) manifold (endowed with the affine-invariant metric), i.e., $$f(\bX, \bY) = \trace({\rm Log}_{\bX}(\bX_0) {\rm Log}_{\bY}(\bY_0))$$ for $\bX_0, \bY_0 \in \sS_{++}^d$, where ${\rm Log}_\bM(\bM') = \{ \bM \log(\bM^{-1} \bM') \}_{\rm S}$ is the logarithm map on the SPD manifold with $\log(\cdot)$ representing the matrix principal logarithm. When the manifold is simply the Euclidean space, the logarithm map reduces to ${\rm Log}_{\bM}(\bM') = \bM'- \bM$. Hence, this resembles a bilinear problem on the manifold. 

For the experiment setting, we consider $\gamma = 0.2$ for RHM-CON and $\bX_0 = \bY_0 = \bI$. 
The convergence results are shown in Fig. \ref{tracelog_plot}, where we notice that both RGDA and RCEG oscillate while all the RHM algorithms are convergent. RHM-CON and RHM-SD-F converge rapidly initially but subsequently have a slow rate of convergence due to the hardness of the problem. RHM-CG, on the other hand, has a faster rate of convergence.

\begin{figure*}[!t]
\centering
{\includegraphics[scale=0.28]{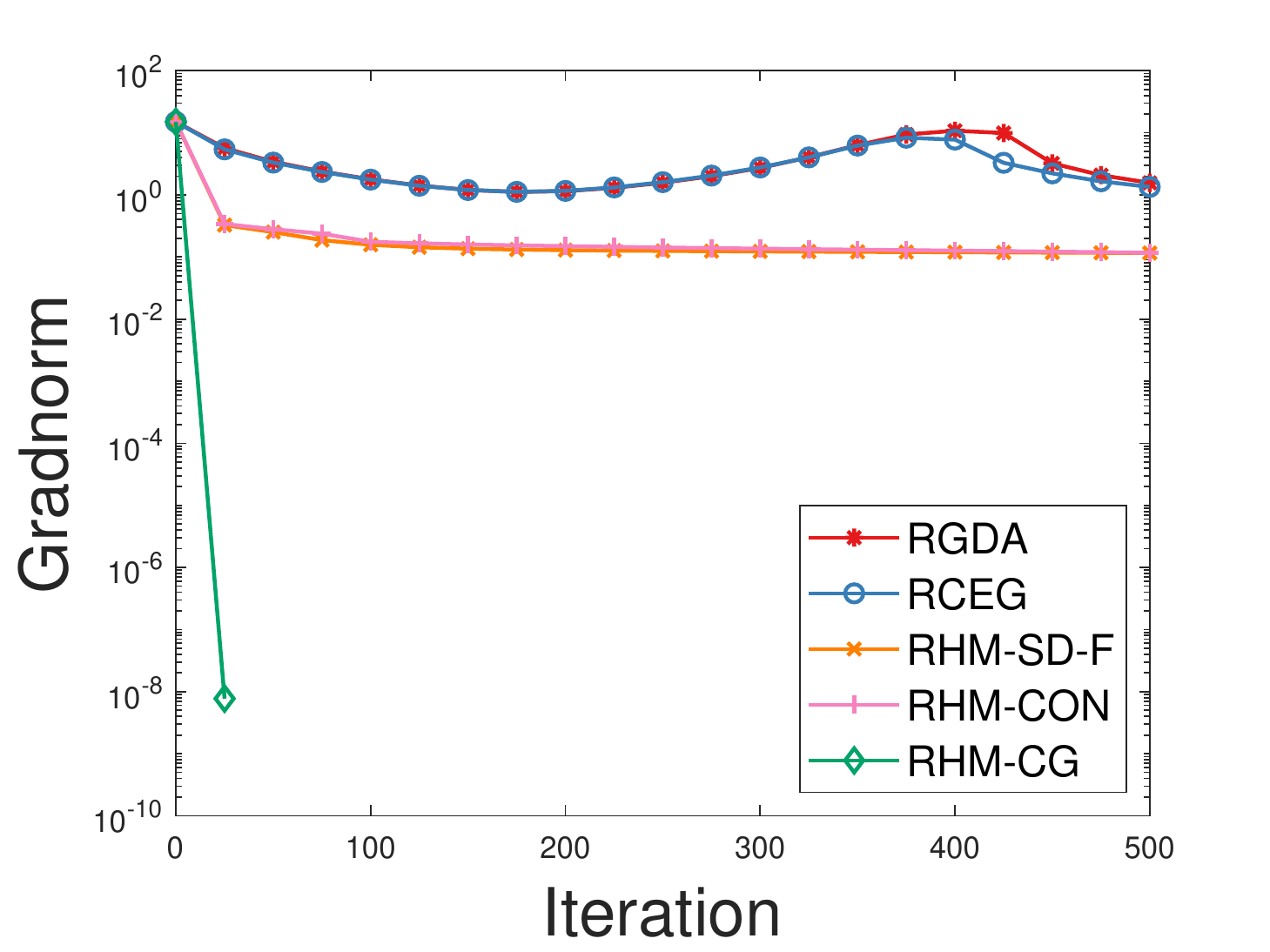}}
\caption{Trace-logarithm bilinear problem on the SPD manifold. RGDA and RCEG diverge while RHM algorithms are convergent (though RHM with steepest descent has a slower rate of convergence).}
\label{tracelog_plot}
\end{figure*}

\section*{Acknowledgments}
Pawan Kumar acknowledges the support of Microsoft Academic Partnership Grant (MAPG) 2021.

\bibliographystyle{siamplain}
\bibliography{references_siopt_final}

\end{document}

%% file: siamshared.tex
\usepackage{lipsum}
\usepackage{amsfonts}
\usepackage{graphicx}
\usepackage{epstopdf}
\usepackage{algorithmic}
\ifpdf
  \DeclareGraphicsExtensions{.eps,.pdf,.png,.jpg}
\else
  \DeclareGraphicsExtensions{.eps}
\fi

\newsiamremark{remark}{Remark}
\newsiamremark{hypothesis}{Hypothesis}
\crefname{hypothesis}{Hypothesis}{Hypotheses}
\newsiamthm{claim}{Claim}

\usepackage{subcaption}

\usepackage[shortlabels]{enumitem}
\usepackage{array}

\usepackage[symbol]{footmisc} 

\def\M{\mathcal{M}}

\def\D{\mathrm{D}}

\def\bzero{{\mathbf 0}}

\def\be{{\mathbf e}}

\def\bp{{\mathbf p}}

\def\bu{\mathbf u}
\def\bv{\mathbf v}
\def\bw{{\mathbf w}}
\def\bx{{\mathbf x}}
\def\by{{\mathbf y}}
\def\bz{{\mathbf z}}

\def\bA{\mathbf A}
\def\bB{\mathbf B}

\def\bG{\mathbf G}
\def\bH{\mathbf H}
\def\bI{{\mathbf I}}
\def\bJ{{\mathbf J}}

\def\bM{\mathbf M}

\def\bU{{\mathbf U}}
\def\bV{{\mathbf V}}
\def\bW{{\mathbf W}}
\def\bX{{\mathbf X}}
\def\bY{{\mathbf Y}}

\def\sR{{\mathbb R}}
\def\sS{{\mathbb S}}
\def\sE{{\mathbb E}}

\def\gD{{\mathcal{D}}}
\def\gM{{\mathcal{M}}}

\def\gH{{\mathcal{H}}}

\def\gS{{\mathcal{S}}}
\def\gU{{\mathcal{U}}}

\def\id{{\mathrm{id}}}

\def\grad{{\mathrm{grad}}}
\def\hess{{\mathrm{Hess}}}

\newcommand{\trace}{\mathrm{tr}}

\newtheorem{assumption}{Assumption}

\def\bGamma{{\mathbf \Gamma}}

\DeclareMathOperator*{\argmin}{arg\,min}

\def\bnabla{{\boldsymbol \nabla}}

\makeatletter
\newcommand{\thickhline}{%
    \noalign {\ifnum 0=`}\fi \hrule height 1pt
    \futurelet \reserved@a \@xhline
}
\newcolumntype{"}{@{\hskip\tabcolsep\vrule width 1pt\hskip\tabcolsep}}
\makeatother

\headers{Riemannian Hamiltonian methods}{A. Han, B. Mishra, P. Jawanpuria, P. Kumar, J. Gao}

\title{{Riemannian Hamiltonian methods \\ for min-max optimization on manifolds}}

\author{Andi Han\thanks{University of Sydney 
  (\email{andi.han@sydney.edu.au}, \email{junbin.gao@sydney.edu.au}).}
\and Bamdev Mishra\thanks{Microsoft India
  (\email{bamdevm@microsoft.com}, \email{pratik.jawanpuria@microsoft.com}).}
 \and Pratik Jawanpuria\footnotemark[2] , \\
  Pawan Kumar\thanks{IIIT Hyderabad India (\email{pawan.kumar@iiit.ac.in}).}
\and Junbin Gao\footnotemark[1]}

\usepackage{amsopn}


%% file: main.bbl
\begin{thebibliography}{10}

\bibitem{abernethy2019last}
{\sc J.~Abernethy, K.~A. Lai, and A.~Wibisono}, {\em Last-iterate convergence
  rates for min-max optimization: Convergence of hamiltonian gradient descent
  and consensus optimization}, in {International Conference on Algorithmic
  Learning Theory}, vol.~132, PMLR, 2021, pp.~3--47.

\bibitem{absil2007trust}
{\sc P.-A. Absil, C.~G. Baker, and K.~A. Gallivan}, {\em Trust-region methods
  on {R}iemannian manifolds}, Found. Comput. Math., 7 (2007), pp.~303--330.

\bibitem{absil2009optimization}
{\sc P.-A. Absil, R.~Mahony, and R.~Sepulchre}, {\em Optimization algorithms on
  matrix manifolds}, {Princeton University Press}, 2009.

\bibitem{adolphs2019local}
{\sc L.~Adolphs, H.~Daneshmand, A.~Lucchi, and T.~Hofmann}, {\em Local saddle
  point optimization: A curvature exploitation approach}, in International
  Conference on Artificial Intelligence and Statistics, PMLR, 2019,
  pp.~486--495.

\bibitem{aflalo11a}
{\sc J.~Aflalo, A.~Ben-Tal, C.~Bhattacharyya, J.~S. Nath, and S.~Raman}, {\em
  Variable sparsity kernel learning}, J. Mach. Learn. Res., 12 (2011),
  pp.~565--592.

\bibitem{agarwal2021adaptive}
{\sc N.~Agarwal, N.~Boumal, B.~Bullins, and C.~Cartis}, {\em Adaptive
  regularization with cubics on manifolds}, Math. Program., 188 (2021),
  pp.~85--134.

\bibitem{arjovsky2017wasserstein}
{\sc M.~Arjovsky, S.~Chintala, and L.~Bottou}, {\em Wasserstein generative
  adversarial networks}, in International Conference on Machine Learning, PMLR,
  2017, pp.~214--223.

\bibitem{balduzzi2018mechanics}
{\sc D.~Balduzzi, S.~Racaniere, J.~Martens, J.~Foerster, K.~Tuyls, and
  T.~Graepel}, {\em The mechanics of n-player differentiable games}, in
  {International Conference on Machine Learning}, PMLR, 2018, pp.~354--363.

\bibitem{bansal2018can}
{\sc N.~Bansal, X.~Chen, and Z.~Wang}, {\em Can we gain more from orthogonality
  regularizations in training deep networks?}, in Advances in Neural
  Information Processing Systems, vol.~31, 2018.

\bibitem{bergmann2022manopt}
{\sc R.~Bergmann}, {\em Manopt. jl: Optimization on manifolds in julia},
  Journal of Open Source Software, 7 (2022), p.~3866.

\bibitem{bertsekas2014constrained}
{\sc D.~P. Bertsekas}, {\em Constrained optimization and Lagrange multiplier
  methods}, {Academic Press}, 2014.

\bibitem{bhatia2009positive}
{\sc R.~Bhatia}, {\em Positive definite matrices}, {Princeton University
  Press}, 2009.

\bibitem{boumal2015riemannian}
{\sc N.~Boumal}, {\em Riemannian trust regions with finite-difference {H}essian
  approximations are globally convergent}, in {International Conference on
  Geometric Science of Information}, Springer, 2015, pp.~467--475.

\bibitem{boumal2020introduction}
{\sc N.~Boumal}, {\em An introduction to optimization on smooth manifolds},
  Available online, May, 3 (2020).

\bibitem{boumal2019global}
{\sc N.~Boumal, P.-A. Absil, and C.~Cartis}, {\em Global rates of convergence
  for nonconvex optimization on manifolds}, IMA J. Numer. Anal., 39 (2019),
  pp.~1--33.

\bibitem{boumal2014manopt}
{\sc N.~Boumal, B.~Mishra, P.-A. Absil, and R.~Sepulchre}, {\em Manopt, a
  {Matlab} toolbox for optimization on manifolds}, J. Mach. Learn. Res., 15
  (2014), pp.~1455--1459.

\bibitem{brock2018large}
{\sc A.~Brock, J.~Donahue, and K.~Simonyan}, {\em Large scale {GAN} training
  for high fidelity natural image synthesis}, in {International Conference on
  Learning Representations}, 2018.

\bibitem{chasnov2020convergence}
{\sc B.~Chasnov, L.~Ratliff, E.~Mazumdar, and S.~Burden}, {\em Convergence
  analysis of gradient-based learning in continuous games}, in Uncertainty in
  Artificial Intelligence, PMLR, 2020, pp.~935--944.

\bibitem{cogswell2015reducing}
{\sc M.~Cogswell, F.~Ahmed, R.~Girshick, L.~Zitnick, and D.~Batra}, {\em
  Reducing overfitting in deep networks by decorrelating representations}, in
  {International Conference on Learning Representations}, 2016.

\bibitem{douik2019manifold}
{\sc A.~Douik and B.~Hassibi}, {\em Manifold optimization over the set of
  doubly stochastic matrices: A second-order geometry}, IEEE Trans. Signal
  Process., 67 (2019), pp.~5761--5774.

\bibitem{el1997robust}
{\sc L.~El~Ghaoui and H.~Lebret}, {\em Robust solutions to least-squares
  problems with uncertain data}, SIAM J. Matrix Anal. Appl., 18 (1997),
  pp.~1035--1064.

\bibitem{fletcher1964function}
{\sc R.~Fletcher and C.~M. Reeves}, {\em Function minimization by conjugate
  gradients}, The computer journal, 7 (1964), pp.~149--154.

\bibitem{extragradient1976}
{\sc K.~G.}, {\em The extragradient method for finding saddle points and other
  problems}, Ekonomika i Matematicheskie Metody, 12 (1976), pp.~747--756.

\bibitem{goodfellow2014generative}
{\sc I.~Goodfellow, J.~Pouget-Abadie, M.~Mirza, B.~Xu, D.~Warde-Farley,
  S.~Ozair, A.~Courville, and Y.~Bengio}, {\em Generative adversarial nets}, in
  Advances in Neural Information Processing Systems, vol.~27, 2014.

\bibitem{han2021improved}
{\sc A.~Han and J.~Gao}, {\em Improved variance reduction methods for
  {R}iemannian non-convex optimization}, IEEE Trans. Pattern Anal. Mach.
  Intell.,  (2021).

\bibitem{han2021riemannian}
{\sc A.~Han, B.~Mishra, P.~Jawanpuria, and J.~Gao}, {\em On {{R}iemannian}
  optimization over positive definite matrices with the {Bures-Wasserstein}
  geometry}, in Advances in Neural Information Processing Systems, vol.~34,
  2021.

\bibitem{horev2016geometry}
{\sc I.~Horev, F.~Yger, and M.~Sugiyama}, {\em Geometry-aware principal
  component analysis for symmetric positive definite matrices}, in {Asian
  Conference on Machine Learning}, PMLR, 2016, pp.~1--16.

\bibitem{huang2020gradient}
{\sc F.~Huang, S.~Gao, and H.~Huang}, {\em Gradient descent ascent for min-max
  problems on {R}iemannian manifolds}, arXiv:2010.06097,  (2020).

\bibitem{huang2018orthogonal}
{\sc L.~Huang, X.~Liu, B.~Lang, A.~W. Yu, Y.~Wang, and B.~Li}, {\em Orthogonal
  weight normalization: Solution to optimization over multiple dependent
  {S}tiefel manifolds in deep neural networks}, in Thirty-Second AAAI
  Conference on Artificial Intelligence, 2018.

\bibitem{huang2021riemannian}
{\sc M.~Huang, S.~Ma, and L.~Lai}, {\em A {R}iemannian block coordinate descent
  method for computing the projection robust {W}asserstein distance}, in
  {International Conference on Machine Learning}, PMLR, 2021, pp.~4446--4455.

\bibitem{huang2015riemannian}
{\sc W.~Huang, P.-A. Absil, and K.~A. Gallivan}, {\em A {R}iemannian symmetric
  rank-one trust-region method}, Math. Program., 150 (2015), pp.~179--216.

\bibitem{huang2018roptlib}
{\sc W.~Huang, P.-A. Absil, K.~A. Gallivan, and P.~Hand}, {\em {ROPTLIB}: an
  object-oriented {C++} library for optimization on {Riemannian} manifolds},
  ACM Trans. Math. Software, 44 (2018), pp.~1--21.

\bibitem{jawanpuria15a}
{\sc P.~Jawanpuria, M.~Lapin, M.~Hein, and B.~Schiele}, {\em Efficient output
  kernel learning for multiple tasks}, in Advances in Neural Information
  Processing Systems, 2015.

\bibitem{jawanpuria18a}
{\sc P.~Jawanpuria and B.~Mishra}, {\em A unified framework for structured
  low-rank matrix learning}, in International Conference on Machine Learning,
  2018.

\bibitem{jawanpuria12a}
{\sc P.~Jawanpuria and J.~S. Nath}, {\em A convex feature learning formulation
  for latent task structure discovery}, in International Conference on Machine
  Learning, 2012.

\bibitem{jawanpuria11a}
{\sc P.~Jawanpuria, J.~S. Nath, and G.~Ramakrishnan}, {\em Efficient rule
  ensemble learning using hierarchical kernels}, in International Conference on
  Machine Learning, 2011.

\bibitem{jawanpuria15b}
{\sc P.~Jawanpuria, J.~S. Nath, and G.~Ramakrishnan}, {\em Generalized
  hierarchical kernel learning}, J. Mach. Learn. Res., 16 (2015), pp.~617--652.

\bibitem{jawanpuria21a}
{\sc P.~Jawanpuria, N.~T.~V. Satya~Dev, and B.~Mishra}, {\em Efficient robust
  optimal transport: formulations and algorithms}, in IEEE Conference on
  Decision and Control, 2021.

\bibitem{jin2020local}
{\sc C.~Jin, P.~Netrapalli, and M.~Jordan}, {\em What is local optimality in
  nonconvex-nonconcave minimax optimization?}, in International Conference on
  Machine Learning, PMLR, 2020, pp.~4880--4889.

\bibitem{jordan2022first}
{\sc M.~I. Jordan, T.~Lin, and E.-V. Vlatakis-Gkaragkounis}, {\em First-order
  algorithms for min-max optimization in geodesic metric spaces},
  arXiv:2206.02041,  (2022).

\bibitem{karimi2016linear}
{\sc H.~Karimi, J.~Nutini, and M.~Schmidt}, {\em Linear convergence of gradient
  and proximal-gradient methods under the {{P}olyak-{\L}ojasiewicz} condition},
  in Joint European Conference on Machine Learning and Knowledge Discovery in
  Databases, Springer, 2016, pp.~795--811.

\bibitem{kasai2018riemannian}
{\sc H.~Kasai, H.~Sato, and B.~Mishra}, {\em Riemannian stochastic recursive
  gradient algorithm}, in International Conference on Machine Learning, PMLR,
  2018, pp.~2516--2524.

\bibitem{kochurov2020geoopt}
{\sc M.~Kochurov, R.~Karimov, and S.~Kozlukov}, {\em Geoopt: {R}iemannian
  optimization in pytorch}, in ICML 2020 Workshop on Graph Representation
  Learning and Beyond, 2020.

\bibitem{lecun1998gradient}
{\sc Y.~LeCun, L.~Bottou, Y.~Bengio, and P.~Haffner}, {\em Gradient-based
  learning applied to document recognition}, Proceedings of the IEEE, 86
  (1998), pp.~2278--2324.

\bibitem{li2020stochastic}
{\sc J.~Li, K.~Balasubramanian, and S.~Ma}, {\em Stochastic zeroth-order
  {R}iemannian derivative estimation and optimization}, arXiv:2003.11238,
  (2020).

\bibitem{lin2020projection}
{\sc T.~Lin, C.~Fan, N.~Ho, M.~Cuturi, and M.~Jordan}, {\em Projection robust
  wasserstein distance and {R}iemannian optimization}, in Advances in Neural
  Information Processing Systems, vol.~33, 2020, pp.~9383--9397.

\bibitem{loizou2020stochastic}
{\sc N.~Loizou, H.~Berard, A.~Jolicoeur-Martineau, P.~Vincent,
  S.~Lacoste-Julien, and I.~Mitliagkas}, {\em Stochastic hamiltonian gradient
  methods for smooth games}, in {International Conference on Machine Learning},
  PMLR, 2020, pp.~6370--6381.

\bibitem{madras2018learning}
{\sc D.~Madras, E.~Creager, T.~Pitassi, and R.~Zemel}, {\em Learning
  adversarially fair and transferable representations}, in International
  Conference on Machine Learning, PMLR, 2018, pp.~3384--3393.

\bibitem{madry2017towards}
{\sc A.~Madry, A.~Makelov, L.~Schmidt, D.~Tsipras, and A.~Vladu}, {\em Towards
  deep learning models resistant to adversarial attacks}, in International
  Conference on Learning Representations, 2018.

\bibitem{mazumdar2019finding}
{\sc E.~V. Mazumdar, M.~I. Jordan, and S.~S. Sastry}, {\em On finding local
  nash equilibria (and only local nash equilibria) in zero-sum games},
  arXiv:1901.00838,  (2019).

\bibitem{meghwanshi2018mctorch}
{\sc M.~Meghwanshi, P.~Jawanpuria, A.~Kunchukuttan, H.~Kasai, and B.~Mishra},
  {\em {McTorch}, a manifold optimization library for deep learning},
  arXiv:1810.01811,  (2018).

\bibitem{mertikopoulos2018cycles}
{\sc P.~Mertikopoulos, C.~Papadimitriou, and G.~Piliouras}, {\em Cycles in
  adversarial regularized learning}, in Proceedings of the Annual ACM-SIAM
  Symposium on Discrete Algorithms, SIAM, 2018, pp.~2703--2717.

\bibitem{mescheder2017numerics}
{\sc L.~Mescheder, S.~Nowozin, and A.~Geiger}, {\em The numerics of {GAN}s}, in
  Advances in Neural Information Processing Systems, vol.~30, 2017.

\bibitem{mishra2021manifold}
{\sc B.~Mishra, N.~Satyadev, H.~Kasai, and P.~Jawanpuria}, {\em Manifold
  optimization for non-linear optimal transport problems}, arXiv:2103.00902,
  (2021).

\bibitem{mokhtari2020unified}
{\sc A.~Mokhtari, A.~Ozdaglar, and S.~Pattathil}, {\em A unified analysis of
  extra-gradient and optimistic gradient methods for saddle point problems:
  Proximal point approach}, in International Conference on Artificial
  Intelligence and Statistics, PMLR, 2020, pp.~1497--1507.

\bibitem{mokhtari2020convergence}
{\sc A.~Mokhtari, A.~E. Ozdaglar, and S.~Pattathil}, {\em Convergence rate of
  {O}(1/k) for optimistic gradient and extragradient methods in smooth
  convex-concave saddle point problems}, SIAM J. Optim., 30 (2020),
  pp.~3230--3251.

\bibitem{monteiro2010complexity}
{\sc R.~D. Monteiro and B.~F. Svaiter}, {\em On the complexity of the hybrid
  proximal extragradient method for the iterates and the ergodic mean}, SIAM J.
  Optim., 20 (2010), pp.~2755--2787.

\bibitem{monteiro2011complexity}
{\sc R.~D. Monteiro and B.~F. Svaiter}, {\em Complexity of variants of tseng's
  modified fb splitting and korpelevich's methods for hemivariational
  inequalities with applications to saddle-point and convex optimization
  problems}, SIAM J. Optim., 21 (2011), pp.~1688--1720.

\bibitem{moosavi2017universal}
{\sc S.-M. Moosavi-Dezfooli, A.~Fawzi, O.~Fawzi, and P.~Frossard}, {\em
  Universal adversarial perturbations}, in {Proceedings of the Conference on
  Computer Vision and Pattern Recognition}, 2017, pp.~1765--1773.

\bibitem{muller2019orthogonal}
{\sc J.~M{\"u}ller, R.~Klein, and M.~Weinmann}, {\em {Orthogonal Wasserstein
  GANs}}, arXiv:1911.13060,  (2019).

\bibitem{nemirovski2004prox}
{\sc A.~Nemirovski}, {\em Prox-method with rate of convergence {O}(1/t) for
  variational inequalities with {L}ipschitz continuous monotone operators and
  smooth convex-concave saddle point problems}, SIAM J. Optim., 15 (2004),
  pp.~229--251.

\bibitem{neumann1928theorie}
{\sc J.~v. Neumann}, {\em Zur theorie der gesellschaftsspiele}, Math. Ann., 100
  (1928), pp.~295--320.

\bibitem{jawanpuria18b}
{\sc M.~Nimishakavi, P.~Jawanpuria, and B.~Mishra}, {\em A dual framework for
  low-rank tensor completion}, in Advances in Neural Information Processing
  Systems, 2018.

\bibitem{nocedal1999numerical}
{\sc J.~Nocedal and S.~J. Wright}, {\em Numerical optimization}, Springer,
  1999.

\bibitem{paty19}
{\sc F.-P. Paty and M.~Cuturi}, {\em Subspace robust wasserstein distances}, in
  International Conference on Machine Learning, 2019.

\bibitem{paty2019subspace}
{\sc F.-P. Paty and M.~Cuturi}, {\em Subspace robust wasserstein distances}, in
  {International Conference on Machine Learning}, PMLR, 2019, pp.~5072--5081.

\bibitem{pennec2020manifold}
{\sc X.~Pennec}, {\em Manifold-valued image processing with {SPD} matrices}, in
  {Riemannian Geometric Statistics in Medical Image Analysis}, Elsevier, 2020,
  pp.~75--134.

\bibitem{peyre2019computational}
{\sc G.~Peyr{\'e}, M.~Cuturi, et~al.}, {\em Computational optimal transport:
  With applications to data science}, Foundations and Trends{\textregistered}
  in Machine Learning, 11 (2019), pp.~355--607.

\bibitem{polyak1963gradient}
{\sc B.~T. Polyak}, {\em Gradient methods for minimizing functionals}, Zhurnal
  Vychislitel'noi Matematiki i Matematicheskoi Fiziki, 3 (1963), pp.~643--653.

\bibitem{popov1980modification}
{\sc L.~D. Popov}, {\em A modification of the {Arrow-Hurwicz} method for search
  of saddle points}, Mathematical Notes of the Academy of Sciences of the USSR,
  28 (1980), pp.~845--848.

\bibitem{rakotomamonjy08a}
{\sc A.~Rakotomamonjy, F.~Bach, S.~Canu, and Y.~Grandvalet}, {\em Simplemkl},
  J. Mach. Learn. Res., 9 (2008), pp.~2491--2521.

\bibitem{ring2012optimization}
{\sc W.~Ring and B.~Wirth}, {\em Optimization methods on {R}iemannian manifolds
  and their application to shape space}, SIAM J. Optim., 22 (2012),
  pp.~596--627.

\bibitem{rockafellar1976monotone}
{\sc R.~T. Rockafellar}, {\em Monotone operators and the proximal point
  algorithm}, SIAM J. Control Optim., 14 (1976), pp.~877--898.

\bibitem{royden1988real}
{\sc H.~L. Royden and P.~Fitzpatrick}, {\em Real analysis}, vol.~32, Macmillan
  New York, 1988.

\bibitem{sato2016dai}
{\sc H.~Sato}, {\em {A Dai--Yuan-type Riemannian conjugate gradient method with
  the weak Wolfe conditions}}, Computational Optimization and Applications, 64
  (2016), pp.~101--118.

\bibitem{sato2021RCGnew}
{\sc H.~Sato}, {\em Riemannian conjugate gradient methods: General framework
  and specific algorithms with convergence analyses}, arXiv:2112.02572,
  (2021).

\bibitem{sato2021riemannian}
{\sc H.~Sato}, {\em Riemannian Optimization and Its Applications}, Springer,
  2021.

\bibitem{sato2019riemannian}
{\sc H.~Sato, H.~Kasai, and B.~Mishra}, {\em Riemannian stochastic variance
  reduced gradient algorithm with retraction and vector transport}, SIAM J.
  Optim., 29 (2019), pp.~1444--1472.

\bibitem{schafer2019competitive}
{\sc F.~Sch{\"a}fer and A.~Anandkumar}, {\em Competitive gradient descent},
  Advances in Neural Information Processing Systems, 32 (2019).

\bibitem{shi2021coupling}
{\sc D.~Shi, J.~Gao, X.~Hong, S.~Boris~Choy, and Z.~Wang}, {\em Coupling matrix
  manifolds assisted optimization for optimal transport problems}, Mach.
  Learn., 110 (2021), pp.~533--558.

\bibitem{sinkhorn1964relationship}
{\sc R.~Sinkhorn}, {\em A relationship between arbitrary positive matrices and
  doubly stochastic matrices}, The Annals of Mathematical Statistics, 35
  (1964), pp.~876--879.

\bibitem{sion1958general}
{\sc M.~Sion}, {\em On general minimax theorems.}, Pacific J. Math., 8 (1958),
  pp.~171--176.

\bibitem{smirnov2021tensorflow}
{\sc O.~Smirnov}, {\em {TensorFlow RiemOpt}: a library for optimization on
  {Riemannian} manifolds}, arXiv:2105.13921,  (2021).

\bibitem{sra2015conic}
{\sc S.~Sra and R.~Hosseini}, {\em Conic geometric optimization on the manifold
  of positive definite matrices}, SIAM J. Optim., 25 (2015), pp.~713--739.

\bibitem{stacho1980minimax}
{\sc L.~Stach{\'o}}, {\em Minimax theorems beyond topological vector spaces},
  Acta Sci. Math.(Szeged), 42 (1980), pp.~157--164.

\bibitem{townsend2016pymanopt}
{\sc J.~Townsend, N.~Koep, and S.~Weichwald}, {\em Pymanopt: A python toolbox
  for optimization on manifolds using automatic differentiation}, J. Mach.
  Learn. Res., 17 (2016), pp.~1--5,
  \url{http://jmlr.org/papers/v17/16-177.html}.

\bibitem{tseng1995linear}
{\sc P.~Tseng}, {\em On linear convergence of iterative methods for the
  variational inequality problem}, J. Comput. Appl. Math., 60 (1995),
  pp.~237--252.

\bibitem{udriste2013convex}
{\sc C.~Udriste}, {\em Convex functions and optimization methods on Riemannian
  manifolds}, vol.~297, Springer Science \& Business Media, 2013.

\bibitem{vishnoi2018geodesic}
{\sc N.~K. Vishnoi}, {\em Geodesic convex optimization: Differentiation on
  manifolds, geodesics, and convexity}, arXiv:1806.06373,  (2018).

\bibitem{wang2020orthogonal}
{\sc J.~Wang, Y.~Chen, R.~Chakraborty, and S.~X. Yu}, {\em Orthogonal
  convolutional neural networks}, in Proceedings of the Conference on Computer
  Vision and Pattern Recognition, 2020, pp.~11505--11515.

\bibitem{zhang2016riemannian}
{\sc H.~Zhang, S.~J~Reddi, and S.~Sra}, {\em Riemannian {SVRG}: Fast stochastic
  optimization on {R}iemannian manifolds}, in Advances in Neural Information
  Processing Systems, vol.~29, 2016.

\bibitem{zhang2022minimax}
{\sc P.~Zhang, J.~Zhang, and S.~Sra}, {\em Minimax in geodesic metric spaces:
  {S}ion's theorem and algorithms}, arXiv:2202.06950,  (2022).

\bibitem{zhou2019faster}
{\sc P.~Zhou, X.-T. Yuan, and J.~Feng}, {\em Faster first-order methods for
  stochastic non-convex optimization on {R}iemannian manifolds}, in
  {International Conference on Artificial Intelligence and Statistics}, PMLR,
  2019, pp.~138--147.

\end{thebibliography}
